\documentclass[11pt]{amsart}
\usepackage{subfiles}
\usepackage{euscript}
\usepackage{tabu}
\usepackage[margin=1in]{geometry} 		
\usepackage{graphicx}
\usepackage{faktor}
\usepackage{xcolor}
\usepackage{amsmath,amsthm,amssymb}
\usepackage{mathrsfs}
\usepackage{tikz}
\usepackage{tikz-cd}
\usepackage{accents}
\usepackage{upgreek}
\usepackage{enumitem}
\usepackage{bm}
\usepackage{mathtools}
\usepackage[all]{xy}
\usepackage{caption}

\usepackage{url}
\usepackage{float}
\usepackage{todonotes}
\usepackage{bbm}
\usepackage{colonequals}
\usepackage{longtable}

\usepackage[full]{textcomp}
\usepackage[cal=cm]{mathalfa}
\usepackage{xparse}
\usepackage{comment}
\usepackage{cite}
\usepackage{stmaryrd}

\allowdisplaybreaks

\tikzset{
  commutative diagrams/.cd, 
  arrow style=tikz, 
  diagrams={>=stealth}
}

\theoremstyle{definition}

\makeatletter
\def\@tocline#1#2#3#4#5#6#7{\relax
  \ifnum #1>\c@tocdepth 
  \else
    \par \addpenalty\@secpenalty\addvspace{#2}%
    \begingroup \hyphenpenalty\@M
    \@ifempty{#4}{%
      \@tempdima\csname r@tocindent\number#1\endcsname\relax
    }{%
      \@tempdima#4\relax
    }%
    \parindent\z@ \leftskip#3\relax \advance\leftskip\@tempdima\relax
    \rightskip\@pnumwidth plus4em \parfillskip-\@pnumwidth
    #5\leavevmode\hskip-\@tempdima
      \ifcase #1
       \or\or \hskip 1em \or \hskip 2em \else \hskip 3em \fi%
      #6\nobreak\relax
    \dotfill\hbox to\@pnumwidth{\@tocpagenum{#7}}\par
    \nobreak
    \endgroup
  \fi}
\makeatother

\makeatletter
\DeclareRobustCommand{\cev}[1]{%
  \mathpalette\do@cev{#1}%
}
\newcommand{\do@cev}[2]{%
  \fix@cev{#1}{+}%
  \reflectbox{$\m@th#1\vec{\reflectbox{$\fix@cev{#1}{-}\m@th#1#2\fix@cev{#1}{+}$}}$}%
  \fix@cev{#1}{-}%
}
\newcommand{\fix@cev}[2]{%
  \ifx#1\displaystyle
    \mkern#23mu
  \else
    \ifx#1\textstyle
      \mkern#23mu
    \else
      \ifx#1\scriptstyle
        \mkern#22mu
      \else
        \mkern#22mu
      \fi
    \fi
  \fi
}

\makeatother

\usetikzlibrary{calc}
\usetikzlibrary{fadings}
\usetikzlibrary{decorations.pathmorphing}
\usetikzlibrary{decorations.pathreplacing}
\usetikzlibrary{shapes}

\newcounter{marginnote}
\DeclareMathAlphabet{\mathpzc}{OT1}{pzc}{m}{it}

\usepackage[pagebackref]{hyperref}
\hypersetup{
  colorlinks   = true,          
  urlcolor     = blue,          
  linkcolor    = purple,          
  citecolor   = blue             
}
\usepackage{cleveref}
\theoremstyle{theorem}
\newtheorem{theorem}{Theorem}[section]
\newtheorem*{theorem*}{Theorem}

\newtheorem{corollary}[theorem]{Corollary}
\newtheorem{lemma}[theorem]{Lemma}
\newtheorem{proposition}[theorem]{Proposition}
\theoremstyle{definition}
\newtheorem{remark}[theorem]{Remark}
\newtheorem*{remark*}{Remark}

\newtheorem*{runningexample*}{Running example}

\newtheorem*{aside*}{Aside}

\newtheorem{construction}[theorem]{Construction}

\newtheorem{definition}[theorem]{Definition}
\newtheorem{example}[theorem]{Example}

\newtheorem{notation}[theorem]{Notation}
\newtheorem{proposition-definition}[theorem]{Proposition-Definition}

\newcommand{\NN}{\mathbf{N}}

\newcommand{\Zcal}{\mathcal{Z}}

\newcommand{\Gm}{\mathbb{G}_{\rm{m}}}

\newcommand{\ol}[1]{\overline{#1}}

\newcommand{\Min}[1]{\widetilde\Mcal_{1,#1}}
\newcommand{\disc}[1]{{#1}^{\rm{disc}}}

\newcommand{\bcd}{\begin{center}\begin{tikzcd}}
\newcommand{\ecd}{\end{tikzcd}\end{center}}

\newcommand{\GG}{\mathbf{G}}

\newcommand{\Aaff}{\mathbb{A}}

\newcommand{\PP}{\mathbb{P}}
\newcommand{\OO}{\mathcal{O}}
\renewcommand{\NN}{\mathbb{N}}

\newcommand{\ZZ}{\mathbb{Z}}

\newcommand{\Tail}[1]{\mathbf{T}_{#1}}
\newcommand{\oTail}[1]{\mathbf{T}_{#1}}
\newcommand{\cTail}[1]{\overline{\mathbf{T}}_{#1}}

\newcommand{\Ell}[1]{\mathbf{Ell}_{#1}}
\newcommand{\oEll}[1]{\mathbf{Ell}_{#1}}
\newcommand{\cEll}[1]{\overline{\mathbf{Ell}}_{#1}}

\newcommand{\Ocal}{\mathcal{O}}
\newcommand{\Mcal}{\mathcal{M}}
\newcommand{\Hcal}{\mathcal{H}}

\newcommand{\Dcal}{\mathcal{D}}

\newcommand{\Gcal}{\mathcal{G}}

\newcommand{\Ucal}{\mathcal{U}}

\newcommand{\Pcal}{\mathcal{P}}

\newcommand{\Qcal}{\mathcal{Q}}
\newcommand{\Scal}{\mathcal{S}}

\newcommand{\Mbar}{\ol{\Mcal}}

\newcommand{\logcan}{\omega^{\rm{log}}}

\newcommand{\calG}{\mathcal{G}}
\newcommand{\calO}{\mathcal{O}}

\newcommand{\Spec}{\operatorname{Spec}}

\NewDocumentCommand{\compatibilitydatum}{m m m m m m O{} O{} O{}}{
\begin{equation*} \begin{tikzcd}[ampersand replacement=\&]
  \: \arrow{r} \& {#1} \arrow{r} \arrow{d}{#7} \& {#2} \arrow{r} \arrow{d}{#8} \& {#3} \arrow{r}{[1]} \arrow{d}{#9} \& \: \\
  \: \arrow{r} \& {#4} \arrow{r} \& {#5} \arrow{r} \& {#6} \arrow{r} \& \:
\end{tikzcd} \end{equation*}}

\NewDocumentCommand{\commutingsquare}{m m m m o O{} O{} O{} O{}}{
\begin{equation}\begin{tikzcd}[ampersand replacement=\&] \label{#5}
  #1 \arrow{r}{#6} \arrow{d}{#7} \& #2 \arrow{d}{#8} \\
  #3 \arrow{r}{#9} \& #4
\end{tikzcd}\IfValueTF{#5}{\label{#5}}{} \end{equation}}

\NewDocumentCommand{\cartesiansquare}{m m m m O{} O{} O{} O{}}{
\begin{equation*}\begin{tikzcd}[ampersand replacement=\&]
  #1 \arrow{r}{#5} \arrow{d}{#6} \arrow[dr, phantom, "\square"] \& #2 \arrow{d}{#7} \\
  #3 \arrow{r}{#8} \& #4
\end{tikzcd} \end{equation*}}

\NewDocumentCommand{\cartesiansquarelabel}{m m m m m O{} O{} O{} O{}}{
\begin{tikzcd}[ampersand replacement=\&]
  #1 \arrow{r}{#6} \arrow{d}{#7} \arrow[dr, phantom, "\square"] \& #2 \arrow{d}{#8} \\
  #3 \arrow{r}{#9} \& #4
\end{tikzcd}\IfValueTF{#5}{\label{#5}}{}
}

\NewDocumentCommand{\Nodangleofspaces}{m m m O{} O{} O{}}{
\begin{tikzcd} [ampersand replacement=\&]
#1 \arrow{r}{#4} \arrow[bend right]{rr}{#5} \& #2 \arrow{r}{#6} \& #3
\end{tikzcd}}

\newcommand{\on}{\operatorname}
\newcommand{\oM}{\overline{\mathcal M}}

\newcommand{\Nod}{\mathbf{Nod}}
\newcommand{\cNod}{\overline{\Nod}}

\begin{document}
 
\title[Chow rings of modular compactifications of $\Mcal_{1,n\leq 6}$]{Integral Chow rings of modular compactifications of $\Mcal_{1,n\leq 6}$}
\author[L.~Battistella]{Luca Battistella}
\address[Luca Battistella]{Dipartimento di Matematica, Universit\`{a} di Bologna, Italia}
	\email{luca.battistella2@unibo.it}
\author[A.~Di Lorenzo]{Andrea Di Lorenzo}
	\address[Andrea Di Lorenzo]{Dipartimento di matematica, Universit\`{a} di Pisa, Italia}
	\email{andrea.dilorenzo@unipi.it}

\begin{abstract}
For $n\leq 6$, we compute the integral Chow ring of every modular compactification of $\Mcal_{1,n}$ parametrising only Gorenstein curves with smooth, distinct markings. These include the Deligne--Mumford, Schubert, and Smyth compactifications, and many more. They can all be excised from the stack of log-canonically polarised Gorenstein curves. The Chow ring of the latter admits a simple, combinatorial description, which we compute by patching along a natural stratification by \emph{core level}. We deduce that all these modular compactifications satisfy the Chow-K\"{u}nneth generation property, that the cycle class map is an isomorphism, and for $n=4$ we study whether Getzler's relation holds integrally and for other compactifications.

\end{abstract}

\maketitle
\setcounter{tocdepth}{1}
\tableofcontents

\section{Introduction}
The study of rational Chow rings of moduli spaces of stable curves was initiated by Mumford \cite{Mum}, and continues to this day \cite{Fab, Fab2,  Iza, PenevVakil, CL789}. A complete and explicit computation of the Chow ring is typically only possible as long as the geometry of moduli is not too complicated, e.g. when Chow equals cohomology, and in particular the moduli space is rationally connected (so, only for low values of $g$ and $n$). \emph{Integral} Chow rings are harder to compute, but in general they encode substantially more information. 
Keel computed the integral Chow ring of $\oM_{0,n}$ for every $n\geq3$ \cite{Keel}. In recent years, the study of integral Chow rings of moduli of stable curves has picked up pace \cite{Lar, dilorenzo2021polarizedtwistedconicsmoduli, DLPV, Inchiostro, Per, Bishop}. These computations are usually based on a stratification of the moduli space into pieces that admit a simple, finite quotient presentation, and then \emph{patching} or higher Chow groups.

Based on the properties of alternative compactifications of $\Mcal_{1,n}$ \cite{SmythI,LekiliP,BKN}, in this paper we recognise that, at least for $g=1$ and $n\leq 6$, there is an enlargement of the moduli space of stable curves whose Chow ring admits a very simple, mostly combinatorial description: it is the moduli stack $\Gcal_{1,n}$ of Gorenstein curves polarised by the log canonical bundle. The Chow ring of \emph{any} modular compactification of $\Mcal_{1,n}$ can be computed from this one by excision: these include the Deligne--Mumford space of stable curves, Schubert's space of pseudostable curves, Smyth's spaces of $m$-stable curves, and many more introduced by Bozlee, Kuo and Neff. These are denoted by 
$\oM_{1,n}(Q)$, and they depend on a parameter $Q$ that is a collection of partitions of $[n]:=\{1,2,\ldots,n\}$; for instance, for $n=5$ there are 79,814,831 (!) of these compactifications. In the statement of our main theorem and throughout, we assume the indices $i$, $j$, $h$ and $k$ to be distinct. We also refer the reader to \Cref{not:incomparable} and \Cref{not:disc} for the meaning of $\not\sim$ and ${\disc{-}}$ in the case of subsets or partions of $[n]=\{1,2,\ldots,n\}$.

\begin{theorem*}\label{thm:chow mbar 1}
    For $n\leq 5$, the integral Chow ring of $\Mbar_{1,n}(Q)$ is generated by $\lambda$, the first Chern class of the Hodge line bundle, and by the boundary divisors $\tau_B$, $B\subset [n]$ of cardinality $\geq 2$, parametrising curves with a rational tail marked by $B$. The ideal of relations is generated by:
    \begin{align*}
        K_1(B;i,j,h)=&\tau_B ( \sum_{\substack{i,j\in B' \\ h\notin B'}} \tau_{B'} - \sum_{\substack{i,h\in B'' \\ j\notin B''}} \tau_{B''}) \text{ for }i,j,h \in B;\\
        K_1(B;i,j,h,k)=&\tau_B( \sum_{\substack{i,j\in B' \\ h,k \notin B'}} \tau_{B'} + \sum_{\substack{h,k\in B'' \\ i,j \notin B''}} \tau_{B''}  - \sum_{\substack{i,h\in B''' \\ j,k\notin B'''}} \tau_{B'''}  - \sum_{\substack{j,k\in B'''' \\ i,h\notin B''''}} \tau_{B''''} ), \text{ for }i,j,h,k \in B;\\
        K_2(B_1,\ldots,B_k)=&\tau_{B_1} \cdot \tau_{B_2}\cdot\cdots\cdot\tau_{B_k}, \text{ if there are }i,j\text{ such that }B_i\not\sim B_j\text{ or }\disc{\{B_1,\ldots,B_k\}}\notin Q;\\
        N(B)=&\tau_B(\lambda + \sum_{i,j \in B'} \tau_{B'}), \text{ for any choice of }i,j\in B; \\
        [\cEll{S}]&\text{, for every }S\in Q\text{ (all explicit expressions can be found in \Cref{app:poly}).}
    \end{align*}
    
    For $n=6$, we need an extra generator $\nu$ in codimension $2$, the fundamental class of a locus of curves with two non-separating nodes, and extra relations
    given in \Cref{def:rels 2} and \Cref{def:rels 3}.  

    Furthermore, $\oM_{1,n\leq 6}(Q)$ satisfies the Chow--K\"unneth generation property, the cycle class map $A^*(\oM_{1,n\leq 6}(Q))_{{\mathbb Q}_\ell} \to H^{2*}_{\text{\'{e}t}}(\oM_{1,n\leq 6}(Q),\mathbb Q_\ell)$ is an isomorphism, and the space has polynomial point count.

    Finally, for $n=4$, Getzler's relation \cite{Get, Pand} only holds rationally and for $\Mbar_{1,4}$, while, with a correction term of $12\lambda^2$, it holds integrally for every $Q$.
\end{theorem*}

\subsection{Strategy of proof}
First, a few more words on the alternative compactifications $\oM_{1,n}(Q)$ of \cite{BKN}. If $Q$ is the empty set, then $\oM_{1,n}(Q)=\oM_{1,n}$; on the other hand, if $Q$ is the whole power set of $[n]$ minus the partition $S_{\max}=\{\{1\},\ldots,\{n\}\}$, then $\oM_{1,n}(Q)$ is the smallest possible compactification of $\Mcal_{1,n}$, which can be identified with Smyth's $\oM_{1,n}(n-1)$. Using $A_{\infty}$-structures, Lekili and Polishchuk \cite{LekiliP} gave a very explicit description of this space, which can be exploited in order to compute its Chow ring. Alas, this description becomes more and more involved as $n$ grows; moreover, $\oM_{1,n}(n-1)$ is singular for $n\geq 7$, hence we do not have access to its Chow ring - for these reasons our current methods apply only up to $n=6$.

We compute the integral Chow ring of every modular compactification of $\Mcal_{1,n\leq6}$ by bootstrapping from $\oM_{1,n}(n-1)$, by the patching technique introduced by Vistoli and the second-named author \cite{dilorenzo2021polarizedtwistedconicsmoduli}: a stack $X$ is written as the union of a closed substack $Z$ and its open complement $U$, and the Chow ring of $X$ is reconstructed from those of $Z$ and $U$, provided the top Chern class of the normal bundle of $Z$ in $X$ is not a zero-divisor; this forces $Z$ (and a fortiori $X$) to be a bona fide Artin stack.

In practice, we achieve this by studying the stack $\Gcal_{1,n}$ of log-canonically polarised Gorenstein curves. This admits two stratifications: the first one, by tail type, is essentially combinatorial, determined by the markings allowed to move onto a rational tail; we denote the strata by $\Tail{S}$, where $S$ is a set partition of $[n]$. The second one, by singularity type, is of a more geometric nature; we denote the strata by $\Ell{S}$. We harness the first stratification to compute $A^*(\Gcal_{1,n})$ by patching: the Chow rings of the strata can be computed inductively on $g$ and $n$. Relations can be lifted easily; the classes $[\cEll{S}]$ can be computed in the same way, although they give rise to rather complicated expressions (see \Cref{app:poly}). Indeed, patching gives a way of computing the fundamental class of any locus $\Zcal\subset \Gcal_{1,n}$, once the value of the restriction of $\Zcal$ to each strata is known.
A simple application of excision yields the Chow ring of $\oM_{1,n}(Q)$, which is an open substack of $\Gcal_{1,n}$ obtained by removing some tail strata and some complementary singularity strata.

The same stratification can be used to prove that the Chow ring of $\Gcal_{1,n}$ is a free $\ZZ$-module (for $n\leq 5$, even a free $\ZZ[\lambda]$-module), and that $\Gcal_{1,n}$ satisfies the Chow-K\"{u}nneth generation property, which in turn implies that the stacks $\oM_{1,n}(Q)$ do as well; from this one deduces that the cycle class map is an isomorphism and that these stacks have polynomial point count, as stated in the main theorem.

The same method allows us also to prove that a modification of the Getzler's relation (obtained by adding the term $12\lambda^2$ to the original one) holds integrally in $\Gcal_{1,4}$, from which we deduce that the original Getzler's relation does not hold integrally on $\Mbar_{1,n}(Q)$ for any $Q$, and that it holds rationally only on $\Mbar_{1,4}$.

\subsection{Relation to previous work}
The rational Chow ring of $\oM_{1,n}$ is known by work of Belorousski \cite{Belorousski}, Getzler \cite{Get} and Petersen \cite{Petersen}. The integral Chow ring of $\oM_{1,n}$ is known: for $n=1$ from the very beginning of the field \cite{EdidinGrahamIntersection}; for $n=2$ from the work of Pernice, Vistoli and the second-named author \cite{DLPV}, and from Inchiostro's work \cite{Inchiostro}; for $n=3$ from Bishop's work \cite{Bishop}; for $n=4$ from our previous paper \cite{BDL1}. In fact, in the latter work, we computed the integral Chow ring of all Smyth's compactifications $\oM_{1,n}(m)$ for $n=3,4$, and $0\leq m\leq n-1$, by realising that they are all related by a zig-zag of weighted blow-ups. The presentations obtained in this paper, compared with the ones obtained in our previous work, have in our opinion the advantage of being more tied to the geometry of the moduli stacks: the generating relations are either of combinatorial type (the ones that hold on $\mathcal{G}_{1,n}$) or of geometric type (the ones corresponding to fundamental classes of excised loci).

\subsection{Organization of the paper}
In \Cref{sec:gor} we recall the classification and main properties of Gorenstein curve singularities of genus one, including various notions of \emph{level} relevant for compactifying and stratifying the moduli stack, and we introduce the stack $\Gcal_{1,n}$ of log-canonically polarised Gorenstein curves.

In \Cref{sec:classes} we focus on $\Min{n}$, the stack of \emph{minimal} Gorenstein curves of genus one, which is the smallest open in the stratification of $\Gcal_{1,n}$ by core level, and appears recursively as the genus one factor in every further stratum. Minimal curves admit a canonical form, computed in complete generality by Lekili and Polishchuk, yielding an explicit description of $\Min{n}$. When $n$ is small, this space and its Chow are simple to describe. We compute the fundamental classes of several loci of elliptic singularities in $\Min{n}$, which is the basic ingredient in order to set up the inductive computation of the Chow ring of $\Gcal_{1,n}$.

In \Cref{sec:chow G} we apply the patching technique in order to obtain a presentation of the integral Chow ring of $\Gcal_{1,n\leq6}^{sm}$ (\Cref{thm:chow G}), the smooth locus of the stack of log-canonically polarised Gorenstein curves; for $n\leq 5$, this is the whole stack, while for $n=6$ we need to carve out one (stacky) point, representing the most singular curve.

Finally, in \Cref{sec:chow M} we compute the integral Chow ring of any modular compactification $\oM_{1,n}(Q)$ of $\Mcal_{1,n}$ for $n\leq 6$ (\Cref{thm:chow mbar 1}), we show that the cycle class map is an isomorphism for all of these spaces (\Cref{prop:cycle}) and we study the integral Getzler's relation for $n=4$.

In \Cref{app:poly} we display explicit expressions for the fundamental classes of loci of singular curves for $n\leq 5$ (we refrained from including the expressions for $n=6$ because they are fairly long and complicated; we wonder whether they might become easier in a different basis).

\subsection{Conventions} Throughout the paper, when we use indices $i$, $j$, $h$ and $k$, we always assume them to be distinct. We work over a field $k$ of characteristic different from $2$ and $3$.

\subsection{Acknowledgments} We thank Sam Molcho, Rahul Pandharipande, Michele Pernice, Tommaso Rossi and Angelo Vistoli for helpful conversations.  We are very grateful to the anonymous referee for their careful comments.

\section{The stack of log-canonically polarised Gorenstein curves}\label{sec:gor}
\subsection{Gorenstein curves of genus one}
We will only be interested in reduced curves. A reduced curve is always Cohen--Macaulay; it is Gorenstein if the dualising sheaf is a line bundle. This is a local property. The simplest example of a Gorenstein singularity is the node (or any plane curve singularity).

Given a projective curve $C$, its singularities contribute to the arithmetic genus of $C$ by the following formula: $g=\delta-m+1$, where $m$ is the number of branches at the singular point, and $\delta$ measures the difference between functions on $C$ and functions on the normalisation (another way of saying it is that $g$ measures the difference between functions on $C$ and functions on the seminormalisation, which is an ordinary $m$-fold point). In particular, rational singularities are precisely ordinary $m$-fold points; of these, only the node is Gorenstein.

Smyth \cite[Appendix A]{SmythI} classified all Gorenstein singularities of genus one.
\begin{definition}
   Let $k$ be a field of characteristic different from $2,3$. A $k$-point of a curve $C$ is called an \emph{elliptic} $m$\emph{-fold point} if the analytic germ of $C$ at $p$ is one of the following:
    \[\hat{\calO}_{C,p}\simeq 
    \left\{ \begin{matrix}
       &k \llbracket x,y \rrbracket/(y^2-x^3) & m=1 &\text{ordinary cusp, } A_2 \\
       &k \llbracket x,y \rrbracket/(y^2-yx^2) & m=2 &\text{ordinary tacnode, }A_3 \\
       &k \llbracket x,y \rrbracket/(x^2y-yx^2) & m=3 &\text{planar triple point, }D_4 \\
       &k \llbracket x_1,\ldots,x_{m-1} \rrbracket/I_m & m\geq 4 &\text{$m$-general lines through the origin of }\Aaff^{m-1}
    \end{matrix}\right.
    \]
    where $I_m$ is the ideal generated by the binomials $x_ix_j-x_ix_h$ for all $i,j,h\in[m-1]$.
\end{definition}

So, a Gorenstein curve of genus one may only have nodes and at most one elliptic $m$-fold point for singularities. It can always be decomposed  into a \emph{minimal elliptic subcurve} (the \emph{core}) and a union of rational tails (trees), nodally attached to it \cite[\S3.1]{SmythI}. The core may be a smooth elliptic curve, a circle of nodally attached $\PP^1$s, or an elliptic $m$-fold point (whose normalisation consists of exactly $m$ copies of $\PP^1$). In any case, the dualising bundle of the core is trivial \cite[\S2.2]{SmythI}. 

\subsection{Levels and strata}
We introduce our main character.
\begin{definition}
 Let $\calG_{1,n}$ denote the moduli stack whose objects are families of Gorenstein curves $C\to S$ of arithmetic genus one, marked with $n$ smooth and distinct points $p_1,\ldots,p_n\colon S\to C$, and such that the log canonical line bundle $\logcan_{C/S}=\omega_{C/S}(p_1+\ldots+p_n)$ is relatively ample.
\end{definition}
The request that $\logcan_C$ be ample implies that every branch of an elliptic $m$-fold point contains at least one special point (marking or node), and every other rational component contains at least three special points (this coincides with the usual Deligne--Mumford stability). It follows that $\calG_{1,n}$ parametrises curves with at worst elliptic $n$-fold points. An elliptic $m$-fold point such that every branch contains exactly one special point has automorphism group $\Gm$ \cite[\S 2.1]{SmythI}, and these are the only points with infinite stabiliser.
\begin{lemma}
 The stack $\calG_{1,n}$ is a quasi-separated algebraic stack of finite type and with affine diagonal over $\Spec(\ZZ[\frac{1}{6}])$. It is irreducible of dimension $n$, and smooth in codimension $6$. In particular, it is a smooth stack for $n\leq 5$, and it has a single singular point for $n=6$.
\end{lemma}
\begin{proof}
 The stack of all curves $\Ucal_{1,n}$ is algebraic, locally of finite type and quasi-separated over $\Spec(\ZZ)$, see \cite[Appendix B]{SmythTowards} by Hall or \cite[\href{https://stacks.math.columbia.edu/tag/0DSS}{Lemma 0DSS}]{stacks-project}. The conditions that $\logcan$ be (i) a line bundle, and (ii) ample are both open, hence $\Gcal_{1,n}\subseteq \Ucal_{1,n}$ is an open algebraic substack and in particular locally of finite type and quasi-separated. 
 
 Observe that for any object $\pi\colon (C,p_1,\ldots,p_n)\to S$ of $\Gcal_{1,n}$ the relative ampleness of $\omega_{\pi}(p_1+\ldots +p_n)$ implies that the number of irreducible components of the geometric fiber is universally bounded by a value $d$; this implies that the stack $\Gcal_{1,n}$ is actually contained in the stack $\mathcal{U}_{1,n,d}$ of curves with a universally bounded number of components, which is quasi-compact \cite[Corollary B.3]{SmythTowards}.
 
  Furthermore, the sheaf $\omega_{\pi}(p_1+\ldots +p_n)^{\otimes 3}$ is relatively very ample, thus \'{e}tale locally on $S$ we have an embedding $C\subset \mathbb{P}^N_S$; it follows that $\operatorname{Aut}_S((C,\sum_i p_i))$ is a closed subgroup scheme of $\operatorname{PGL}_{N,S}\to S$, and hence affine. We deduce that the diagonal of $\Gcal_{1,n}$ is affine.
  
 Every isolated curve singularity of genus one is smoothable. The last statement follows from the deformation theory of elliptic $m$-fold points \cite[\S4.3]{SmythII}.
\end{proof}

It follows from the classical deformation theory of nodes that the locus where a (separating) node persists forms a divisor in $\Gcal_{1,n}$. There is thus an open substack of $\Gcal_{1,n}$ consisting of curves without separating nodes, i.e. curves that coincide with their core; we call such a curve \emph{minimal}. Observe that this condition poses no further restriction on the type of singularities involved.
\begin{definition}
 We denote by $\Min{n}\subseteq\Gcal_{1,n}$ the open substack of minimal curves.
\end{definition}

Smyth used Gorenstein singularities as a replacement for genus one subcurves with fewer special point (markings and separating nodes; he called this number the \emph{level} of the genus one subcurve): in a smoothing one-parameter family of nodal curves, an elliptic $m$-bridge (subcurve of level $m$) can be contracted to an elliptic $m$-fold point. The suggested variation of stability condition preserves properness while reducing the boundary complexity of the moduli space. Several more compactifications have been introduced by Bozlee, Kuo, and Neff, by refining Smyth's notion of level from a number to a partition of the set of markings \cite[Definition 1.4]{BKN}.

\begin{notation}
 Let $S=\{S_1,\ldots,S_k, S_{k+1},\ldots,S_{s_0}\}$ be a partition of $[n]$, with $|S_i|=1$ if and only if $i>k$. We denote by $\ell(S):=s_0$ the number of parts of $S$. We make $\on{Part}([n])$ into a poset by declaring $S_1\preceq S_2$ if and only if $S_2$ is a refinement of $S_1$. 
\end{notation}
\begin{remark}
               $\on{Part}([n])$ is a complete lattice with minimum the coarsest partition $S_{\min}=\{[n]\}$ and maximum the discrete partition $S_{\max}=\{\{1\},\ldots,\{n\}\}$.
\end{remark}
 
\begin{definition}
    Let $C$ be a log-canonically polarised Gorenstein curve with $n$ markings.  We say that $C$ has \emph{core level} $S$ if the core $E$ of $C$ is marked with $S_{k+1}\cup\ldots\cup S_{s_0}$, and the complement $C\setminus E$ consists of $k$ rational trees, $R_i$ being marked with $S_i$, for $i=1,\ldots,k$. We call the length $s_0$ of the partition $S$ the \emph{numerical core level} of $C$ (this was Smyth's original notion).
\end{definition}
\begin{example}\label{exa1}
    Let $C$ be a smooth genus one curve with two rational trees; suppose that the core is marked with $p_6$ and the two rational trees are marked with $\{p_1,p_2,p_3\}$ and $\{p_4,p_5\}$ respectively, the first one being reducible. Then, the core level of $C$ is $S=\{\{1,2,3\},\{4,5\},\{6\}\}$ and the numerical core level is $3$. 
\end{example}
\begin{figure}[h]
\begin{tikzpicture}[grow=right,
  level distance=1.5cm,
  level 1/.style={sibling distance=1.5cm},
  level 2/.style={sibling distance=1cm}]
  
  \node[circle, draw] {$1$}
    child {node[circle, fill=black, inner sep=2pt] {}
      child {node[circle, fill=black, inner sep=2pt] {}
        child {node {1}}
        child {node {2}}}
      child {node {3}}
    }
    child {node[circle, fill=black, inner sep=2pt] {}
      child {node {4}}
      child {node {5}}
    }
    child {node {6}};
\end{tikzpicture}
\caption{The dual graph of the curve in \Cref{exa1}.}
\end{figure}

We may now associate to a partition $S$ a locally closed substack $\oTail{S}$ of $\Gcal_{1,n}$ consisting of all curves of core level $S$. We have the following explicit description:
\begin{equation}\label{Tails}
 \oTail{S}=\Min{s_0}\times\prod_{i=1}^k\oM_{0,1+s_i}.
\end{equation}
 We call this the locus of \emph{$S$-tails}. It has codimension $k$ in $\Gcal_{1,n}$. Since a separating node remains such under degeneration, it is easily seen that: \[\cTail{S}=\cup_{S'\preceq S}\oTail{S'}.\]
\begin{lemma}
 Loci of $S$-tails form a stratification of $\Gcal_{1,n}$. The open stratum is isomorphic to $\Min{n}$.
\end{lemma}

\begin{remark}
 The numerical core level may only decrease in a degeneration. We may therefore coarsen the previous stratification by putting all (closed) strata with the same (or smaller) numerical core level together $\oTail{s_0}=\bigcup_{\lvert S\rvert=s_0}\oTail{S}$ to get a totally ordered stratification:
 \begin{equation}\label{eq:strat G}
  \cTail{1}
  \subset\ldots\subset\cTail{n-1}\subset\cTail{n}=\Gcal_{1,n}.
 \end{equation}
 Notice that every $\Tail{m}$ contains some divisorial components (the codimension of components is bounded above by $\min({m,\lfloor\frac{n}{2}\rfloor})$).
\end{remark}

This is the stratification that we are going to use for patching the Chow ring of $\Gcal_{1,n}$.

There is a second stratification by singularity type, again refined into a partition by Bozlee, Kuo, and Neff. We close this section by introducing the terminology and recalling their result.

\begin{definition}
    Let $C$ be a log-canonically polarised Gorenstein curve with $n$ markings containing an elliptic $s_0$-fold point $q$. We say that $C$ has \emph{singularity level} $S$ if the $i^{\rm th}$ connected component of the normalisation of $C$ at $q$ is marked by $S_i$. The \emph{numerical singularity level} of $C$ is $s_0$.
\end{definition}
\begin{example}
    Let $C$ be a curve of genus one with an elliptic $3$-fold point, and suppose that the branches are marked respectively with $\{1,2\}$, $\{3,4\}$ and $\{5\}$. Then the singularity level of $C$ is $\{\{1,2\},\{3,4\},\{5\}\}$ and its numerical singularity level is $3$.
\end{example}

We may now associate to a partition $S$ a locally closed substack $\oEll{S}$ of $\Gcal_{1,n}$ consisting of all curves of singularity level $S$; we call it the locus of \emph{elliptic $S$-fold points}. If $m<n$ denotes the number of parts of $S$, every curve parametrised by $\oEll{S}$ contains an elliptic $m$-fold point. 

Recall from \cite[Lemma 2.2]{SmythI} that a marked genus one curve with an elliptic $m$-fold point $q$ is determined by its marked normalisation at $q$ (the geometric data $(D_i,\{p_j\}_{j\in S_i},\star_i)$, i.e. a set of semistable curves of genus zero, marked with the original markings together with the preimage of $q$, that we call $\star_i$), 
together with a \emph{generic} hyperplane $V$ in the direct sum of the cotangent lines of $D_i$ at $\star_i$ (the algebraic or "attaching" data), determining the linear functions that descend to the elliptic singularity. The genus one curve is recovered by first gluing together the genus zero curves $D_i$ at the markings $\star_i$ obtaining a curve $\sigma\colon C^{sn} \to C$ homeomorphic to $C$ but with an ordinary $m$-fold point $\star$ in place of $q$, and then defining the sheaf of regular functions around $q$ to be the subsheaf of $\sigma_*\mathcal{O}_{C^{sn}}$ generated by (liftings of the elements of) $V$ together with functions vanishing at $\star_i$ of order at least $2$.

When the connected components of the normalisation are stable marked curves of genus zero, the geometric data $(D_i,\{p_j\}_{j\in S_i},\star_i)$ can be parametrised by the moduli space ${\prod_{i=1}^m\oM_{0,\star_i\cup S_i}}$. To parametrise the attaching data $V$, let $\mathbb L_{0,\star_i}$ be the line bundle over ${\prod_{i=1}^m\oM_{0,\star_i\cup S_i}}$ corresponding to the cotangent line bundle at the universal marking $\star_i$. Then, the hyperplane $V$ is dual to a line in $\bigoplus_{i=1}^m \mathbb L_{0,\star_i}^{\vee}$ and genericity means that the line is not allowed into any of the coordinate hyperplanes. Note that the singletons in $S$ have been omitted, since the corresponding moduli spaces of $2$-pointed rational curves is $B\Gm$, and the automorphism group of the curve acts with weight $\pm1$ on the tangent lines at the markings; in particular $m=s
_0$. The parameter space of geometric plus attaching data is thus the complement  of the coordinate hyperplanes $\PP (H_j)=\PP(\oplus_{i\neq j} \mathbb{L}_{0,\star_i}^{\vee})$ in $\PP(\oplus_{i=1}^{m} \mathbb{L}_{0,\star_i}^{\vee})$. 

Furthermore, Smyth showed that one can obtain a Gorenstein  elliptic singularity even when the line heads into the boundary $\cup_j \PP(H_j)$ - by blowing up at $\star_i$ the universal curve over the corresponding boundary divisor before performing the attachment. 
Thus, the curve $C$ belongs to the stratum $\cap_{j \in J} H_j$ if and only if the irreducible components of its normalisation containing the markings $\star_j$ have only one other special point (either a marking or a node).
In the end one gets the following explicit description, which is a mild variation on \cite[\S2.3]{SmythII} (note that Smyth uses a different convention than ours, namely he denotes $\PP(V)$ the projective bundle of \emph{hyperplanes} in $V$): \begin{equation}\label{Ells}
\oEll{S}=[\on{Tot}(\bigoplus_{i=1}^m \mathbb L_{0,\star_i}^{\vee})_{\prod_{i=1}^m\oM_{0,\star_i\cup S_i}}/\Gm].                                \end{equation}
Here, the only difference with Smyth's construction is that we do not remove the zero-section of the bundle before taking the quotient by $\Gm$. Those points correspond to the elliptic $S$-singularities whose partial normalisation consists only of strictly semistable components. Indeed, for every $i=1,\ldots,m$, the marking $\star_i$ in the normalisation $D_i$ belongs to a component that has only one other special point. We deduce that $(D_i,\{p_j\}_{j\in S_i},\star_i)$ 
has automorphism group $\Gm^{m}$, i.e. one copy of $\Gm$ for every branch of the $m$-fold point. The subgroup of automorphisms that fix the attaching data is a torus of rank $m-1$, thus the curves corresponding to these points are precisely those whose automorphism group is $\Gm$. Equivalently, they belong to $\oTail{S}\cap\oEll{S}$. Finally, the description above shows that $\oEll{S}$ is smooth.

It follows from the deformation theory of elliptic singularities that: \[\cEll{S}=\bigcup_{S\preceq S'}\oEll{S},\]
a closed substack of codimension $s_0+1$ in $\Gcal_{1,n}$. 

\begin{remark}
 The complement of all $\oEll{S}$ in $\Gcal_{1,n}$ is the Deligne--Mumford stack $\oM_{1,n}$.
\end{remark}
Finally, the definition and main properties of $\oM_{1,n}(Q)$:
\begin{definition}[{\cite[Definition 1.7]{BKN}}]
 Let $Q$ be a downward-closed (closed under coarsening) subset of $\on{Part}([n])$ that does not contain $S_{\max}=\{\{1\},\ldots,\{n\}\}$.
 
 An $n$-pointed Gorenstein curve $C$ of arithmetic genus one is $Q$-stable if:
 \begin{enumerate}
  \item for every genus one subcurve $Z\subseteq C$, the level of $Z$ does not belong to $Q$;
  \item if $q\in C$ is an elliptic singularity, the level of $q$ belongs to $Q$.
  \item as a pointed curve, $C$ has finitely many automorphisms.
 \end{enumerate}
\end{definition}

\begin{theorem}[{\cite[Theorems 1.8 and 1.11]{BKN}}]
 Over $\Spec(\ZZ[\frac{1}{6}])$, the moduli stack of $Q$-stable curves $\oM_{1,n}(Q)$ is a modular compactification of $\Mcal_{1,n}$, i.e. an open and proper Deligne--Mumford substack of $\Ucal_{1,n}$ that contains $\Mcal_{1,n}$ as a dense open. Every modular compactification of $\Mcal_{1,n}$ within $\Gcal_{1,n}$ is of this form.
\end{theorem}

\section{Geometry of the open substack parametrising minimal curves} \label{sec:classes}

In this section we explain an explicit description of $\Min{n}$ due to Lekili and Polishchuk, which we exploit in order to describe the Chow rings of strata. We compute the fundamental classes of strata of elliptic $S$-fold points by realising a connection with loci of non-separating nodes.

Considerations of derived categories and homological mirror symmetry led Lekili and Polishchuk to consider the following moduli problem: let $\Ucal_{1,n}^{sns}$ denote the stack of $n$-pointed curves of arithmetic genus one, such that the line bundle $\OO_C(p_1+\ldots+p_n)$ is ample, and $h^0(C,\Ocal(p_i))=1$ - $sns$ stands for \emph{strongly non-special}, which is precisely this condition. The two conditions together imply that every irreducible component of $C$ must contain at least one marking, and that there cannot be any rational tail. Moreover, Lekili and Polishchuk find very explicit normal forms for strongly non-special curves, which allows them to describe the moduli stack explicitly (we will come back to the normal forms in Paragraphs 3.1-5 below): it turns out that these are all Gorenstein curves, and in particular they contain at worst an elliptic $n$-fold point, they are minimal, and therefore $\OO_C(p_1+\ldots+p_n)=\logcan$ (they are log-canonically polarised). In characteristic $2,3$, the cusp and tacnode possess extra automorphisms that make the normal form more complicated and less homogeneous. We will henceforth work over $\Spec(\ZZ[\frac{1}{6}])$ and identify $\Ucal_{1,n}^{sns}$ with $\Min{n}$ without further mention of the former.

Notice that $\Min{n}$ contains a unique point $[C_{1,n}]$ with $\Gm$-stabiliser, corresponding to the elliptic $n$-fold point. Removing it we obtain Smyth's $\oM_{1,n}(n-1)$. Moreover, $\Min{n}$ is smooth for $n\leq 5$, and $[C_{1,6}]$ is the only singular point of $\Min{6}$, so we better pass to $\oM_{1,6}(5)$ when working with $n=6$.

To state the following theorem we need to introduce some notation: we denote by $V_d$ the irreducible $\Gm$-module of weight $d$, and we denote $\mathcal{P}(d_0,\ldots, d_n)$ the weighted projective stack with weights $d_0,\ldots,d_n$, i.e. the quotient stack $[((\oplus_{i=0}^{n} V_{d_i})\smallsetminus\{0\}) /\Gm]$.
\begin{theorem}[{\cite[Proposition 1.1.5, Theorem 1.4.2, Theorem 1.5.7, Proposition 1.6.1]{LekiliP}}]\label{LPtheorem}
 The $\Gm$-bundle corresponding to the Hodge line bundle $\Hcal=\pi_*\omega_{\pi}$ over $\Min{n}$ is an affine scheme (it is an affine space for $n\leq 5$, and it is cut out by quadratic equations for $n\geq 6$, with the cone point corresponding to $[C_{1,n}]$); in particular we have:
  \begin{enumerate}
        \item [(n=1)] $\Min{1} \simeq [V_4\oplus V_6/\Gm]$ {\hfill and $\oM_{1,1}\simeq \Pcal(4,6)$;}
        \item [(n=2)] $\Min{2} \simeq [V_2\oplus V_3 \oplus V_4/\Gm]$ \hfill and $\oM_{1,2}(1)\simeq \Pcal(2,3,4)$;
        \item [(n=3)] $\Min{3} \simeq [V_1\oplus V_2^{\oplus 2} \oplus V_3/\Gm]$  \hfill  and $\oM_{1,3}(2)\simeq \Pcal(1,2,2,3)$;
        \item [(n=4)] $\Min{4} \simeq [V_1^{\oplus 3}\oplus V_2^{\oplus 2}/\Gm]$ \hfill and $\oM_{1,4}(3)\simeq \Pcal(1,1,1,2,2)$;
        \item [(n=5)] $\Min{5} \simeq [V_1^{\oplus 6}/\Gm]$ \hfill and $\oM_{1,5}(4)\simeq \PP^5$;
        \item [(n=6)] $\Min{6}^{sm}=\oM_{1,6}(5) \simeq \on{Gr}(2,5)$.
    \end{enumerate}
Under these identifications, the Hodge line bundle $\Hcal$ is carried to $\OO_{\Pcal}(1)$ (resp. for $n=6$, to the $\OO(1)$ of the Pl\"ucker embedding).

There always is a flat morphism $\Min{n+1}\to\Min{n}$ induced by a $\Gm$-equivariant morphism between the corresponding $\Gm$-modules, identifying the former with the affine universal curve over the latter, whose fibre over $[(C;p_1,\ldots,p_n)]$ is the affine curve $C\setminus(p_1+\ldots+p_n)$. The rational map $\oM_{1,n+1}(n)\dashrightarrow \oM_{1,n}(n-1)$ identifies the projective universal curve over the latter with the blow-up of the former in the $n$ points $P_{i,n+1}$, corresponding to elliptic $n$-fold points with markings $p_i$ and $p_{n+1}$ on the same branch.
\end{theorem}
\begin{proposition}\label{ChowMin}
 For $n\leq 5$ we have $A^*(\Min{n})=\ZZ[\lambda]$.
 
 For $n=6$ we have $A^*(\Min{6}^{sm})=\ZZ[\lambda,\nu]/(\lambda^4-\lambda^2\nu-\nu^2,\lambda^5-3\lambda^3\nu+2\lambda\nu^2)$.

 Pullback along the forgetful morphism $\Min{n+1}\to\Min{n}$ identifies the Hodge bundles and $\lambda$-classes.
\end{proposition}

\begin{proof}
Assume $n\leq 5$. From \Cref{LPtheorem} we have that $\Min{n}\simeq [W/\Gm]$ for some finite rank representation $W$; in particular, it is a vector bundle over the classifying stack $B\Gm=[\Spec(\bf{k})/\Gm]$ and the induced map to $B\Gm$ classifies the $\Gm$-torsor associated to the Hodge line bundle. As pulling back along vector bundles induces an isomorphism of Chow rings, we deduce that $A^*(\Min{n})=A^*(B\Gm)=\mathbb{Z}[t]$. The generator $t$ is equal to the pullback of the class of universal line bundle over $B\Gm$, and hence we can identify it with the Hodge line bundle $\lambda$.

 We explain the notation in the case $n=6$. Chow rings of Grassmannians are well known \cite[\S\S 4.3 and 5.6]{3264}, so we only highlight their main features. There is a short exact sequence of tautological vector bundles
 \[0\to\Scal\to\OO_G^{\oplus 5}\to\Qcal\to 0,\]
 and by setting $s_i=c_i(\Scal)$ and $q_j=c_j(\Qcal)$ we obtain generators for the Chow ring. We may set $\lambda=q_1$ (as explained below) and $\nu=q_2$. The relations are then given by the homogeneous summands of the polynomial $c(\Scal)c(\Qcal)-1$, where $c(-)$ indicates the total Chern class:
 \[1=(1+s_1+s_2)(1+\lambda+\nu+q_3)\Rightarrow\]
 \begin{align*}
     0=&s_1+\lambda&\Rightarrow &s_1=-\lambda,\\
     0=&s_2+s_1\lambda+\nu&\Rightarrow &s_2=\lambda^2-\nu,\\
     0=&s_2\lambda+s_1\nu+q_3&\Rightarrow &q_3=2\lambda\nu-\lambda^3,\\
     0=&s_2\nu+s_1q_3&\Rightarrow &0=\lambda^4-\lambda^2\nu-\nu^2,\\
    0=&s_2q_3&\Rightarrow&0=\lambda^5-3\lambda^3\nu+2\lambda\nu^2.\\
 \end{align*}
 
 The following set of generators is more useful for our purposes: for a fixed complete flag
 \[F:0\subset F_1\subset F_2\subset F_3\subset F_4\subset F_5=V,\]
 and $3\geq a_1\geq a_2\geq 0$, the Schubert cycles $\Sigma_{a_1,a_2}(F)\subseteq\on{Gr}(2,5)$ are defined as the loci of lines $\ell\subset \PP(V)$ that intersect $\PP(F_{5-a_1})$ and are contained in $\PP(F_{6-a_2})$. They are closed subvarieties of codimension $a_1+a_2$. Schubert classes, denoted by $\sigma_{a_1,a_2}$ or simply $\sigma_{a_1}$ in case $a_2=0$, form an additive basis of $A^*(\on{Gr}(2,5))$. 
 From Pieri's formula it follows that $\sigma_{a_1,a_2}=\sigma_{a_1}\sigma_{a_2} - \sigma_{a_1+1}\sigma_{a_2-1}$, with the convention that $\sigma_a=0$ for $a>3$. In particular, the special Schubert classes $\sigma_1$ and $\sigma_2$ are multiplicative generators.

The $\OO(1)$ of the Pl\"ucker embedding can be identified with $\on{det}(\Qcal)$. Moreover, the vanishing locus of the Pl\"ucker coordinate $p_{45}$ can be identified with the locus of lines intersecting the codimension two linear subspace $F_3=\{z_4=z_5=0\}$, that is $\Sigma_1(F)$.
Thus $\lambda=q_1=\sigma_1$.
 
 The isomorphism between the two presentations is given by:
 \[s_1=-\sigma_1, s_2=\sigma_{1,1}, (\lambda=)q_1=\sigma_1, (\nu=)q_2=\sigma_2, q_3=\sigma_3.\]

The last statement requires no restrictions on $n$. Lekili and Polishchuk identify $\Min{n}$ with the $\Gm$-quotient of an explicit affine variety - that is nothing but the total space of the $\Gm$-bundle associated to the Hodge line bundle over $\Min{n}$ - and the forgetful morphism $\Min{n+1}\to\Min{n}$ is induced by a $\Gm$-equivariant map between the corresponding affine varieties, therefore it identifies the respective Hodge bundles.
\end{proof}
\begin{definition}
Let $\Nod_S$ denote the closure of the locus of curves with $\ell(S)$ non-separating nodes and that many rational irreducible components in the core, each marked with a part $S_i\subseteq [n]$.
\end{definition}
The main goal of this section is calculating the fundamental classes of singularity strata in the open stratum $\Min{n}^{sm}$. We exploit Lekili and Polishchuk's normal forms.  We will also use, without further mention, the natural action of the symmetric group $S_n$ on $\Min{n}$, inducing a trivial action on $A^*(\Min{n})$\footnote{It is clear that the action is trivial on $\lambda$. For $n=6$, one has $\nu=\Nod_{\{i,j,h,k\},\{\ell,m\}}$ for any choice of the indexes \cite[\S 2.1]{Coskun}, hence $\nu$ is invariant as well. This can also be proved directly by looking at the degree $4$ relation $\nu^2=\lambda^4-\lambda^2\nu$, which implies that either $\nu$ is fixed, or $\nu\mapsto -\lambda^2-\nu$. Plugging the latter into the degree $5$ relation $\lambda^5-3\lambda^3\nu-2\lambda\nu$ one gets $6\lambda^5+6\lambda^3\nu+\lambda\nu$, which does not belong to the ideal of relations since the following matrix is full rank:
\[\begin{pmatrix}
    1&-1&-1\\ 1&-3&2\\ 6&6&1
\end{pmatrix}.\]}; it follows that $[\Ell{S}]=[\Ell{\sigma\cdot S}]$ for every partition $S$ and every permutation $\sigma$. 

We next establish the relationship between strata of elliptic singularities and loci of non-separating nodes. For an intuition, consider a minimal genus one curve obtained by gluing two $\PP^1$s at two nodes. What happens when the nodes come together? Now, the two $\PP^1$s are joined at a single point, but the arithmetic genus of the curve must still be one, hence that point must be a tacnode. The general idea is then to express loci of elliptic singularities as the intersection of loci of nodal singularities and certain appropriate cycles.

\begin{lemma}\label{lem:nod}
    The relation $\Nod_{[n]}=12\lambda$ holds in $A^1(\Min{n})$ for every $n\geq 1$.
\end{lemma}
\begin{proof}
    $\Nod_{[n]}$ can be identified with the discriminant locus of singular curves. For $n=1$ we have $\Nod_{[1]}=V(27b^2+16a^3)$ (see \S3.1 below for the equations of the universal curve), hence the identification of $[\Nod_{[1]}]$ with $12\lambda$. The formula then holds for every $n$ by pulling it back along the forgetful morphism $\Min{n+1}\to\Min{n}$.
\end{proof}
\begin{remark}
    $\Nod_S$ is irreducible for $\ell(S)\leq 3$ and reducible otherwise, since the ordering of the parts is relevant up to cycling and inverting. 
\end{remark}

\subsection{$n=1$} We have an identification $\Min{1}\simeq [V_{4,6}/\Gm]$, with coordinates $(a,b)$, and affine universal curve $y^2-x^3=ax+b$ in $\Aaff^2_{x,y}$ \cite[Eqn. (1.11)]{LekiliP}.
\begin{lemma}\label{lem:cusps}
 $[\Ell{[1]}]=24\lambda^2\in A^*(\Min{1})$ and $[\Ell{[n]}]=24\lambda^2\in A^*(\Min{n})$ for every $n\geq1$.
\end{lemma}
\begin{proof}
 The unique cuspidal curve corresponds to the origin $\{a=b=0\}$ of $[V_{4,6}/\Gm]$, with $\Gm$-equivariant fundamental class $4\lambda\ \cdot 6\lambda=24\lambda^2$, hence the first claim. The second claim follows from pullback along the forgetful map $\Min{n+1}\to\Min{n}$ of \Cref{LPtheorem}.
\end{proof}

\subsection{$n=2$}\label{sub:n=2}
We have an identification $\Min{2}\simeq [V_{2,3,4}/\Gm]$, with coordinates $(a,b,c)$, and affine universal curve $y^2-yx^2=a(y-x^2)+bx+c$ in $\Aaff^2_{x,y}$ \cite[Eqn. (1.9)]{LekiliP}.

\begin{lemma}\label{n2}
 $[\Nod_{\{1\},\{2\}}]=12\lambda^2$ and $[\Ell{\{1\},\{2\}}]=2\lambda\cdot[\Nod_{\{1\},\{2\}}]$.
 
 Moreover for $n\geq 2$ and for every partition $S=\{S_1,S_2\}$ we have $[\Ell{S_1,S_2}]=2\lambda\cdot [\Nod_{S_1, S_2}]$ in $A^*(\Min{n}^{sm})$.
\end{lemma}
\begin{proof}
 The curve corresponding to $(a,b,c)$ is binodal if and only if $b=c=0$, so the $\Gm$-equivariant fundamental class of $[\Nod_{\{1\},\{2\}}]$ is $3\lambda\cdot\ 4\lambda=12\lambda^2$, hence the first claim.
 The unique tacnodal curve corresponds to the origin $\{a=b=c=0\}$ of $[V_{2,3,4}/\Gm]$, with $\Gm$-equivariant fundamental class $2\lambda\ \cdot [\Nod_{\{1\},\{2\}}]=24\lambda^3$, hence the second claim.
 
 For the last claim, we argue by induction on $n$: let $A_n$ be the preimage in $\Min{n}$ (along the forgetful map of \Cref{LPtheorem}) of $\{a=0\}\subseteq\Min{2}\simeq [V_{2,3,4}/\Gm]$, so $[A_n]=2\lambda$. Let $Y$ be an irreducible component of maximal dimension of $A_n\cap \Nod_{S_1,S_2}$; without loss of generality, we can assume that $n\in S_2$ and $|S_2|\geq 2$. By construction $\pi(Y)\subset A_{n-1}\cap \Nod_{S_1, S_2\smallsetminus\{n\}}=\Ell{S_1,S_2\smallsetminus\{n\}}$, hence $Y\subset \Ell{S_1,S_2}\cup \Ell{S_1\cup\{n\}, S_2\smallsetminus\{n\}}$. This, combined with the fact that $\on{codim}(Y)\leq 3$ by construction, implies that $\on{codim}(Y)=3$. Suppose that $Y=\Ell{S_1\cup\{n\},S_2\smallsetminus\{n\}}$: then $\Ell{S_1\cup\{n\},S_2\smallsetminus\{n\}}= Y\cap \Nod_{S_1\cup\{n\},S_2\smallsetminus\{n\}} \subset \Nod_{S_1,S_2} \cap \Nod_{S_1\cup\{n\},S_2\smallsetminus\{n\}}=\Nod_{S_1,S_2\smallsetminus\{n\},\{n\}}$, which is not the case. It follows that $Y=\Ell{S_1,S_2}$ and $[A_n]\cdot [\Nod_{S_1,S_2}]=[Y]=[\Ell{S_1,S_2}]$.
\end{proof}

\subsection{$n=3$}
We have an identification $\Min{3}\simeq [V_{1,2,2,3}/\Gm]$, with coordinates $(a,b,c,d)$, and affine universal curve $xy^2-x^2y=axy+bx+cy+d$ in $\Aaff^2_{x,y}$ \cite[Eqn. (1.2)]{LekiliP}. 

\begin{lemma}\label{n3}
 $[\Nod_{\{i,j\},\{k\}}]=6\lambda^2$ and $[\Ell{\{i,j\},\{k\}}]=12\lambda^3$.
 
 $[\Nod_{\{1\},\{2\},\{3\}}]=12\lambda^3$ and $[\Ell{\{1\},\{2\},\{3\}}]=\lambda\cdot[\Nod_{\{1\},\{2\},\{3\}}]=12\lambda^4$.
 
 Moreover, for every $\neq 3$ and for every partition $S=\{S_1,S_2,S_3\}$ we have $\Ell{\{S_1,S_2,S_3\}}=\lambda\cdot\Nod_{\{S_1,S_2,S_3\}}$ in $A^*(\Min{n}^{sm})$.
\end{lemma}
\begin{proof}
 By pulling back along the flat forgetful morphism $\Min{3}\to\Min{2}$ we deduce $[\Nod_{\{1\},\{2,3\}}]+[\Nod_{\{1,3\},\{2\}}]=\pi^*[\Nod_{\{1\},\{2\}}]=12\lambda^2$, hence $[\Nod_{\{i,j\},\{h\}}]=6\lambda^2$; from \Cref{n2} we deduce $[\Ell{\{i,j\},\{k\}}]=12\lambda^3$.
 
  From the equation of the universal curve, we see that $\Nod_{\{1\},\{2\},\{3\}}$ is cut out by the equations $b=c=d=0$, whence it has $\Gm$-equivariant fundamental class $2\lambda\cdot2\lambda\cdot3\lambda=12\lambda^3$. On the other hand, the unique elliptic $3$-fold point is represented by the origin, which is cut out by the extra equation $a=0$ of weight $1$, whence the second claim. We can then argue by induction.
  
  The induction step works as in the case $n=2$: let $A_n$  denote the preimage in $\Min{n}$ of $\{a=0\}$ in $\Min{3}$. Without loss of generality, assume $n \in S_3$ and $|S_3|\geq 2$; if $Y$ is an irreducible component of $A_n\cap \Nod_{S_1,S_2,S_3}$, then $Y$ must be contained in $\pi^{-1}(\Ell{S_1,S_2,S_3\smallsetminus\{n\}})=\Ell{S_1,S_2,S_3}\cup \Ell{S_1\cup\{n\},S_2,S_3\smallsetminus\{n\}}\cup \Ell{S_1,S_2\cup\{n\},S_3\smallsetminus\{n\}}$. We deduce that $Y$ has codimension $4$ and moreover it cannot be equal to neither $\Ell{S_1\cup\{n\},S_2,S_3\smallsetminus\{n\}}$ or $\Ell{S_1,S_2\cup\{n\},S_3\smallsetminus\{n\}}$, as otherwise it would be contained in $\Nod_{S_1,S_2,S_3}\cap \Nod_{S_1\cup\{n\},S_2,S_3\smallsetminus\{n\}}=\Nod_{S_1\cup\{n\},S_2,S_3\smallsetminus\{n\},\{n\}}$ or $\Nod_{S_1,S_2,S_3}\cap \Nod_{S_1,S_2\cup\{n\},S_3\smallsetminus\{n\}}=\Nod_{S_1,S_2,\{n\},S_3\smallsetminus\{n\}}$.
\end{proof}

\subsection{$n=4$}
We have an identification $\Min{4}\simeq [V_{1,1,1,2,2}/\Gm]$, with coordinates $(a,c_4,\overline{c}_4,c,\overline{c})$, and universal affine curve defined by the following equations in $\Aaff^3_{x,y,z}$:
\[ xz=xy+c_4z+\overline{c}_4x-c,\quad yz=xy+(a+c_4+\overline{c}_4)(z-\overline{c}_4)+(\overline{c}-c). \]
The morphism $\pi\colon [V_{1,1,1,2,2}/\Gm]\to [V_{1,2,2,3}/\Gm]$ is given by:
\[ (a,c_4,\overline{c}_4,c,\overline{c}) \longmapsto (a,\overline{c}-c,c,c_4(a+c_4+\overline{c}_4)^2-c_4^2(a+c_4+\overline{c}_4)-ac_4(a+c_4+\overline{c}_4-(\overline{c}-c)c_4-c(a+c_4+\overline{c}_4))\]
and it corresponds to the universal affine curve over the latter by setting $x=c_4$ and $y=a+c_4+\overline{c}_4$ \cite[Proposition 1.1.5]{LekiliP}.
\begin{lemma}\label{n4} The following equalities hold in $A^*(\Min{4})$:
\begin{itemize}
 \item $[\Nod_{\{i,j,h\},\{k\}}]=4\lambda^2$ and $[\Ell{\{i,j,h\},\{k\}}]=8\lambda^3$, 
 
 \noindent $[\Nod_{\{i,j\},\{h,k\}}]=2\lambda^2$ and $[\Ell{\{i,j\},\{h,k\}}]=4\lambda^3$.
 \item $[\Nod_{\{i,j\},\{h\},\{k\}}]=4\lambda^3$ and $[\Ell{\{i,j\},\{h\},\{k\}}]=4\lambda^4$.
 \item $[\Nod_{\{1\},\{2\},\{3\},\{4\}}]=4\lambda^4$ and $[\Ell{\{1\},\{2\},\{3\},\{4\}}]=\lambda\cdot[\Nod_{\{1\},\{2\},\{3\},\{4\}}]=4\lambda^5$.
\end{itemize}
Moreover, $[\Ell{S_1,S_2,S_3,S_4}]=\lambda\cdot[\Nod_{S_1,S_2,S_3,S_4}]$ holds in $A^*(\Min{n}^{sm})$ for all $n\geq4$.
\end{lemma}
\begin{proof}
 Pulling back from $\Min{3}$ we deduce:
 \[[\Nod_{\{1,2,4\},\{3\}}]+[\Nod_{\{1,2\},\{3,4\}}]=\pi^*[\Nod_{\{1,2\},\{ 3\}}]=6\lambda^2.\]
 This does not determine either class directly; to solve this, we make use of the explicit equations we wrote down for the $n=3$ case. The universal affine curve over $V(d+ac,b+c)\subset V_{1,2,2,3}$ has equation $(xy-c)(y-x-a)=0$; the identification of the affine curve with $V_{1,1,1,2,2}$ sends $x\mapsto c_4$ and $y\mapsto (a+c_4+\overline{c}_4)$, and by combining this with $b=\overline{c}-c$ we get that the preimage of $V(d+ac,b+c)$ in $V_{1,1,1,2,2}$ corresponds to the subscheme \[ V(\overline{c},\overline{c}_4(c_4(a+c_4+\overline{c}_4)-c)) = V(\overline{c},\overline{c}_4) \cup V(\overline{c},c_4(a+c_4+\overline{c}_4)-c). \]
A point in the first of the two components on the right corresponds to a point on the irreducible component of the affine curve over $V(d+ac,b+c)$ of equation $y-x-a=0$, which contains only one marking, i.e. the only point at infinity in the projective closure. Up to rearranging the labels of the markings, we deduce that $\Nod_{\{1,2\},\{3,4\}}=V(\overline{c},\overline{c}_4)$, hence its fundamental class is equal to $2\lambda\cdot\lambda=2\lambda^2$ and consequently $[\Nod_{\{i,j,h\},\{k\}}]=4\lambda^2$.
By \Cref{n2} we deduce $[\Ell{\{i,j\},\{h,k\}}]=4\lambda^3$ and $[\Ell{\{i,j,h\},\{k\}}]=8\lambda^3$.

The second claim follows directly from \Cref{n3}.

For the third claim, $[\Nod_{\{1\},\{2\},\{3\},\{4\}}]=\{c_4=\bar{c}_4=c=\bar{c}=0\}$ holds by direct inspection. The elliptic $4$-fold point corresponds to the origin of $\Aaff^5$, which is cut out by the extra equation $a=0$. The last statement can be proved by induction as in the previous lemmas.
\end{proof}

\subsection{$n=5$}
We have an identification $\Min{5}=[V_{1,1,1,1,1,1}/\Gm]$. 
\begin{lemma}\label{n5} The following equalities hold in $A^*(\Min{5})$:
 \begin{itemize}
  \item $[\Nod_{\{i,j,h,k\},\{\ell\}}]=3\lambda^2$ and $[\Ell{\{i,j,h,k\},\{\ell\}}]=6\lambda^3$,
  
  \noindent $[\Nod_{\{i,j,h\},\{k,\ell\}}]=\lambda^2$ and $[\Ell{\{i,j,h\},\{k,\ell\}}]=2\lambda^3$.
  \item $[\Nod_{\{i,j\},\{h,k\},\{\ell\}}]=\lambda^3$ and $[\Ell{\{i,j\},\{h,k\},\{\ell\}}]=\lambda^4$, 
  
  \noindent $[\Nod_{\{i,j,h\},\{k\},\{\ell\}}]=2\lambda^3$ and $[\Ell{\{i,j,h\},\{k\},\{\ell\}}]=2\lambda^4$.
  \item $[\Nod_{\{i,j\},\{h\},\{k\},\{\ell\}}]=\lambda^4$ and $[\Ell{\{i,j\},\{h\},\{k\},\{\ell\}}]=\lambda^5$.
  \item $[\Ell{\{1\},\{2\},\{3\},\{4\},\{5\}}]=\lambda^6$.
 \end{itemize}
\end{lemma}
\begin{proof}
 The first three statements follow by pulling back from \Cref{n4} and symmetry. The last one follows from the identification of the unique elliptic $5$-fold point with the origin of $\Aaff^6$.
\end{proof}

\subsection{$n=6$} We explain the identification of $\oM_{1,6}(5)$ with $\GG=\on{Gr}(2,5)$ briefly; see \cite[\S1.7]{LekiliP}. Pl\"ucker embeds $\GG$ in $\PP^9$ of degree $5$, indeed $\sigma_1^6=5\sigma_{3,3}$ (and $\sigma_{3,2}$ is the class of a line). Fix a linear space $L$ of codimension $6$ in $\PP^9$, intersecting $\GG$ in five distinct points $p_1,\ldots,p_5$ (this can be achieved over $\ZZ$ by \cite[Proposition 1.7.1]{LekiliP}). For any point $q$ of $\GG$ other than these five, the linear span $M_q=\langle q,p_1,\ldots,p_5\rangle$ has dimension $4$, and it intersects $\GG$ in a curve $C_q$ which, marked with $p_1,\ldots,p_5$, is indeed $4$-stable. The rational map $\GG\dashrightarrow\oM_{1,5}(4),q\mapsto (C_q,p_1,\ldots,p_5)$ can be identified with the restriction to $\GG$ of the linear projection $\PP^9\dashrightarrow\PP^5, q\mapsto M_q$ out of the subspace $L$, and thus resolved by blowing up $\GG$ in $p_1,\ldots,p_5$, identifying the blow-up with the universal curve over $\oM_{1,5}(4)$, and $\GG$ with $\oM_{1,6}(5)$. In particular, a point $q \in \GG$ corresponds to the $5$-stable $6$-marked curve $(C_q,p_1,\ldots,p_5,q)$ except when $q=p_i$ or $q$ is an elliptic $m$-fold point: in the first case, the corresponding curve sprouts a $\PP^1$ marked by $p_i$ and $q$; in the second case, the curve sprouts a $\PP^1$ marked by $q$, glued at the singular point so to form an elliptic $(m+1)$-fold point. Notice, moreover, that the embedding of $C_q$ in $M_q$ is provided by the very ample line bundle $\omega_{C_q}(p_1+\ldots+p_5)$, so the degree (in $M_q$ and therefore also in $\PP^9$) of any irreducible component of the curve is determined by the number of $p_i$ it contains.

\begin{lemma}\label{lm:schubert classes}
    The following equalities hold in $A^*(\oM_{1,6}(5))$:
    \begin{itemize}
     \item $[\Nod_{\{i,j,h,k\},\{\ell,m\}}]=\sigma_2=\nu$, 
     \item $[\Nod_{\{i,j,h\},\{k,\ell,m\}}]=\sigma_{1,1}=\lambda^2-\nu$,
     \item $[\Nod_{\{i,j\},\{h,k\},\{\ell,m\}}]=\sigma_3=2\lambda\nu-\lambda^3$,
     \item $[\Nod_{\{i,j,h\},\{k,\ell\},\{m\}}]=\sigma_{2,1}=\lambda^3-\lambda\nu$.
    \end{itemize}
\end{lemma}
\begin{proof}
The first item is explained in \cite[\S2.1]{Coskun}. Let $\Lambda_i$ denote the $2$-plane corresponding to $p_i$. Consider the Schubert variety $\Sigma_2(\Lambda_1)$.  If $q\in \Sigma_2(\Lambda_1)$, then $\Lambda_q$ and $\Lambda_2$ intersect in a line $\ell$; the Schubert variety $\Sigma_{3,2}(\ell)$ is a line in $\on{Gr}(2,5)$ passing through $q$ and $p_1$. This implies that $C_q=M_q \cap \on{Gr}(2,5)$ contains a line marked by $p_1$ and $q$ only, i.e. $[C_q]\in\Nod_{\{1,6\},\{2,3,4,5\}}$.

Similarly, we can identify the Schubert variety $\Sigma_{1,1}(F_3\subseteq F_4= \langle\Lambda_1,\Lambda_2\rangle)$ with the locus of reducible curves containing a conic marked by $p_1,p_2$, and $q$, i.e. with $\Nod_{\{1,2,6\},\{3,4,5\}}$. Indeed, the Schubert variety $\Sigma_{1,1}(F_3\subset F_4)$ can be identified with the Grassmannian $\on{Gr}(2,F_4)$; in particular, both $p_1$ and $p_2$ belong to it. Recall that any conic $D\subset\on{Gr}(2,5)$ is of the form 
\[D=\{\Lambda \subset k^{5} \text{ such that }\PP\Lambda\subset Q\} \]
where $Q$ is a quadric surface in $\PP^4$, and viceversa. In particular, if $q=[\Lambda_q]$ is a point in $\Sigma_{1,1}(F_3\subset F_4)$ we have that $\PP\Lambda_q$, $\PP\Lambda_1$, and $\PP\Lambda_2$ are lines in $\PP F_4\simeq \PP^3$, thus there exists a quadric surface containing these three lines. We deduce that there exists a conic $D$ containing $q$, $p_1$ and $p_2$, and consequently $C_q=D \cup D'$. As $D$ is a conic, only two of the five distinct points $p_1,\ldots, p_5$ can lie on it, and thus $C_q$ belongs to 
$\Nod_{\{1,2,6\},\{3,4,5\}}$.

Consider now the Schubert variety $\Sigma_{2,1}(F)$ with respect to the partial flag $F_2=\Lambda_{1}$ and $F_4=\langle\Lambda_{1},\Lambda_{2}\rangle$. By the previous paragraph, for $q\in \Sigma_{2,1}(F)$, the curve $C_q$ is reducible with $p_1,p_2$, and $q$ contained in a conic, but also $p_1,q$ contained in a line, hence the conic must be reducible, and $\Sigma_{2,1}(F)=\Nod_{\{1,6\},\{2\},\{3,4,5\}}$.

By \Cref{n5} we know that $[\Nod_{\{1,2\},\{3,4\},\{5\}}]=\lambda^3$ in $A^3(\Min{5})$. It follows then that $\pi^*[\Nod_{\{1,2\},\{3,4\},\{5\}}]=\sigma_1^3$. On the other hand $[\pi^{-1}(\Nod_{\{1,2\},\{3,4\},\{5\}})]=[\Nod_{\{1,2,6\},\{3,4\},\{5\}}]+[\Nod_{\{1,2\},\{3,4,6\},\{5\}}]+[\Nod_{\{1,2\},\{3,4\},\{5,6\}}]$, hence $[\Nod_{\{1,2\},\{3,4\},\{5,6\}}]=\sigma_1^3-2\sigma_{2,1}=\sigma_3$.

Finally, observe that all the conclusions are unaffected by renumbering.
\end{proof}

\begin{lemma}\label{n6-nod}
The following equalities hold in $A^*(\oM_{1,6}(5))$:
\begin{itemize}
 \item $[\Ell{\{i,j,h,k\},\{\ell,m\}}]=2\lambda\nu$, $[\Ell{\{i,j,h\},\{k,\ell,m\}}]=2\lambda^3-2\lambda\nu$, $[\Ell{\{i,j,h,k,\ell\},\{m\}}]=6\lambda^3-2\lambda\nu$.
 \item $[\Ell{\{i,j\},\{h,k\},\{\ell, m\}}]=\lambda^2\nu-\nu^2$, $[\Ell{\{i,j,h\},\{k,\ell\},\{m\}}]=\nu^2$, $[\Ell{\{i,j,h,k\},\{\ell\},\{m\}}]=2\lambda^2\nu$.
 \item $[\Ell{\{i,j,h\},\{k\},\{\ell\},\{m\}}]=\lambda\nu^2$, $[\Ell{\{i,j\},\{h,k\},\{\ell\},\{m\}}]=\sigma_{3,2}=2\lambda\nu^2-\lambda^3\nu$.
 \item $[\Ell{\{i\},\{j\},\{h\},\{k\},\{\ell,m\}}]=\sigma_{3,3}=\lambda^2\nu^2-\nu^3$.
\end{itemize}
\end{lemma}

\begin{proof}
The classes of tacnodal loci can be deduced from \Cref{n2}, \Cref{lm:schubert classes}, and pullback from \Cref{n5}:
\[[\Ell{\{1,2,3,4,6\},\{5\}}]=\pi^*[\Ell{\{1,2,3,4\},\{5\}}]-[\Ell{\{1,2,3,4\},\{5,6\}}]=6\lambda^3-2\lambda\nu.\]

Similarly, the classes of elliptic $3$-fold points can be deduced from \Cref{n3}, \Cref{lm:schubert classes}, and pullback from \Cref{n5}, by applying the relation $\lambda^4=\lambda^2\nu+\nu^2$, see \Cref{ChowMin}:
\begin{align*}
 [\Ell{\{i,j\},\{h,k\},\{\ell, m\}}]&=\lambda\cdot(2\lambda\nu-\lambda^3)=\lambda^2\nu-\nu^2;\\
 [\Ell{\{i,j,h\},\{k,\ell\},\{m\}}]&=\lambda\cdot(\lambda^3-\lambda\nu)=\nu^2;\\
 [\Ell{\{i,j,h,m\},\{k\},\{\ell\}}]&=\pi^*[\Ell{\{i,j,h\},\{k\},\{\ell\}}]-2[\Ell{\{i,j,h\},\{k,m\},\{\ell\}}]=2\lambda^4-2\nu^2=2\lambda^2\nu.
\end{align*}

Next, observe that $\Ell{\{1,2\},\{3\},\{4\},\{5\}}$ is a closed point in $\Min{5}\smallsetminus \{[C_{1,5}]\} \simeq \PP^5$. Its preimage along the rational morphism $\on{Gr}(2,5)\dashrightarrow \PP^5$  is the curve with a quadruple point and having its components marked respectively by $\{p_1,p_2\},\{p_3\},\{p_4\}$ and $\{p_5\}$. In particular, the components with only one marking are necessarily lines in $\on{Gr}(2,5)$, hence their class is equal to $\sigma_{3,2}$. In other words, we have proved that $[\Ell{\{i,j\},\{h,k\},\{l\},\{m\}}]=\sigma_{3,2}=\sigma_3\sigma_2=2\lambda\nu^2-\lambda^3\nu$. As $\pi^*[\Ell{\{i,j\},\{h\},\{k\},\{\ell\}}]=\lambda^5=3\lambda^3\nu-2\lambda\nu^2$, we deduce that $[\Ell{\{i,j,h\},\{k\},\{\ell\},\{m\}}]=3\lambda^3\nu-2\lambda\nu^2 - 3(2\lambda\nu^2-\lambda^3\nu)=6\lambda^3\nu-8\lambda\nu^2=\lambda\nu^2$, where we have used the relation $2\lambda^3\nu-3\lambda\nu^2=\lambda^5-\lambda\cdot\lambda^4=0$.

Finally, the class of $\Ell{\{i,j\},\{h\},\{k\},\{\ell\},\{m\}}$ is the class of a point of $\on{Gr}(2,5)$, hence it is equal to $\sigma_{3,3}=\sigma_3^2=(2\lambda\nu-\lambda^3)^2=
\lambda^6-4\lambda^4\nu+4\lambda^2\nu^2=\lambda^2\nu^2-\nu^3$.
\end{proof}

The following summarize the classes of binodal and singularity loci computed in this section.
\begin{proposition}\label{prop:binodal_classes}
The following equalities hold in $A^2(\Min{n}^{sm})$:
    \begin{itemize}
        \item[(n=2)] $[\Nod_{\{1\},\{2\}}]=12\lambda^2$.
        \item[(n=3)] $[\Nod_{\{i,j\},\{k\}}]=6\lambda^2$.
        \item[(n=4)] $[\Nod_{\{i,j,h\},\{k\}}]=4\lambda^2$; 
 
 \noindent $[\Nod_{\{i,j\},\{h,k\}}]=2\lambda^2$.
 \item[(n=5)] $[\Nod_{\{i,j,h,k\},\{\ell\}}]=3\lambda^2$;
  
  \noindent $[\Nod_{\{i,j,h\},\{k,\ell\}}]=\lambda^2$.
  \end{itemize}
  \end{proposition}
\begin{proposition}\label{prop:sing_classes}
The following equalities hold in $A^*(\Min{n}^{sm})$:
    \begin{itemize}
        \item[(n=1)] cusps $[\Ell{\{1\}}] = 24\lambda^2$.
        \item[(n=2)] cusps $[\Ell{\{1,2\}}]=24\lambda^2$;
        
        \noindent tacnodes $[\Ell{\{1\},\{2\}}]=24\lambda^3$.
        \item[(n=3)] cusps $[\Ell{\{1,2,3\}}]=24\lambda^2$;
        
        \noindent tacnodes $[\Ell{\{i,j\},\{h\}}]=12\lambda^3$;
        
        \noindent $3$-fold points $[\Ell{\{1\},\{2\},\{3\}}]=12\lambda^4$.
        \item[(n=4)] cusps $[\Ell{\{1,2,3,4\}}]=24\lambda^2$;
        
        \noindent tacnodes $[\Ell{\{i,j,h\},\{k\}}]=8\lambda^3$, $[\Ell{\{i,j\},\{h,k\}}]=4\lambda^3$;
        
        \noindent $3$-fold points $[\Ell{\{i,j\},\{h\},\{k\}}]=4\lambda^4$;
        
        \noindent  $4$-fold points $[\Ell{\{1\},\{2\},\{3\},\{4\}}]=4\lambda^5$.
        \item[(n=5)] cusps $[\Ell{\{1,2,3,4,5\}}]=24\lambda^2$;
        
        \noindent tacnodes $[\Ell{\{i,j,h,k\},\{\ell\}}]=6\lambda^3$, $[\Ell{\{i,j,h\},\{k,\ell\}}]=2\lambda^3$;
        
        \noindent $3$-fold points $[\Ell{\{i,j,h\},\{k\},\{\ell\}}]=2\lambda^4$, $[\Ell{\{i,j\},\{h,k\},\{\ell\}}]=\lambda^4$;
        
        \noindent $4$-fold points $[\Ell{\{i,j\},\{h\},\{k\},\{\ell\}}]=\lambda^5$;
        
        \noindent $5$-fold points $[\Ell{\{1\},\{2\},\{3\},\{4\},\{5\}}]=\lambda^6$.
        \item[(n=6)] cusps $[\Ell{\{1,2,3,4,5\}}]=24\lambda^2$;
        
        \noindent tacnodes $[\Ell{\{i,j,h,k\},\{\ell,m\}}]=2\lambda\nu$, $[\Ell{\{i,j,h\},\{k,\ell,m\}}]=2\lambda^3-2\lambda\nu$, $[\Ell{\{i,j,h,k,\ell\},\{m\}}]=6\lambda^3-2\lambda\nu$;
        
 \noindent $3$-fold pts. $[\Ell{\{i,j\},\{h,k\},\{\ell, m\}}]=\lambda^2\nu-\nu^2$, $[\Ell{\{i,j,h\},\{k,\ell\},\{m\}}]=\nu^2$, $[\Ell{\{i,j,h,k\},\{\ell\},\{m\}}]=2\lambda^2\nu$.
 
 \noindent $4$-fold points $[\Ell{\{i,j,h\},\{k\},\{\ell\},\{m\}}]=\lambda\nu^2$, $[\Ell{\{i,j\},\{h,k\},\{\ell\},\{m\}}]=2\lambda\nu^2-\lambda^3\nu$.
 
 \noindent $5$-fold points $[\Ell{\{i\},\{j\},\{h\},\{k\},\{\ell,m\}}]=\lambda^2\nu^2-\nu^3$.
    \end{itemize}
\end{proposition}

\section{The integral Chow ring of $\calG_{1,n\leq6}^{sm}$} \label{sec:chow G}

\subsection{Patching}
We employ the patching lemma \cite[Lemma 3.4]{dilorenzo2021polarizedtwistedconicsmoduli}, which we recall below.
\begin{lemma}\label{lm:patch}
    Let $G$ be an algebraic group, and let $X$ be a $G$-equivariant smooth scheme with a $G$-invariant, smooth closed subscheme $\iota\colon Z\hookrightarrow X$. Set $j\colon U=X\smallsetminus Z \hookrightarrow X$ and let $N$ be the normal bundle of $Z$. Suppose that the equivariant top Chern class $c_{\rm{top}}^G(N)$ is not a zero-divisor in $A^*_G(Z)$. Then we have a cartesian diagram of rings
    \[\begin{tikzcd}
        A^*_G(X) \ar[r, "j^*"] \ar[d, "\iota^*"] \ar[dr,phantom,"\ulcorner"] & A^*_G(U) \ar[d] \\
        A^*_G(Z) \ar[r] & A^*_G(Z)/(c_{\rm{top}}^G(N)).
    \end{tikzcd}\]
\end{lemma}
The vertical arrow on the right takes an element $\xi\in A^*_G(U)$, lifts $\xi$ to $A^*_G(X)$ \emph{in any way}, and then takes restriction to $Z$ followed by projection to the quotient; it is well-defined since any two lifts will differ by some $\iota_*\gamma$, and $\iota^*\iota_*\gamma$ is divisible by $c_{\rm{top}}^G(N)$.

We will always find ourselves in the following favourable situation.
\begin{lemma}\label{lm:patch2}
    In the setup of \Cref{lm:patch}, suppose that the pullback homomorphism \emph{$\iota^*$ is surjective}. Write:
    \[ A^*_G(U)=\ZZ[\eta_1,\ldots,\eta_r]/(f_1,\ldots,f_s). \]
    Then, the equivariant Chow ring of $X$ admits the following presentation:
    \[ A^*_G(X)=\ZZ[\eta'_1,\ldots,\eta'_r,\zeta]/(h_1,\ldots,h_s)+\zeta\cdot\ker(\iota^*) \] where $\eta'_i$ is any lift of $\eta_i$ to $A^*_G(X)$, the cycle $\zeta$ is the fundamental class of $Z$, and $h_i$ are relations that lift the relations $f_i$, i.e. $h_i(\eta_1,\ldots,\eta_r,0)=f_i$ and $\iota^*(h_i)=0$.
\end{lemma}

\begin{proof}
    It follows directly from excision that $A^*_G(X)$ is generated by any lift of the generators of $A^*(U)$ together with all the cycles of the form $\iota_*\gamma$ for $\gamma\in A^*_G(Z)$. As $\iota^*$ is surjective, we can write any such $\gamma$ as $\iota^*\beta$, hence $\iota_*\gamma=\iota_*\iota^*\beta=\zeta\cdot\beta$. From this we see that $\eta_i'$ and $\zeta$ generate.

    The relations in $A^*_G(X)$ are given by those polynomials $h(\eta_1',\ldots,\eta_r',\zeta)$ such that (1) $j^*h=0$ in $A_G^*(U)$ and (2) $\iota^*h=0$ in $A^*_G(Z)$. The polynomials $h_i$ satisfy these properties by assumption, while any class of the form $\zeta\cdot g$ with $g\in\ker(\iota^*)$ does because $\zeta$ vanishes on $U$ and $g$ vanishes on $Z$ (and restrictions are ring homomorphisms).  Now let $h=h(\eta_1',\ldots,\eta_r',\zeta)$ be any relation. Property (1) implies that we can write $h=\sum p_i\cdot h_i 
    + \zeta \cdot g$ for some $g=g(\eta_1',\ldots,\eta_r',\zeta)$ lifting a class $\bar g\in A^*_G(Z)$. Restricting to $Z$ we get $0=c_{\rm{top}}^G(N)\cdot\iota^*g$ and, 
    since $c_{\rm{top}}^G(N)$ is not a zero divisor in $A^*_G(Z)$, we deduce that $g$ is in $\ker(\iota^*)$.
\end{proof}

The previous lemma  immediately implies the following.
\begin{lemma}\label{lm:gens}
    Suppose that we have an equivariant stratification $X \supset Z_{0} \supset \cdots \supset Z_{N-1} \supset Z_{N}=\emptyset$, by smooth, closed subschemes. Denote $X\setminus Z_{i}$ by $U_i$. Suppose moreover that for each triple $(U_i,U_{i-1},Z_i)$ the hypotheses of \Cref{lm:patch2} hold true. Then, the Chow ring of $X$ is generated by the lifts of some generators of $A^*(U_0)$ together with the fundamental classes of the strata $[Z_i]_G$.
\end{lemma}

The following lemma and construction explain how to lift relations.
\begin{lemma}\label{lm:lifting}
    In the setup of \Cref{lm:patch2}, let $f=f(\eta_1,\ldots,\eta_r)$ be a relation in $A^*_G(U)$, and set $f'=f(\eta'_1,\ldots,\eta'_r)$ and $g=\iota^*f'$. Then:
    \begin{enumerate}
        \item the class $g$ is divisible by $c_{\rm{top}}^G(N)$.
        \item Let $\bar g'$ be any lift of $\bar g:=g\cdot c_{\rm{top}}^G(N)^{-1}$ to $A^*_G(X)$. Then $h=f' - \zeta \cdot \bar g'$ is a relation that lifts $f$.
    \end{enumerate}
\end{lemma}
\begin{proof}
As $f'$ restricts to zero in $A^*_G(U)$, we deduce that $f'=\iota_*\gamma$. This implies that $g=\iota^*f'=c_{\rm{top}}^G(N)\cdot \gamma$. The restriction of $h$ to $U$ is nothing but $f$, while its restriction to $Z$ is 0 by construction.
\end{proof}
\begin{construction}\label{const:lift}
   Suppose that we have an equivariant stratification $X \supset Z_{0} \supset \cdots \supset Z_{N-1} \supset Z_{N}=\emptyset$, by smooth, closed subschemes. Denote $X\setminus Z_{i}$ by $U_i$ and $[Z_i]$ by $\zeta_i$. Suppose moreover that for each triple $(U_i,U_{i-1},Z_i)$ the hypotheses of \Cref{lm:patch2} hold true. Then, given a relation $f_i$ in the equivariant Chow ring of $U_i$, we can lift it to a relation $f_N$ in the equivariant Chow ring of $X$ by iterating the following procedure, for $j=i,\ldots,N$:
    \begin{enumerate}
        \item define $f'_j$ in $A^*_G(U_{j+1})$ by rewriting $f_j$ using some lifts of the generators of $A^*_G(U_{j})$,
        \item compute $\bar g_j=\iota^*f'_j\cdot c_{{\rm top}}(N_{Z_{j+1}})^{-1}$ in $A^*(Z_{j+1})$, and lift it to $A^*_G(U_{j+1})$ in any way,
        \item set $f_{j+1}=f'_j - \zeta_i\cdot \bar g'_j$ in $A^*_G(U_{j+1})$.
    \end{enumerate}
\end{construction}
Importantly, a variation on the above construction is useful to compute the fundamental class of an invariant closed subscheme.
\begin{construction}\label{const:class}
    In the setting of \Cref{const:lift}, let $Y\subset X$ be an invariant closed subscheme, and let $\gamma_i$ be the restriction of $[Y]_G$ to $A^*_G(Z_i)$. Let $f_0$ be the restriction of $[Y]_G$ to $U_0$. Then, we can inductively compute $[Y]_G=f_N$ as follows, letting $j=0,\ldots,N-1$:
    \begin{enumerate}
        \item define $f'_j$ in $A^*_G(U_{j+1})$ by rewriting $f_j$ using some lift of the generators of $A^*_G(U_{j+1})$,
        \item compute $\bar g_j=(\iota^*f'_j - \gamma_{j+1})\cdot c_{{\rm top}}(N_{Z_{j}})^{-1}$ and lift it to $\bar g_j'$ in $A^*_G(U_{j+1})$ in any way,
        \item set $f_{j+1}=f'_j - \zeta_i\cdot \bar g'_j$  in $A^*_G(U_{j+1})$.
    \end{enumerate}
\end{construction}

\subsection{Chow rings of strata}
As the $G$-equivariant Chow ring of $X$ is the integral Chow ring of the quotient stack $[X/G]$ \cite{EdidinGrahamIntersection}, we can use the lemmas above to compute the integral Chow ring of quotient stacks.
Therefore, the first ingredient is the following.
\begin{lemma}
    The stack $\Gcal_{1,n}^{sm}$ is a smooth quotient stack, and so is every locally closed substack $\Tail{S}\cap \Gcal_{1,n}^{sm}$.
\end{lemma}
\begin{proof}
    It is enough to prove that $\Gcal_{1,n}^{sm}$ is a quotient stack. Indeed, its smoothness implies the smoothness of the covering scheme. We deduce that $\Tail{S}\cap \Gcal_{1,n}^{sm}$ is also a smooth quotient stack by its identification with $\Min{s_0}^{sm} \times \prod_{i=1}^{k} \oM_{0, 1+s_i}$.

    To prove that $\Gcal_{1,n}$ is a quotient stack, observe that the third power of the log-canonical bundle is very ample. 
    We can thus present $\Gcal_{1,n}$ as the quotient by projectivities of the locus in the Hilbert scheme where the curve is Gorenstein and log-canonically polarised.
\end{proof}
\begin{remark}
 By \Cref{LPtheorem} we know that the total space of the Hodge $\Gm$-bundle $\Hcal$ is a scheme over every stratum. This is enough to prove that $\Hcal$ is an algebraic space.
\end{remark}

\subsubsection{Chow rings of $\oM_{0,n}$}\label{sec:Keel}
Let us recall Keel's presentation of the Chow ring of $\oM_{0,S\cup \star}$ \cite[Theorem 1]{Keel}, where $S$ is any finite set and $\star$ is an element not in $S$.

Let $T$ be a \emph{proper} subset of $S$ with $|T|\geq 2$, and let $D_T$ be (the class of) the divisor in $\oM_{0,S \cup \star}$ of curves with a node separating the markings indexed by $T$ from those indexed by $T^c$. 

Then $A^*(\oM_{0,S\cup \star})$ is generated by the $D_T$ modulo the ideal generated by the following elements:
\begin{equation} 
    \tag{$K_1(S;i,j,h,k)$}  \sum_{\substack{T\subsetneq S \\i,j \in T \\ h,k\notin T}} D_T +  \sum_{\substack{T\subsetneq S \\h,k \in T \\ i,j\notin T}} D_T - \sum_{\substack{T\subsetneq S \\i,h \in T \\ j,k\notin T}} D_T - \sum_{\substack{T\subsetneq S \\j,k \in T \\i,h\notin T}} D_T\text{ for }i,j,h,k \in S,
\end{equation}
\begin{equation}
    \tag{$K_1(S;i,j,h)$}  \sum_{\substack{T\subsetneq S \\i,j \in T \\ h\notin T}} D_T - \sum_{\substack{T\subsetneq S \\i,h \in T \\j \notin T}} D_T  \text{ for }i,j,h \in S,
\end{equation}
\begin{equation}
    \tag{$K_2(S;T,T')$} D_T\cdot D_{T'} \text{ if } T,T'\subsetneq S \text{ have cardinality at least } 2 \text{ and } \emptyset\neq T\cap T'\subsetneq T,T'. 
\end{equation}
The difference with Keel's presentation \cite[Theorem 1]{Keel} lies in the fact that we only use generators that do not contain $\star$ (then $T\subsetneq S$ implies $S\cup\{\star\}\setminus T$ contains at least two elements), whereas Keel uses $D_T$ for every $T\subset [n+1]$ with $|T|\geq 2$ and $|T^c|\geq 2$, but he has an extra relation $D_T=D_{T^c}$. Similarly, the condition in $K_2(S;T,T')$ is the translation of Keel's condition that neither of $T,T'$, and their complements be contained in one another. We will say that such a pair $(T,T')$ is \emph{incomparable}.

\begin{proposition}\label{prop:Chow_strata}
 Let $S\neq S_{\max}$ be a partition of $[n]$. The Chow ring of $\Tail{S}$ is (generated in degree $1$):
 \[ A^*(\Tail{S})\simeq \ZZ[\lambda_S,\{\Dcal^{S_\alpha}_{T_\alpha}\}]/I_S \]
 where the generators $\Dcal^{S_\alpha}_{T_\alpha}$ are indexed by $T_\alpha\subsetneq S_\alpha$, $|T_\alpha|\geq 2$, and the ideal $I_S$ is generated by the relations:
 \[K_1(S_\alpha;i,j,h),\quad K_1(S_\alpha;i,j,h,k), \quad K_2(S_\alpha;T_\alpha,T'_\alpha)\]
 as in \Cref{sec:Keel},  with the symbols $D_T$ replaced by $\Dcal^{S_\alpha}_{T_\alpha}$.
\end{proposition}
\begin{proof}
 By the identification:
 \[ \Tail{S} \simeq \Min{s_0} \times \prod_{\alpha=1}^{k} \oM_{0,S_\alpha\cup\star_\alpha}, \]
 and since $s_0\leq 5$, \Cref{LPtheorem} allows us to identify $\Tail{S}$ with the $\Gm$-quotient of an affine bundle over $\prod_{\alpha=1}^{k} \oM_{0,S_\alpha\cup\star_\alpha}$. The base satisfies Chow--K\"unneth decomposition by \cite[\S 2, Theorem 2, p. 561]{Keel}, hence its Chow ring is a tensor product of the Chow rings of the factors. By homotopy invariance, the Chow ring of $\Tail{S}$ is identified with that of $B\Gm\times \prod_{\alpha=1}^{k} \oM_{0,S_\alpha\cup\star_\alpha}$, adding to the base ring one free variable $\lambda_S$ which can be identified with the pullback of the $\lambda$-class from $\Min{s_0}$.
\end{proof}

The second ingredient is the top Chern class of the normal bundle.
\begin{lemma}\label{lem:N_stratum}
Let $S$ be a partition of $[n]$ into $s_0$ parts. Let us denote by $\bullet_\alpha$ the core marking to which the rational tail marked by $\star_\alpha\sqcup S_\alpha$ cleaves. Then:
 \[c_{\rm{top}}(N_{\Tail{S}})=\prod_{\alpha=1}^{s_0}(-\psi_{\bullet_\alpha}-\psi_{\star_\alpha})=\prod_{\alpha=1}^{s_0}\left(-\lambda_S-\sum_{\{i_\alpha,j_\alpha\}\subseteq T_\alpha\subsetneq S_\alpha}\Dcal_{T_\alpha} \right),\]
where $i_\alpha,j_\alpha$ are any two elements of $S_\alpha$. In particular, $c_{\rm{top}}(N_{\Tail{S}})$ is not a zero-divisor in $A^*(\Tail{S})$.
\end{lemma}
\begin{proof}
 The first expression follows from standard deformation theory of nodes. 

 The identification of $\psi_{\bullet_{\alpha}}$ with $\lambda_{S}$ (independent of $\alpha$) holds on any minimal curve by \cite[Lemma 1.1.1(3)]{LekiliP}. The expression for $\psi_{\star_{\alpha}}$ can be found in \cite[\S 1.5.2]{Kock}, for instance.
 
 The last statement can be argued for every factor as these are monic in the free variable $\lambda_S$.
\end{proof}

The third ingredient is the surjectivity of pullbacks. Recall the stratification by numerical core level from \Cref{eq:strat G}. Let $U_m=\Gcal_{1,n}^{sm}\setminus\cTail{m}$ denote the locus of curves of numerical core level $m+1$ or higher, so that: 
\begin{equation}                                                                                                                                                                                             \label{open_stratification}
\Min{n}^{sm}=U_{n-1}\subseteq\ldots\subseteq U_0=\Gcal_{1,n}^{sm}.                                                                                                                                                                                                                                            \end{equation}
See Table \ref{table_stratification_of_M16} for a pictorial description of the stratification in case $n=6$.

\newcommand{\whitedisk}[1]{
    if #1>0
        \foreach \n in {1,...,#1}{
            \pgfmathsetmacro\angle{230+60*\n/#1}
            \draw[thick] (0,0) -- ++ (\angle:0.5);
        }
    \draw[thick,fill=white] (0,0) circle (0.3);
    \draw (0,0) node {1};
    }
    
\newcommand{\whitedisknoleg}{
    \draw[thick,fill=white] (0,0) circle (0.3);
    \draw (0,0) node {1};
    }
    
\newcommand{\blackdisk}[3]{
    \filldraw[black] (#1,#2) circle (0.2);
    \draw[thick] (0,0) -- (#1,#2);
    if #3>0
        \foreach \n in {1,...,#3}{
            \pgfmathsetmacro\angle{230+60*\n/#3}
            \draw[thick] (#1,#2) -- ++ (\angle:0.4);
        }
    }

\begin{center}
\renewcommand{\arraystretch}{2}
\begin{table}[hbt]

\begin{tabular}{|c|c|c|c|c|c|}
\hline
$U_5(=\Min{6})$ & $+T_5=U_4$ & $+T_4=U_3$ & $+T_3=U_2$ & $+T_2=U_1$ & $+T_1=U_0(=\Gcal_{1,6})$ \\
\hline
\begin{tikzpicture}
\whitedisk{6}
\end{tikzpicture}
&
\begin{tikzpicture}
\blackdisk{1.4}{0}{2}
\whitedisk{4}
\end{tikzpicture}
&
\begin{tikzpicture}
\blackdisk{1.4}{0}{3}
\whitedisk{3}
\end{tikzpicture}
&
\begin{tikzpicture}
\blackdisk{1.4}{0}{4}
\whitedisk{2}
\end{tikzpicture}
&
\begin{tikzpicture}
\blackdisk{1.4}{0}{5}
\whitedisk{1}
\end{tikzpicture}
&
\begin{tikzpicture}
\blackdisk{1.4}{0}{6}
\whitedisknoleg
\end{tikzpicture}
\\
& & 
\begin{tikzpicture}
\blackdisk{1.4}{.5}{2}
\blackdisk{1.4}{-.5}{2}
\whitedisk{2}
\end{tikzpicture}
&
\begin{tikzpicture}
\blackdisk{1.4}{.5}{3}
\blackdisk{1.4}{-.5}{2}
\whitedisk{1}
\end{tikzpicture}
&
\begin{tikzpicture}
\blackdisk{1}{.5}{4}
\blackdisk{1}{-.5}{2}
\whitedisknoleg
\end{tikzpicture}

\begin{tikzpicture}
\blackdisk{1}{.5}{3}
\blackdisk{1}{-.5}{3}
\whitedisknoleg
\end{tikzpicture}
&
\\
& & &
\begin{tikzpicture}
\blackdisk{1.4}{.7}{2}
\blackdisk{1.4}{0}{2}
\blackdisk{1.4}{-.7}{2}
\whitedisknoleg
\end{tikzpicture}
& &
\\
\hline
\end{tabular}

\caption{Stratification of $\Gcal_{1,6}$ by numerical core level.}
\label{table_stratification_of_M16}
\end{table}
\end{center}

\begin{notation}
We denote by $\lambda$ the first Chern class of the Hodge bundle over $\Gcal_{1,n}$. For every partition $S\neq S_{\max}$ of $[n]$ we denote by $\tau_S$ the fundamental class of $\cTail{S}$, of degree/codimension equal to $s_0$, the number of non-singleton parts of $S$.
Finally, we denote by $\nu\in A^2(\Gcal_{1,6}^{sm})$ the fundamental class of the banana curves locus $\Nod_{\{1,2,3,4\},\{5,6\}}\subseteq \Gcal_{1,6}^{sm}$; by \Cref{lm:schubert classes}, this class extends geometrically the generator $\nu$ in $A^2(\Min{6}^{sm})$.
\end{notation}
\begin{notation}\label{not:disc}
    If $P$ is a partition of a subset $U\subset [n]$, we denote by $\disc{P}$ the partition of $[n]$ obtained from $P$ by adding every element of $[n]\setminus U$ as a singleton. For a subset $B\subseteq[n]$ we shall write $\disc{B}$ for $\disc{\{B\}}$, $\Tail{B}$ for $\Tail{\disc{B}}$ and $\tau_B$ for $\tau_{\disc{B}}$.
\end{notation}

For instance, if $n=6$ and $T=\{\{1,2\},\{3,5\}\}$, then $\disc{T}=\{\{1,2\},\{3,5\},\{4\},\{6\}\}$. 

\begin{lemma}\label{lem:surjective_pullback}
Let $\iota\colon\Tail{S}\hookrightarrow \Gcal_{1,n\leq 6}^{sm}$ denote the locally closed embedding. Then $\iota^*(\lambda)=\lambda_S$, and
 \[\iota^*(\tau_B)=\begin{cases}
 
 -\lambda_S-\sum_{i,j\in B'\subsetneq B}\Dcal^{S_\alpha}_{B'}, & \text{ if } \exists  \alpha\in\{1,\ldots,k\}: B=S_\alpha \\

 \Dcal^{S_\alpha}_B, & \text{ if } \exists \alpha \in\{1,\ldots,k\}: B\subsetneq S_\alpha \\
 
 0 & \text{ otherwise.}
                    \end{cases}
\]
Finally, we have $\iota^*(\nu)=[\Nod_{\{1,2,3,4\},\{5,6\}}\cap\Tail{S}]$. In particular, $\iota^*$ is surjective.
\end{lemma}
\begin{proof}
The claimed equalities follow from the deformation theory of nodes, cf. \Cref{lem:N_stratum} as well. As for the surjectivity, if $S$ has numerical core level $m$, then $\Tail{S}\subseteq U_{m-1}\setminus U_m$. If $B\subseteq S_\alpha$ for some part of $S$, then $\disc{B}$ is a refinement of $S$, and in particular it has a higher numerical core level than the latter, hence $\Tail{B}\subseteq U_{m}$ already.
\end{proof}

\subsection{The Chow ring of $\Gcal_{1,n\leq 6}^{sm}$}

The previous lemmas and \Cref{lm:gens} yield the following.
\begin{proposition}
Let $n\leq 5$. The integral Chow ring of $\Gcal_{1,n}$ is generated by the first Chern class $\lambda$ of the Hodge bundle, and by the classes of the boundary 
strata $\tau_S$, for $S\neq S_{\max}$ a partition of $[n]$. 

The integral Chow ring of $\Gcal_{1,6}^{sm}$ is generated by $\lambda$, the boundary classes $\tau_S$, and a class $\nu\in A^2(\Gcal_{1,6}^{sm})$.
\end{proposition}

We are left with computing the relations among these classes. 
\begin{notation}\label{not:incomparable}
 We say that two proper subsets $B,B'\subsetneq [n]$ of cardinality at least $2$ are \emph{incomparable} ($B \not\sim B'$) if $\emptyset\neq B\cap B'\subsetneq B,B'$. 
\end{notation}

\begin{definition}\label{def:rels 1}
    For $n\leq 6$, $B\subseteq [n]$ such that $2\leq |B|\leq n$, and $i,j,h,k\in B$, define the following:
    \begin{align*}
        &K_1(B;i,j,h)=\tau_B ( \sum_{\substack{i,j\in B' \\ h\notin B'}} \tau_{B'} - \sum_{\substack{i,h\in B'' \\ j\notin B''}} \tau_{B''});\\
        &K_1(B;i,j,h,k)=\tau_B( \sum_{\substack{i,j\in B' \\ h,k \notin B'}} \tau_{B'} + \sum_{\substack{h,k\in B'' \\ i,j \notin B''}} \tau_{B''}  - \sum_{\substack{i,h\in B''' \\ j,k\notin B'''}} \tau_{B'''}  - \sum_{\substack{j,k\in B'''' \\ i,h\notin B''''}} \tau_{B''''} ) ;\\
        &K_2(B,B')=\tau_{B}\cdot\tau_{B'}\text{ if }B\not\sim B'; \\
        &K_2(B=:B_1,\ldots, B_k)=\tau_{\Sigma} - \prod_{\alpha=1}^{k} \tau_{B_\alpha}\text{for }\Sigma=\disc{\{B_1,\ldots,B_k\}}\text{, if the } B_\alpha \text{ are pairwise disjoint}; \\
        &N(B)=\tau_B(\lambda + \sum_{i,j \in B'} \tau_{B'})\text{ for any choice of }i,j\in B.
    \end{align*}
    Note that, as a consequence of the relations $K_2(B,-)$, in relations $K_1$ and $N$ it is equivalent to take the sum over subsets $B'\subseteq B$ or $B'\subseteq[n]$.
\end{definition}
\begin{lemma}\label{lm:rels 1 hold}
    All the polynomials from \Cref{def:rels 1} restrict to zero in $A^*(\Tail{S})$, for every partition $S$.
\end{lemma}
\begin{proof}
    Pick $B\subset [n]$ such that $|B|\geq 2$: if $\disc{B}\not\succ S$ (in particular, $\Sigma\not\succ S$ either), then all the polynomials above 
    restrict to zero, since $\tau_B|_{\Tail{S}}=0$ thanks to \Cref{lem:surjective_pullback}.

    Up to permutation, we can therefore assume that $B\subseteq S_1\in S$. In this case, both $K_1(B;i,j,h)$ and $K_1(B;i,j,h,k)$ restrict to multiples of the analogous Keel relations (see \Cref{prop:Chow_strata}) - note that the assumption $i,j\in B'$ and $h\notin B'$ forces $\emptyset\subsetneq B'\subsetneq B$ - hence they are zero. The same holds for $K_2(B,B')$, as we can assume that $B,B'\subsetneq S_1$.
    The relation $N(B)$ restricts to:
    \[\Dcal_B \cdot (\lambda_S + \sum_{\substack{i,j\in B' \\ B'\subsetneq S_1}} \Dcal_{B'} + \tau_{S_1}|_{\Tail{S}}) = \Dcal_B\cdot (\lambda_S + \sum_{\substack{i,j\in B' \\ B'\subsetneq S_1}} \Dcal_{B'} - \lambda_S - (\sum_{\substack{i,j\in B' \\ B'\subsetneq S_1}} \Dcal_{B'}))=0.\]
   
   We are only left with proving that $K_2(B_1,\ldots,B_k)$ is zero. This is easier to prove in $A^*(\Gcal_{1,n}^{sm})$ directly: since $\cTail{\disc{\{B_1,\ldots,B_k\}}}=
   \cTail{B_1}\cap\ldots\cap\cTail{B_k}$ is a complete intersection, the relation follows.
\end{proof}
For $n=6$, we need more relations: on one hand, the Chow ring of the open stratum $\Min{6}^{sm}\simeq\on{Gr}(2,5)$ is not free in $\lambda$ and $\nu$. Consider the two generators of the ideal of relations of $A^*(\on{Gr}(2,5))$, see \Cref{ChowMin}:
\[
B^{(0)}=\lambda^4-\lambda^2\nu-\nu^2, \quad C^{(0)}=\lambda^5-3\lambda^3\nu+2\lambda\nu^2.
\]
\begin{definition}\label{def:rels 2}
We define the polynomial $B$ (respectively $C$) as the polynomial obtained by applying \Cref{const:lift} to $B^{(0)}$ (respectively $C^{(0)}$).
\end{definition}
On the other hand, by \Cref{lm:schubert classes}, the generator $\nu\in A^2(\on{Gr}(2,5))$ extends naturally to the class $[\cNod_{\{1,2,3,4\},\{5,6\}}]\in A^2(\Gcal_{1,6}^{sm})$, but the restriction of the latter to any of the closed strata is expressible in terms of the relevant $\lambda$-class, see \Cref{prop:binodal_classes}, thus giving rise to more relations, which we now make explicit.
For this we introduce the following.
\begin{definition}\label{def:comp part}
    Let $P$, $S$ be partitions of $[n]$. Suppose that $P\preceq S$: we define the partition $P\circ S$ of $\ell(S)$ into $\ell(P)$ parts such that $(P\circ S)_{\beta}=\{S_\alpha\text{ such that }S_\alpha\subseteq P_\beta \}$.
\end{definition}
\begin{lemma}\label{lm:intersection strata}
    Let $P$, $S$ be partitions of $[n]$. Then 
    \[\cNod_{P}\cap\oTail{S}=\begin{cases}
       \Nod_{P\circ S} \times \prod_{|S_i|>1} \oM_{0,S_i\cup \star_{P_i}} & \text{if } P\preceq S, \\
       \emptyset & \text{otherwise,}
      \end{cases}\]
       where $\Nod_{P\circ S}$ is regarded as a substack of $\Min{\ell(S)}$.
\end{lemma}

\begin{figure}
\begin{tikzpicture}[>=stealth, node distance=2cm]

\node (L) at (0,0) {};
\node (A) at (1.5,0) {};
\node (B) at (3.5,0) {};
\node (R) at (5,0) {};

\draw[-] (L) -- ++(-0.8,0.4) node[left]{1};
\draw[-] (L) -- ++(-0.8,0.0) node[left]{2};
\draw[-] (L) -- ++(-0.8,-0.4) node[left]{3};

\draw[-] (A) -- ++(-0.5,-0.7) node[left]{4};

\draw (A) .. controls +(0.8,1) and +(-0.8,1) .. (B)
      -- (B) .. controls +(-0.8,-1) and +(0.8,-1) .. (A);

\draw (L) -- (A);
\draw (B) -- (R);

\draw[-] (R) -- ++(0.8,0.4) node[right]{5};
\draw[-] (R) -- ++(0.8,-0.4) node[right]{6};

\draw[fill=black] (L) circle(3pt);
\draw[fill=black] (A) circle(3pt);
\draw[fill=black] (B) circle(3pt);
\draw[fill=black] (R) circle(3pt);
\end{tikzpicture}

\caption{Dual graph of a typical member of $\cNod_{P}\cap\oTail{S}$ for $P=\{\{1,2,3,4\},\{5,6\}\}$ and $S=\{\{1,2,3\},\{4\},\{5,6\}\}$. Here $P\circ S=\{\{\star_{123},\star_4\},\{\star_{56}\}\}.$}
\end{figure}

Set $P=\{\{1,2,3,4\},\{5,6\}\}$. For any partition $S$ that refines $P$, set $\gamma_S=\gamma_S(\lambda_S):=[\Nod_{P\circ S}]\in A^2(\Tail{S})$ be obtained by evaluating the expressions for the binodal loci from \Cref{prop:binodal_classes} in $\lambda=\lambda_S$. Let $s_0$ denote the numerical core level of $S$. The class $\nu-\gamma_S$ restricts to $0$ on $\Tail{S}$, hence \[A^{(s_0)}(S):=\tau_S(\nu-\gamma_S)\] is a relation on the open $U_{s_0-1}$.

\begin{definition}\label{def:rels 3}
 Define the polynomial $A(S)$ by applying \Cref{const:lift} to $A^{(s_0)}(S)$.
\end{definition}
By construction, we have the following.
\begin{lemma}\label{lm:rels 2 hold}
    The polynomials $A(S)$, $B$ and $C$ define relations in $A^*(\Gcal_{1,6}^{sm})$.
\end{lemma}
We are ready to state the main theorem of this section.
\begin{theorem}\label{thm:chow G}
    The integral Chow ring of $\Gcal_{1,n\leq6}^{sm}$ is generated by the class $\lambda$, the boundary classes $\tau_S$ for $S\vdash[n]$, and, when $n=6$, by the codimension $2$ class $\nu$. 
    
    The ideal of relations is generated by
    the quadrics: 
    \[ K_1(B;i,j,h), K_1(B;i,j,h,k), K_2(B,B'), K_2(B_1,\ldots,B_k),
    N(B)\] given in \Cref{def:rels 1}, and, for $n=6$, by the polynomials $(A(S))_{\{\{1,2,3,4\},\{5,6\}\}\preceq S\neq [6]_{\max}}$, $B$ and $C$ of Definitions \ref{def:rels 2} and \ref{def:rels 3}. 
\end{theorem}
\begin{remark}
   The relations $K_2(B_1,\ldots,B_k)$ express any $\tau_S$ as a product of boundary divisor classes $\tau_{B_\alpha}$, hence this smaller subset, together with $\lambda$ and $\nu$, suffices to generate the Chow ring.
\end{remark}

\begin{remark}
 Patching shows that the relation $[\Ell{S}]=c(S)\lambda\cdot[\Nod_S]$, where $c(S)\in\NN$ is a constant depending only on $\ell(S)$, holds in $A^*(\Gcal_{1,n}^{sm})$ too, at least in the range $\ell(S)=2,3,4$, see \S\ref{sec:classes}.
\end{remark}

\begin{proof}
    For any of the polynomials $f$ given in \Cref{def:rels 1}, \Cref{def:rels 2} and \Cref{def:rels 3}, denote by $f^{(m)}$ its restriction to $U_m$. We claim that 
    \[A^*(U_m) = \ZZ[\lambda,\nu, \{\tau_S\}_{\ell(S)\geq m+1}]/I^{(m)},\]
    where $I^{(m)}$ is generated by the $f^{(m)}$. 
    We prove the claim by descending induction on $m$.

    The base case $m=n-1$ follows from \Cref{ChowMin}. Consider now $\iota\colon\Tail{m}\hookrightarrow U_m$, with open complement $U_{m+1}$.
    By \Cref{lm:patch2}, we obtain:
    \[ A^*(U_m) = A^*(U_{m+1})[\{\tau_S\}_{\ell(S)=m+1}]/({I^{(m+1)}}',\tau_S\cdot\ker(\iota_S^*)), \]
    where ${I^{(m+1)}}'$ is the ideal generated by polynomials $f^{(m+1)'}$ having the property that $f^{(m+1)'}$ restricts to $f^{(m+1)}$ when we impose $\tau_S=0$ for $\ell(S)=m+1$ and $f^{(m+1)'}$ are zero in $A^*(\Tail{m})$. With a slight abuse of notation, we are going to indicate the generators of $A^*(U_{m+1})$ and their liftings in the same way.

    We have to prove that $({I^{(m+1)}}',\tau_{S}\cdot\ker(\iota_S^*))=I^{(m)}$.
    \begin{lemma}\label{lm:step 1}
        For $n\leq 5$, the ideal ${I^{(m+1)}}'$ is generated by: \[ K_1(B;{i,j,h})^{(m)}\text{, }K_1(B;{i,j,h,k})^{(m)}\text{, }K_2^{(m)}(B,B')\text{, }K_2^{(m)}(B_1,\ldots,B_k)\text{, }N^{(m)}(B)\]
        with $B\subset[n]$ of cardinality $2\leq\lvert B\rvert<n-m$. For $n=6$ we need the additional generators:
    \[ A^{(m)}(S)\text{ for }\ell(S)>m+1, B^{(m)},C^{(m-1)}. \]
    \end{lemma}

    \begin{proof}
    This follows by observing that we can choose $f^{(m+1)'}$ to be $f^{(m)}$: indeed the latter restricts to $f^{(m+1)}$ when we impose $\tau_S=0$ for $\ell(S)=m+1$, so we only need to prove that the polynomials $f^{(m)}$ are actual relations. But this was already shown in \Cref{lm:rels 1 hold} and \Cref{lm:rels 2 hold}.
    \end{proof}
    
    The kernel ideal $\ker(\iota_S^*)$ can be computed using \Cref{lem:surjective_pullback} and the presentation of $A^*(\Tail{S})$ given in \Cref{prop:Chow_strata}.  Write $S=\disc{\{B_1,\ldots,B_k\}}$. Generators of $\ker(\iota_S^*)$ are:
\begin{itemize}
    \item[(1)] $\tau_B$ for $B$ such that $B\not\subseteq B_\alpha$ for any $\alpha$;
    \item[(2)] $\tau_{B_\alpha}+\lambda+\sum_{i,j\in B'\subsetneq B_\alpha}\tau_{B'}$ for every $\alpha$;
    \item[(3)] $\sum_{\substack{i,j \in B' \\ h\notin B'}}\tau_{B'} - \sum_{\substack{i,h \in B' \\ j \notin B'}} \tau_{B'} $, where the sum runs over those $B'\subset B_\alpha$ for a fixed $\alpha$;
    \item[(4)] $\sum_{\substack{i,j \in B' \\ h,k\notin B'}}\tau_{B'} + \sum_{\substack{h,k \in B' \\i,j\notin B'}}\tau_{B'} - \sum_{\substack{i,h \in B' \\ j,k \notin B'}} \tau_{B'} - \sum_{\substack{j,k \in B' \\ i,h \notin B'}} \tau_{B'} $, sum over $B'\subset B_\alpha$ for a fixed $\alpha$;
    \item [(5)] $\tau_{B}\tau_{B'}$ for every $B,B'\subsetneq B_\alpha$ (for some $\alpha$) that are incomparable as subsets of $B_\alpha$.
\end{itemize}
Furthermore, if $n=6$, we also have
\begin{itemize}
    \item[(6)] $\nu-\gamma_S(\lambda)$, see \Cref{def:rels 3}.
\end{itemize}
To conclude our proof, we only need the following.
\begin{lemma}\label{lm:step 2}
 When $S=\disc{\{B\}}$, the ideal $\tau_S\cdot\ker(\iota_S^*)$ is generated by the relations $f^{(m)}(B_1,\ldots, B_k)$ given in \Cref{def:rels 1}, \Cref{def:rels 2} and \Cref{def:rels 3}, where one of the $B_\alpha$ is equal to $B$.

    When $S=\disc{\{B_1,\ldots, B_k\}}$, the ideal $({I^{(m+1)}}',\tau_S\cdot\ker(\iota^*))$ is generated by the relations computed in \Cref{lm:step 1} together with $K_2(B_1,\ldots,B_k)$ and $A^{(m)}(S)$.
\end{lemma}
\begin{proof}
    In the following, the numbering refers to the list appearing right before the statement of the lemma.
    We start by considering the case $S=\disc{\{B\}}$.
    We have to identify the polynomials $\tau_S\cdot (i)$ for $i=1,\ldots, 6$ with the the relations $f^{(m)}(B_1,\ldots, B_k)$ given in \Cref{def:rels 1}, \Cref{def:rels 2} and \Cref{def:rels 3}, where one of the $B_\alpha$ is equal to $B$.
    
    The polynomials $\tau_S\cdot (1)$ are equal to $K_2^{(m)}(B,B')$ for any $B'\not\sim B$.  Similarly, the polynomials $\tau_S\cdot (5)$ are multiples of $K_2^{(m)}(B,B')$.
    The polynomial $\tau_S\cdot (2)$ corresponds to $N^{(m)}(B)$.
    The polynomials $\tau_S\cdot (3)$ and $\tau_S\cdot (4)$ are equal to $K_1^{(m)}(B)_{i,j,h}$ and $K_1^{(m)}(B)_{i,j,h,k}$.
    
    Now we tackle the case $S= \disc{\{B_1,\ldots, B_k\}}$; the arguments are similar, although in this case we have to show that $({I^{(m+1)}}',\tau_S\cdot\ker(\iota^*))$ is generated by the relations computed in \Cref{lm:step 1} together with $K_2^{(m)}(B_1,\ldots,B_k)$ and $A^{(m)}(S)$.
    First, observe that in this case it follows from \Cref{lm:step 1} that $\tau_S-\prod \tau_{B_\alpha}$ is in ${I^{(m+1)}}'$.
    
    The polynomials $\tau_S\cdot (1)$ are equal, up to a multiple of the relation $\tau_S-\prod \tau_{B_\alpha}$, to a multiple of $K_2^{(m)}(B_1',B_2')$ for some $B_1',B_2'$ such that $\ell(\disc{\{B_\alpha'\}})>m+1$. A similar argument shows that $\tau_S\cdot (5)$ is also multiple of a relation $K_2^{(m)}(B'_1,B_2')$.
    The polynomial $\tau_S\cdot (2)$ corresponds to a multiple of $N^{(m)}(B_\alpha)$, again up to a multiple of the relation $\tau_S-\prod \tau_{B_\alpha}$.

    The polynomials $\tau_S\cdot (3)$ and $\tau_S\cdot (4)$ are equal to a multiple of $K_1^{(m)}(B_\alpha)_{i,j,h}$ or $K_1^{(m)}(B_\alpha)_{i,j,h,k}$, again up to a multiple of the relation $\tau_S-\prod \tau_{B_\alpha}$.

    Finally, in both cases the polynomial $\tau_S\cdot (6)$ is equal to $A^{(m)}(S)$.
\end{proof}
\Cref{lm:step 1} and \Cref{lm:step 2} prove that $({I^{(m+1)}}',\tau_{S}\cdot\ker(\iota_S^*))=I^{(m)}$, thus completing the inductive step and the proof of \Cref{thm:chow G}.
\end{proof}
\begin{corollary}\label{cor:chow G is free}
    The Chow ring of $\Gcal_{1,n\leq 5}$ is a free $\ZZ[\lambda]$-module. For $n=6$, the Chow ring of $\Gcal_{1,6}^{sm}$ is a free $\ZZ$-module.
\end{corollary}
\begin{proof}
    Consider the stratification
    \[\widetilde{\Mcal}_{1,n}^{sm}=U_{n-1} \subset U_{n-2} \subset \ldots \subset U_0=\Gcal_{1,n}^{sm}, \]
    and further define $U_n=\emptyset$.
    If $\alpha$ is a non-zero cycle such that $f(\lambda)\cdot \alpha=0$ for some polynomial $f(\lambda)$, then there must exist a maximal $i\leq n-1$ such that $\alpha|_{(U_i \smallsetminus U_{i+1})}\neq 0$. This implies that the restriction of $\alpha$ to a stratum
    \[ \Tail{S} \simeq \widetilde{\Mcal}_{1,s_0} \times \prod_{|S_i|>1} \Mbar_{0,S_i\cup\star_i}\]
    is non-zero. As the Chow ring of the latter is a free $\ZZ[\lambda]$-module, we deduce that $f(\lambda)=0$. For $n=6$ the same argument works, albeit with base ring $\ZZ$ instead of $\ZZ[\lambda]$.
\end{proof}
\begin{remark}
    Recall that the Chow ring of $\Mbar_{0,s}$ is a free $\ZZ$-module, generated by the closed subschemes $D_\Gamma$ that are the closure of the locally closed strata parametrizing curves whose dual graph is equal to $\Gamma$.
    
    Given a partition $S$ and a set of graphs $\Gamma=\{\Gamma_i\}$ for every $i$ such that $S_i\subset S$ a subset of cardinality $>1$, we can consider the closed substack $\overline{\Tail{S,\Gamma}}$ in $\Gcal_{1,n}^{sm}$, given by the curves having core level $S$, and such that if $R_i$ is a rational tail marked by $S_i$, then $[R_i]$ belongs to $D_{\Gamma_i}$ in $\Mbar_{0,S_i \cup \star_i}$.
    It is not hard to see then that $\Gcal_{1,n\leq 5}$ is generated as a $\ZZ[\lambda]$-module by $\tau_{S,\Gamma}=[\overline{\Tail{S,\Gamma}}]$.
\end{remark}

\section{The integral Chow ring of $\Mbar_{1,n\leq 6}(Q)$ for every $Q$} \label{sec:chow M}
In this section, we give a general formula for the integral Chow ring of $\Mbar_{1,n}(Q)$. Recall that 
\[ \oM_{1,n}(Q) = \Gcal_{1,n} \smallsetminus \left((\cup_{S\in Q}\Tail{S}) \cup (\cup_{S\not\in Q} \Ell{S})\right).\]
The main result of the section is the following.
\begin{theorem}\label{thm:chow mbar 1}
    For $n\leq 5$, the integral Chow ring of $\Mbar_{1,n}(Q)$ is generated by $\lambda$ and the boundary divisors $\tau_B$ for $B$ such that $\disc{B}\not\in Q$, modulo the relations
    \begin{align*}
        &\tau_B ( \sum_{\substack{i,j\in B' \\ h\notin B'}} \tau_{B'} - \sum_{\substack{i,h\in B'' \\ j\notin B''}} \tau_{B''}) \text{ for }i,j,h \in B;\\
        &\tau_B( \sum_{\substack{i,j\in B' \\ h,k \notin B'}} \tau_{B'} + \sum_{\substack{h,k\in B'' \\ i,j \notin B''}} \tau_{B''}  - \sum_{\substack{i,h\in B''' \\ j,k\notin B'''}} \tau_{B'''}  - \sum_{\substack{j,k\in B'''' \\ i,h\notin B''''}} \tau_{B''''} ), \text{ for }i,j,h,k \in B;\\
        &\tau_{B_1} \cdot \tau_{B_2}\cdot\cdots\cdot\tau_{B_k}, \text{ if there are }1\leq i,j\leq k\text{ such that }B_i\not\sim B_j\text{ or }\disc{\{B_{1},\ldots,B_{k}\}}\in Q;\\
        &\tau_B(\lambda + \sum_{i,j \in B'} \tau_{B'}), \text{ for any choice of }i,j\in B; \\
        &[{\cEll{S}}] \text{ for every }S\notin Q\text{, (explicit expressions can be found in \Cref{app:poly}).}
    \end{align*}
    For $n=6$, the integral Chow ring of $\Mbar_{1,n}(Q)$ is generated by $\lambda$, $\nu$ and the boundary divisors $\tau_B$ for $B$ such that $\disc{B}\notin Q$, modulo the same relations above plus the polynomials
    \[ A(S)\text{ for every partition }S\notin Q,\quad B,\quad C \]
    given in \Cref{def:rels 2} and \Cref{def:rels 3}.   
\end{theorem}
To compute the integral Chow ring of $\oM_{1,n}(Q)$, we make use of the following technical result.
\begin{lemma}\label{lm:strat}
    Let $Y$ be a $G$-invariant closed subscheme of a smooth $G$-scheme $X$, for some algebraic group $G$. Suppose that $\emptyset=Y_{-1}\subset Y_0\subset\ldots\subset Y_n=Y$ is a $G$-invariant stratification by closed subschemes such that $Y_i\smallsetminus Y_{i-1}$ is irreducible, and the pullback homomorphisms $A_G^*(X) \to A_G^*(Y_i\smallsetminus Y_{i-1})$ are surjective for every $i$.
    Then the image of $A_G^*(Y)\to A_G^*(X)$ is generated as an ideal by the set $\{[\overline{Y_i}]_G\}_{i=0,\ldots,n}$.
\end{lemma}
\begin{proof}
    We argue by induction on the length of the stratification. If $n=0$, the statement follows from the projection formula. 
    
    For the inductive step, let us denote by $\iota:Y\hookrightarrow X$ the closed embedding, and $j\colon Y\smallsetminus Y_{n-1} \hookrightarrow Y$ the open embedding. For $\xi$ a cycle on $Y$, by hypothesis there exists a cycle $\zeta$ on $X$ such that $j^*\xi=j^*\iota^*\zeta$. From the localization exact sequence associated to the closed embedding $g\colon Y_{n-1} \hookrightarrow Y$ we deduce that $\xi=\iota^*\zeta + g_*\eta$, where $\eta$ is some cycle supported on $Y_{n-1}$, hence $\iota_*\xi=\iota_*\iota^*\zeta + \iota_*g_*\eta=\zeta\cdot [Y] + (\iota\circ g)_*\eta$. This implies that $\on{im}(\iota_*)=([Y],\on{im}(\iota\circ g)_*)$. We conclude by the inductive assumption on the closed embedding $\iota\circ g:Y_{n-1}\hookrightarrow Y$.
\end{proof}
We aim at applying the lemma above to the complement of $\oM_{1,n}(Q)$ in $\Gcal_{1,n}^{sm}$. Surjectivity of $A^*(\Gcal_{1,n}^{sm}) \to A^*(\Tail{S})$ has already been established in \Cref{lem:surjective_pullback}, so we only have to prove the surjectivity of the other pullback homomorphisms.
\begin{lemma}\label{lm:pullback to ell sur}
    The pullback homomorphism $A^*(\Gcal_{1,n}^{sm})\to A^*(\Ell{S})$ is surjective for every $S$.
\end{lemma}
\begin{proof}
Recall from \Cref{Ells} that $\Ell{S}$ admits an explicit description as the $\Gm$-quotient of a vector bundle over a product of moduli space of stable rational curves:
\[\oEll{S}=[\on{Tot}(\bigoplus_{\alpha=1}^k \mathbb L_{0,\star_\alpha}^{\vee})_{\prod_{\alpha=1}^k\oM_{0,\star_\alpha\cup S_\alpha}}/\Gm].\]
By homotopy invariance and Keel's results, we obtain the following description of the Chow ring:
\[ A^*(\Ell{S})\simeq \ZZ[\xi_S,\{\Dcal^{S_\alpha}_{T_\alpha}\}]/I_S \]
 where $\xi_S$ comes from the gerbe, and the generators $\Dcal^{S_\alpha}_{T_\alpha}$ are indexed by $T_\alpha\subsetneq S_\alpha$, $|T_\alpha|\geq 2$. (The ideal $I_S$ is generated by the relations $K_1(S_\alpha;i,j,h),\ K_1(S_\alpha;i,j,h,k), \ K_2(S_\alpha;T_\alpha,T'_\alpha)$ as in \Cref{sec:Keel},  with the symbols $D_T$ replaced by $\Dcal^{S_\alpha}_{T_\alpha}$, but we will not need this information.)
 
 Let $\imath\colon\Ell{S}\hookrightarrow \Gcal_{1,n}^{sm}$ denote the locally closed embedding. Then $\imath^*(\lambda)=-\xi_S$ follows from \cite[Proposition 3.4]{SmythII}, and, for every $T_\alpha\subsetneq S_\alpha$, we have $\imath^*(\tau_B)=
 \Dcal^{S_\alpha}_B$, from which the surjectivity follows.
\end{proof}
\begin{proof}[Proof of \Cref{thm:chow mbar 1}]
As in the proof of \Cref{thm:chow G}, pick a total order on the set of partitions of $[n]$ by first ordering by numerical core level, and then by ordering partitions with the same numerical core level in any way. This induces a 
stratification on the closed subscheme $\cup_{S\in Q} \Tail{S}$ where the $i^{\rm th}$-stratum is $\cup_{j\leq i} \Tail{S_j}$. 

The difference between consecutive strata is isomorphic to $\Tail{S}$ for some $S$. As the pullback homomorphism $A^*(\Gcal_{1,n}^{sm})\to A^*(\Tail{S})$ is surjective by \Cref{lem:surjective_pullback}, we can apply \Cref{lm:strat} to deduce that
\[ \on{Im}(A^*(\cup_{S\in Q} \Tail{S}) \to A^*(\Gcal_{1,n}^{sm})) = (\{\tau_B\}_{\disc{B}{\in Q}}).\]
Similarly, pick a total order on the set of partitions of $[n]$ by first reverse-ordering by numerical singularity level, and then by ordering in any way the partitions sharing the same numerical singularity level. This induces a 
stratification on $\cup_{S\notin Q}\Ell{S}$ where the strata are $\cup_{j\leq i}\Ell{S_j}$.

The difference between consecutive strata is isomorphic to $\Ell{S}$ for some $S$. As the pullback homomorphism $A^*(\Gcal_{1,n}^{sm})\to A^*(\Ell{S})$ is surjective by \Cref{lm:pullback to ell sur}, we conclude by \Cref{lm:strat} that
\[ \on{Im}(A^*(\cup_{S\notin Q} \Ell{S}) \to A^*(\Gcal_{1,n}^{sm})) = (\{[\cEll{S}]\}_{S\notin Q}).\]
The localization exact sequence for Chow groups gives us an exact sequence
\[ A^*(\cup_{S\in Q}\Tail{S})\oplus A^*(\cup_{S\notin Q}\Ell{S}) \longrightarrow A^*(\Gcal_{1,n}^{sm}) \longrightarrow A^*(\oM_{1,n}(Q)) \longrightarrow 0.\]
As we already have a presentation of $A^*(\Gcal_{1,n}^{sm})$ from \Cref{thm:chow G}, we deduce the claimed presentation for the integral Chow ring of $\oM_{1,n}(Q)$.
\end{proof}

\subsection{Computation of fundamental classes}\label{sec:poly}
In order to make the presentation given in \Cref{thm:chow mbar 1} completely explicit, we are left with finding an expression of the fundamental classes $[\cEll{S}]$ in terms of the generators $\lambda$, $\tau_B$, and possibly $\nu$ (when $n=6$).

Fundamental classes can be computed by \Cref{const:class}, once their restriction to the strata $\Tail{P}$ is known. By \Cref{lm:intersection strata} we know that $\Ell{S}\cap\Tail{P}$ is equal to $\Ell{S\circ P}\times \prod_{|P_i|>1} \oM_{0,P_1\cup \star_{P_i}}$ when $P\succeq S$, and zero otherwise. The class of this intersection is therefore a monomial in $\lambda$ and is equal to $[\Ell{S\circ P}]$ in $A^*(\Min{\lvert P\rvert})$, which we explicitly calculated in \Cref{prop:sing_classes}. We therefore have all the data we need to apply \Cref{const:class} and obtain $[\cEll{S}]$. We have implemented the algorithm in Sage: the results for $n\leq5$ are shown in \Cref{app:poly}. 

The next observation can be made in general.
\begin{lemma}\label{lem:cuspidal_class}
The class of the locus of worse-than-nodal curves is \[[\Ell{[n]}]=24\lambda^2\in A^2(\Gcal_{1,n}) \text{ for every } n\geq1.\]
\end{lemma}
\begin{proof}
    This follows from patching and \Cref{lem:cusps}.
\end{proof}
In particular, the formula holds in $A^2(\oM_{1,n}(Q))$ for every $Q\neq\emptyset$, and $24\lambda^2=0$ in $A^2(\oM_{1,n})$.
\subsection{Rational cohomology}
Our computation of the integral Chow ring of $\oM_{1,n}(Q)$ also gives access to the rational cohomology of the coarse moduli spaces.
\begin{proposition}\label{prop:cycle}
    For $n\leq 6$, the moduli stack $\oM_{1,n}(Q)$ satisfies the Chow--K\"unneth generation property. In particular, the cycle class map $\on{cl}\colon A^*(\oM_{1,n}(Q))_{\mathbb Q} \to H^{2*}_{\text{\'{e}t}}(\oM_{1,n}(Q),\mathbb Q_\ell)$ is an isomorphism, and the odd cohomology vanishes.
\end{proposition}
\begin{proof}
    We prove that $\Gcal_{1,n}^{sm}$ has the Chow-K\"{u}nneth generation property (CKgP) \cite[Definition 2.5]{BaeSchmittII}. First observe that $\Min{n}^{sm}$ has the CKgP: if $n\leq 5$, then $\Min{n}\simeq [\Aaff^{n+1}/\Gm]$ is an affine bundle over $B\Gm$, which has the CKgP \cite[Lemma 3.8]{CL}; we deduce that $[\Aaff^{n+1}/\Gm]$ has the CKgP as well \cite[Lemma 3.5]{CL}. If $n=6$, then $\Min{6}^{sm}\simeq \on{Gr}(2,5)$, which has the CKgP \cite[Lemma 3.7]{CL}.

    Consider the stratification of $\Gcal_{1,n}^{sm}$ with strata $\Tail{S}$. If we prove that $\Tail{S}$ has the CKgP, then we are done by \cite[Lemma 3.4]{CL}. As $\Tail{S}\simeq \Min{\lvert S\rvert}\times \prod_{\alpha} \oM_{0,S_\alpha\cup\star_\alpha}$, if each term of the product has the CKgP, then $\Tail{S}$ has it too \cite[Lemma 3.2]{CL}. We already proved that $\Min{n}^{sm}$ has the CKgP for every $m$, and $\oM_{0,d}$ has the CKgP for every $d\geq 3$ \cite[\S 5.1]{CL}.

    As $\oM_{1,n}(Q)$ is an open subset of $\Gcal_{1,n}^{sm}$, we deduce that $\oM_{1,n}(Q)$ has the CKgP \cite[Lemma 3.3]{CL}. As $\oM_{1,n}(Q)$ is a smooth and proper Deligne-Mumford stack having the CKgP, we deduce that the cycle class map is an isomorphism \cite[Lemma 3.11]{CL}.
\end{proof}
\begin{remark}
    In particular, using the explicit presentation of \Cref{thm:chow mbar 1}, we can compute the Hilbert-Poincar\'{e} polynomial of $\oM_{1,n}(Q)$. On the other hand, it follows from \Cref{prop:cycle} that $H^{2i}_{\text{\'{e}t}}(\oM_{1,n}(Q),{\mathbb Q}_{\ell})$ is of pure weight $i$, and that the odd cohomology vanishes. By the Grothendieck-Lefschetz trace formula for stacks \cite{BehrendLefschetz}, the Hilbert-Poincar\'{e} polynomial $h_{\oM_{1,n}(Q)(q)}$ is equal to the point count $|\oM_{1,n}(\mathbb{F}_q)|$. The latter can be computed easily using the stratification by core level, and as a sanity check of our result we verified that the two expressions agree for the stacks $\oM_{1,n}(m)$ of $m$-stable curves.
\end{remark}

\subsection{Additive structure and Getzler's relation}
The following is an easy consequence of our main theorem.
\begin{corollary}
    Let $\alpha$ be a cycle in $A^*(\Mbar_{1,n}(Q))$ such that $r\cdot \alpha=0$ for $r$ a positive integer. Then $r\cdot\alpha$ belongs to the ideal generated by $[\cEll{S}]$ for $S\in Q$.
\end{corollary}
\begin{proof}
    It follows easily from \Cref{cor:chow G is free} that the Chow ring of $\Gcal_{1,n}^{sm}\smallsetminus \cup_{S\in Q} \Tail{S}$ is a free $\ZZ$-module. Let $\alpha'$ be any lifting of $\alpha$ to $A^*(\Gcal_{1,n}^{sm}\smallsetminus \cup_{S\in Q} \Tail{S})$, then $r\cdot \alpha'\neq 0$, hence it belongs to the kernel of the pullback along the open embedding
    \[ \Mbar_{1,n}(Q) \hookrightarrow \Gcal_{1,n}^{sm}\smallsetminus \cup_{S\in Q} \Tail{S}. \] 
    The latter is generated by the fundamental classes $[\cEll{S}]$ for $S\in Q$.
\end{proof}
\begin{corollary}\label{cor:torsion}
    For $n\leq 6$, we have $A^2(\Mbar_{1,n})=\ZZ^{d} \oplus \ZZ/24$. Moreover, for any $Q$ different from the empty set, we see that $A^2(\Mbar_{1,n}(Q))$ is a free $\ZZ$-module.
\end{corollary}
For $n=4$, 
there is an interesting relation in $A^2_{\mathbb Q}(\oM_{1,4})$, known as Getzler's relation \cite[Theorem 1.8]{Get}, see also \cite{Pand}. Setting
\[ \tau_i=\sum_{|B|=i} \tau_B , \qquad \Nod_{2,2}=\Nod_{\{1,2\},\{3,4\}} + \Nod_{\{1,3\},\{2,4\}}+\Nod_{\{1,4\},\{2,3\}},\]
\[ \tau_{2,2} = \tau_{\{1,2\}}\tau_{\{3,4\}} + \tau_{\{1,3\}}\tau_{\{2,4\}}+\tau_{\{1,4\}}\tau_{\{2,3\}},  \]
and $\tau_0=[\Nod_{[4]}]$, the fundamental class of the locus of curves with a non-separating node, Getzler's relation can be formulated as follows:
\begin{equation}\label{eq:Getzler}
    \tau_0\tau_3 - 4 \tau_2 \tau_3 + 12\tau_{2,2} - 2\tau_2\tau_4 + 6\tau_3\tau_4 + \tau_0\tau_4 - 2[\Nod_{2,2}]  = 0.
\end{equation}

It is natural to ask whether this relation holds with integral coefficients as well, or whether it holds on every $\Mbar_{1,n}(Q)$. With this goal in mind we establish the following.
\begin{proposition}
    The following relation holds in $A^2(\Gcal_{1,4})$:
    \begin{equation}\label{eq:Getzler on G}
         [\Nod_{2,2}] = 6\lambda^2 + 6\lambda\tau_3 - 2\tau_2\tau_3 + 6\tau_{2,2} + 6\lambda\tau_4 + 3\tau_3\tau_4 - \tau_2\tau_4. 
     \end{equation}
\end{proposition}
\begin{proof}
    To verify that this relation holds in $A^*(\Gcal_{1,4})$, we verify that its restriction to every stratum $\Tail{S}\simeq \widetilde{\Mcal}_{1,s_0}\times \prod_{|S_i|>1} \Mbar_{0, S_i\cup\star_i}$ vanishes.
    First we have $[\Nod_{2,2}]|_{\widetilde{\Mcal}_{1,4}}=3\cdot 2\lambda^2$ by \Cref{n4}, and the right hand side of \eqref{eq:Getzler on G} restricts to $6\lambda^2$.

    Given any stratum $\Tail{\disc{\{i,j\}}}$, the restriction of $[\Nod_{2,2}]$ is equal to $[\Nod_{\{i,j\},\{h,k\}}]=6\lambda^2$. As every $\tau_i$ for $i\neq 2$ restricts to zero, the restriction of \eqref{eq:Getzler on G} holds.

    Given any stratum $\Tail{\disc{\{i,j,h\}}}$, the restriction of $[\Nod_{2,2}]$ is zero. Observe that $\tau_4$ and $\tau_{2,2}$ restricts to zero, $\tau_2$ restricts to $3\Dcal_{\{i,j\}}$ and $\tau_3$ restrict to $-\lambda-\Dcal_{\{i,j\}}$. It is straightforward to check that the right hand side of \eqref{eq:Getzler on G} restricts to zero.

    Given any stratum $\Tail{\{i,j\},\{h,k\}}$, the restriction of $[\Nod_{2,2}]$ is equal to $12\lambda^2$ by \Cref{n2}. We see that $\tau_{3}$ and $\tau_{4}$ restrict to zero, and $\tau_{2,2}$ restricts to $\lambda^2$. We then check that the right hand side of \eqref{eq:Getzler on G} restricts to $12\lambda^2$.

    Finally, the restriction of $[\Nod_{2,2}]$ to $\Tail{[4]}$ is zero. We also have the following:
    \[ \tau_2 \longmapsto \sum_{i,j\in[4]} \Dcal_{\{i,j\}},\quad \tau_3 \longmapsto \sum_{i,j,h\in[4]} \Dcal_{\{i,j,h\}},\quad \tau_{4} \longmapsto -\lambda - \sum_{s_1,s_2\in S} \Dcal_S\]
    (recall our conventions from \Cref{prop:Chow_strata}), and $\tau_{2,2}$ restricts to zero. A straightforward computation\footnote{After performing the aforementioned substitutions, the linear term in $\lambda$ is \[\lambda\left(\sum_{i,j\in[4]}\Dcal_{\{i,j\}}+3\sum_{i,j,h\in[4]} \Dcal_{\{i,j,h\}}-6\sum_{s_1,s_2\in S}\Dcal_S\right)=\lambda\left(\sum_{i,j\in[4]}\Dcal_{\{i,j\}}+3\sum_{i,j,h\in[4]} \Dcal_{\{i,j,h\}}-\sum_{\{s_1,s_2\}\in{[4]\choose 2}}\sum_{s_1,s_2\in S}\Dcal_S\right)=0.\]} in the Chow ring of $\widetilde{\Mcal}_{1,1}\times \Mbar_{0,5}$ shows that the right hand side of \eqref{eq:Getzler on G} restricts to zero.
\end{proof}
\begin{corollary}
    Let $G$ be Getzler's cycle 
    \eqref{eq:Getzler}. The following relations hold in $A^2(\Gcal_{1,4})$:
    \[ G + 12\lambda^2 = 0, \qquad 2G+[\cEll{[4]}] = 0 \]
    In particular:
    \begin{enumerate}
        \item Getzler's relation does not hold integrally in $A^2(\Mbar_{1,4})$, but $2G$ does;
        \item for any $Q$ different from the empty set, Getzler's relation does not hold in $A^*(\Mbar_{1,4}(Q))$, even with $\mathbb Q$-coefficients.
    \end{enumerate}
\end{corollary}
\begin{proof}
    By substituting \Cref{eq:Getzler on G} in \Cref{eq:Getzler}, and by identifying $\tau_0$ with $12\lambda$ (\Cref{lem:nod}), we obtain the first equation. The second one follows from the identification of the class of the locus of worse-than-nodal curves with $24\lambda^2$ (\Cref{lem:cuspidal_class}).
\end{proof}

\appendix
\section{Fundamental classes of loci of singular curves}\label{app:poly}
 We only write the fundamental classes up to $n=5$, as for $n=6$ the explicit expressions are quite long. Furthermore, the fundamental classes of $\cEll{S}$ for $n\leq 5$ not appearing in the list below can be obtained from the ones in the list by permutation.
 
$n=1$.

\parbox[b]{0.9\textwidth}{\raggedright\hangafter=1\hangindent=2em$\displaystyle [\cEll{\{{1}\}}]= 24 \, {{\lambda}}^{2} $}

$n=2$.

\parbox[b]{0.9\textwidth}{\raggedright\hangafter=1\hangindent=2em$\displaystyle [\cEll{\{{1, 2}\}}]= 24 \, {{\lambda}}^{2} $}

 \parbox[b]{0.9\textwidth}{\raggedright\hangafter=1\hangindent=2em$\displaystyle [\cEll{\{{1}\}\{{2}\}}]= 24 \, {{\lambda}}^{3} + 24 \, {{\lambda}}^{2} {{\tau_{12}}} $}

$n=3$.

\parbox[b]{0.9\textwidth}{\raggedright\hangafter=1\hangindent=2em$\displaystyle [\cEll{\{{1, 2, 3}\}}]= 24 \, {{\lambda}}^{2} $}

 \parbox[b]{0.9\textwidth}{\raggedright\hangafter=1\hangindent=2em$\displaystyle [\cEll{\{{1}\}\{{2, 3}\}}]= 12 \, {{\lambda}}^{3} + 12 \, {{\lambda}}^{2} {{\tau_{12}}} + 12 \, {{\lambda}}^{2} {{\tau_{13}}} - 12 \, {{\lambda}}^{2} {{\tau_{23}}} + 12 \, {{\lambda}}^{2} {{\tau_{123}}} $}

 \parbox[b]{0.9\textwidth}{\raggedright\hangafter=1\hangindent=2em$\displaystyle [\cEll{\{{1}\}\{{2}\}\{{3}\}}]= 12 \, {{\lambda}}^{4} + 12 \, {{\lambda}}^{3} {{\tau_{12}}} + 12 \, {{\lambda}}^{3} {{\tau_{13}}} + 12 \, {{\lambda}}^{3} {{\tau_{23}}} + 12 \, {{\lambda}}^{3} {{\tau_{123}}} + 24 \, {{\lambda}}^{2} {{\tau_{12}}} {{\tau_{123}}} $}
 
$n=4$.

\parbox[b]{0.9\textwidth}{\raggedright\hangafter=1\hangindent=2em$\displaystyle [\cEll{\{{1, 2, 3, 4}\}}]= 24 \, {{\lambda}}^{2} $}

 \parbox[b]{0.9\textwidth}{\raggedright\hangafter=1\hangindent=2em$\displaystyle [\cEll{\{{1}\}\{{2, 3, 4}\}}]= 8 \, {{\lambda}}^{3} + 8 \, {{\lambda}}^{2} {{\tau_{12}}} - 16 \, {{\lambda}}^{2} {{\tau_{234}}} + 8 \, {{\lambda}}^{2} {{\tau_{1234}}} + 4 \, {{\lambda}} {{\tau_{12}}} {{\tau_{1234}}} - 8 \, {{\lambda}} {{\tau_{234}}} {{\tau_{1234}}} + 8 \, {{\lambda}}^{2} {{\tau_{13}}} + 4 \, {{\lambda}} {{\tau_{1234}}} {{\tau_{13}}} - 4 \, {{\lambda}}^{2} {{\tau_{23}}} + 4 \, {{\lambda}} {{\tau_{234}}} {{\tau_{23}}} - 4 \, {{\lambda}} {{\tau_{1234}}} {{\tau_{23}}} + 8 \, {{\lambda}}^{2} {{\tau_{14}}} - 4 \, {{\lambda}} {{\tau_{23}}} {{\tau_{14}}} - 4 \, {{\lambda}}^{2} {{\tau_{24}}} - 4 \, {{\lambda}} {{\tau_{13}}} {{\tau_{24}}} - 4 \, {{\lambda}}^{2} {{\tau_{34}}} - 4 \, {{\lambda}} {{\tau_{12}}} {{\tau_{34}}} + 8 \, {{\lambda}}^{2} {{\tau_{123}}} + 4 \, {{\lambda}} {{\tau_{12}}} {{\tau_{123}}} + 8 \, {{\lambda}}^{2} {{\tau_{124}}} + 4 \, {{\lambda}} {{\tau_{12}}} {{\tau_{124}}} + 8 \, {{\lambda}}^{2} {{\tau_{134}}} + 4 \, {{\lambda}} {{\tau_{13}}} {{\tau_{134}}} $}

 \parbox[b]{0.9\textwidth}{\raggedright\hangafter=1\hangindent=2em$\displaystyle [\cEll{\{{1, 2}\}\{{3, 4}\}}]= 4 \, {{\lambda}}^{3} - 8 \, {{\lambda}}^{2} {{\tau_{12}}} + 4 \, {{\lambda}}^{2} {{\tau_{234}}} + 4 \, {{\lambda}}^{2} {{\tau_{1234}}} - 4 \, {{\lambda}} {{\tau_{12}}} {{\tau_{1234}}} + 8 \, {{\lambda}} {{\tau_{234}}} {{\tau_{1234}}} + 4 \, {{\lambda}}^{2} {{\tau_{13}}} - 4 \, {{\lambda}} {{\tau_{1234}}} {{\tau_{13}}} + 4 \, {{\lambda}}^{2} {{\tau_{23}}} - 4 \, {{\lambda}} {{\tau_{234}}} {{\tau_{23}}} + 4 \, {{\lambda}} {{\tau_{1234}}} {{\tau_{23}}} + 4 \, {{\lambda}}^{2} {{\tau_{14}}} + 4 \, {{\lambda}} {{\tau_{23}}} {{\tau_{14}}} + 4 \, {{\lambda}}^{2} {{\tau_{24}}} + 4 \, {{\lambda}} {{\tau_{13}}} {{\tau_{24}}} - 8 \, {{\lambda}}^{2} {{\tau_{34}}} + 4 \, {{\lambda}} {{\tau_{12}}} {{\tau_{34}}} + 4 \, {{\lambda}}^{2} {{\tau_{123}}} - 4 \, {{\lambda}} {{\tau_{12}}} {{\tau_{123}}} + 4 \, {{\lambda}}^{2} {{\tau_{124}}} - 4 \, {{\lambda}} {{\tau_{12}}} {{\tau_{124}}} + 4 \, {{\lambda}}^{2} {{\tau_{134}}} - 4 \, {{\lambda}} {{\tau_{13}}} {{\tau_{134}}} $}

 \parbox[b]{0.9\textwidth}{\raggedright\hangafter=1\hangindent=2em$\displaystyle [\cEll{\{{1}\}\{{2, 4}\}\{{3}\}}]= 4 \, {{\lambda}}^{4} + 4 \, {{\lambda}}^{3} {{\tau_{12}}} + 4 \, {{\lambda}}^{3} {{\tau_{234}}} + 4 \, {{\lambda}}^{3} {{\tau_{1234}}} - 4 \, {{\lambda}}^{2} {{\tau_{12}}} {{\tau_{1234}}} + 8 \, {{\lambda}}^{2} {{\tau_{234}}} {{\tau_{1234}}} + 4 \, {{\lambda}}^{3} {{\tau_{13}}} + 8 \, {{\lambda}}^{2} {{\tau_{1234}}} {{\tau_{13}}} + 4 \, {{\lambda}}^{3} {{\tau_{23}}} - 4 \, {{\lambda}}^{2} {{\tau_{234}}} {{\tau_{23}}} + 4 \, {{\lambda}}^{2} {{\tau_{1234}}} {{\tau_{23}}} + 4 \, {{\lambda}}^{3} {{\tau_{14}}} + 4 \, {{\lambda}}^{2} {{\tau_{23}}} {{\tau_{14}}} - 8 \, {{\lambda}}^{3} {{\tau_{24}}} - 8 \, {{\lambda}}^{2} {{\tau_{13}}} {{\tau_{24}}} + 4 \, {{\lambda}}^{3} {{\tau_{34}}} + 4 \, {{\lambda}}^{2} {{\tau_{12}}} {{\tau_{34}}} + 4 \, {{\lambda}}^{3} {{\tau_{123}}} + 8 \, {{\lambda}}^{2} {{\tau_{12}}} {{\tau_{123}}} + 4 \, {{\lambda}}^{3} {{\tau_{124}}} - 4 \, {{\lambda}}^{2} {{\tau_{12}}} {{\tau_{124}}} + 4 \, {{\lambda}}^{3} {{\tau_{134}}} + 8 \, {{\lambda}}^{2} {{\tau_{13}}} {{\tau_{134}}} $}

 \parbox[b]{0.9\textwidth}{\raggedright\hangafter=1\hangindent=2em$\displaystyle [\cEll{\{{1}\}\{{2}\}\{{3}\}\{{4}\}}]= 4 \, {{\lambda}}^{5} + 4 \, {{\lambda}}^{4} {{\tau_{12}}} + 4 \, {{\lambda}}^{4} {{\tau_{234}}} + 4 \, {{\lambda}}^{4} {{\tau_{1234}}} - 4 \, {{\lambda}}^{3} {{\tau_{12}}} {{\tau_{1234}}} - 24 \, {{\lambda}}^{2} {{\tau_{12}}}^{2} {{\tau_{1234}}} + 8 \, {{\lambda}}^{3} {{\tau_{234}}} {{\tau_{1234}}} + 4 \, {{\lambda}}^{4} {{\tau_{13}}} - 4 \, {{\lambda}}^{3} {{\tau_{1234}}} {{\tau_{13}}} + 4 \, {{\lambda}}^{4} {{\tau_{23}}} + 8 \, {{\lambda}}^{3} {{\tau_{234}}} {{\tau_{23}}} + 16 \, {{\lambda}}^{3} {{\tau_{1234}}} {{\tau_{23}}} + 4 \, {{\lambda}}^{4} {{\tau_{14}}} + 12 \, {{\lambda}}^{3} {{\tau_{1234}}} {{\tau_{14}}} + 4 \, {{\lambda}}^{3} {{\tau_{23}}} {{\tau_{14}}} + 4 \, {{\lambda}}^{4} {{\tau_{24}}} + 4 \, {{\lambda}}^{3} {{\tau_{13}}} {{\tau_{24}}} + 4 \, {{\lambda}}^{4} {{\tau_{34}}} + 4 \, {{\lambda}}^{3} {{\tau_{12}}} {{\tau_{34}}} + 4 \, {{\lambda}}^{4} {{\tau_{123}}} + 8 \, {{\lambda}}^{3} {{\tau_{12}}} {{\tau_{123}}} + 4 \, {{\lambda}}^{4} {{\tau_{124}}} + 8 \, {{\lambda}}^{3} {{\tau_{12}}} {{\tau_{124}}} + 4 \, {{\lambda}}^{4} {{\tau_{134}}} + 8 \, {{\lambda}}^{3} {{\tau_{13}}} {{\tau_{134}}} $}

$n=5$.

\parbox[b]{0.9\textwidth}{\raggedright\hangafter=1\hangindent=2em$\displaystyle
 [\cEll{\{{1, 2, 3, 4, 5}\}}]= 24 \, {{\lambda}}^{2} $}
\\\\
\parbox[b]{0.9\textwidth}{\raggedright\hangafter=1\hangindent=2em$\displaystyle
 [\cEll{\{{1}\}\{{2, 3, 4, 5}\}}]= 6 \, {{\lambda}}^{3} + 6 \, {{\lambda}}^{2} {{\tau_{12}}} - 2 \, {{\lambda}}^{2} {{\tau_{45}}} - 2 \, {{\lambda}} {{\tau_{12}}} {{\tau_{45}}} + 6 \, {{\lambda}}^{2} {{\tau_{123}}} + 4 \, {{\lambda}} {{\tau_{12}}} {{\tau_{123}}} - 2 \, {{\lambda}} {{\tau_{45}}} {{\tau_{123}}} + 6 \, {{\lambda}}^{2} {{\tau_{124}}} + 4 \, {{\lambda}} {{\tau_{12}}} {{\tau_{124}}} + 6 \, {{\lambda}}^{2} {{\tau_{125}}} + 4 \, {{\lambda}} {{\tau_{12}}} {{\tau_{125}}} + 6 \, {{\lambda}}^{2} {{\tau_{134}}} + 6 \, {{\lambda}}^{2} {{\tau_{135}}} - 6 \, {{\lambda}}^{2} {{\tau_{234}}} - 6 \, {{\lambda}}^{2} {{\tau_{235}}} + 6 \, {{\lambda}}^{2} {{\tau_{145}}} - 6 \, {{\lambda}}^{2} {{\tau_{245}}} + 6 \, {{\lambda}}^{2} {{\tau_{13}}} - 2 \, {{\lambda}} {{\tau_{45}}} {{\tau_{13}}} + 4 \, {{\lambda}} {{\tau_{134}}} {{\tau_{13}}} + 4 \, {{\lambda}} {{\tau_{135}}} {{\tau_{13}}} - 6 \, {{\lambda}} {{\tau_{245}}} {{\tau_{13}}} - 6 \, {{\lambda}}^{2} {{\tau_{345}}} - 6 \, {{\lambda}} {{\tau_{12}}} {{\tau_{345}}} + 6 \, {{\lambda}}^{2} {{\tau_{1234}}} + 2 \, {{\lambda}} {{\tau_{12}}} {{\tau_{1234}}} + 2 \, {{\lambda}} {{\tau_{13}}} {{\tau_{1234}}} + 6 \, {{\lambda}}^{2} {{\tau_{1235}}} + 2 \, {{\lambda}} {{\tau_{12}}} {{\tau_{1235}}} + 2 \, {{\lambda}} {{\tau_{13}}} {{\tau_{1235}}} + 6 \, {{\lambda}}^{2} {{\tau_{1245}}} + 2 \, {{\lambda}} {{\tau_{12}}} {{\tau_{1245}}} + 6 \, {{\lambda}}^{2} {{\tau_{1345}}} + 2 \, {{\lambda}} {{\tau_{13}}} {{\tau_{1345}}} - 18 \, {{\lambda}}^{2} {{\tau_{2345}}} + 12 \, {{\lambda}} {{\tau_{345}}} {{\tau_{2345}}} + 6 \, {{\lambda}}^{2} {{\tau_{12345}}} - 2 \, {{\lambda}} {{\tau_{45}}} {{\tau_{12345}}} - 6 \, {{\lambda}} {{\tau_{123}}} {{\tau_{12345}}} - 3 \, {{\lambda}} {{\tau_{1234}}} {{\tau_{12345}}} - 3 \, {{\lambda}} {{\tau_{1235}}} {{\tau_{12345}}} - 6 \, {{\lambda}} {{\tau_{2345}}} {{\tau_{12345}}} - 2 \, {{\lambda}}^{2} {{\tau_{23}}} + 2 \, {{\lambda}} {{\tau_{45}}} {{\tau_{23}}} - 2 \, {{\lambda}} {{\tau_{145}}} {{\tau_{23}}} - 6 \, {{\lambda}} {{\tau_{2345}}} {{\tau_{23}}} - 2 \, {{\lambda}} {{\tau_{12345}}} {{\tau_{23}}} + 6 \, {{\lambda}}^{2} {{\tau_{14}}} - 6 \, {{\lambda}} {{\tau_{235}}} {{\tau_{14}}} + 4 \, {{\lambda}} {{\tau_{145}}} {{\tau_{14}}} + 2 \, {{\lambda}} {{\tau_{1234}}} {{\tau_{14}}} + 2 \, {{\lambda}} {{\tau_{1245}}} {{\tau_{14}}} + 2 \, {{\lambda}} {{\tau_{1345}}} {{\tau_{14}}} + 3 \, {{\lambda}} {{\tau_{12345}}} {{\tau_{14}}} - 2 \, {{\lambda}} {{\tau_{23}}} {{\tau_{14}}} - 2 \, {{\lambda}}^{2} {{\tau_{24}}} - 2 \, {{\lambda}} {{\tau_{135}}} {{\tau_{24}}} - 2 \, {{\lambda}} {{\tau_{13}}} {{\tau_{24}}} - 6 \, {{\lambda}} {{\tau_{2345}}} {{\tau_{24}}} + {{\lambda}} {{\tau_{12345}}} {{\tau_{24}}} - 2 \, {{\lambda}}^{2} {{\tau_{34}}} - 2 \, {{\lambda}} {{\tau_{12}}} {{\tau_{34}}} - 2 \, {{\lambda}} {{\tau_{125}}} {{\tau_{34}}} + 12 \, {{\lambda}} {{\tau_{2345}}} {{\tau_{34}}} + {{\lambda}} {{\tau_{12345}}} {{\tau_{34}}} + 6 \, {{\lambda}}^{2} {{\tau_{15}}} - 6 \, {{\lambda}} {{\tau_{234}}} {{\tau_{15}}} + 2 \, {{\lambda}} {{\tau_{1235}}} {{\tau_{15}}} + 2 \, {{\lambda}} {{\tau_{1245}}} {{\tau_{15}}} + 2 \, {{\lambda}} {{\tau_{1345}}} {{\tau_{15}}} + 3 \, {{\lambda}} {{\tau_{12345}}} {{\tau_{15}}} - 2 \, {{\lambda}} {{\tau_{23}}} {{\tau_{15}}} - 2 \, {{\lambda}} {{\tau_{24}}} {{\tau_{15}}} - 2 \, {{\lambda}} {{\tau_{34}}} {{\tau_{15}}} - 2 \, {{\lambda}}^{2} {{\tau_{25}}} - 2 \, {{\lambda}} {{\tau_{134}}} {{\tau_{25}}} - 2 \, {{\lambda}} {{\tau_{13}}} {{\tau_{25}}} + 6 \, {{\lambda}} {{\tau_{2345}}} {{\tau_{25}}} + {{\lambda}} {{\tau_{12345}}} {{\tau_{25}}} - 2 \, {{\lambda}} {{\tau_{14}}} {{\tau_{25}}} + 2 \, {{\lambda}} {{\tau_{34}}} {{\tau_{25}}} - 2 \, {{\lambda}}^{2} {{\tau_{35}}} - 2 \, {{\lambda}} {{\tau_{12}}} {{\tau_{35}}} - 2 \, {{\lambda}} {{\tau_{124}}} {{\tau_{35}}} + {{\lambda}} {{\tau_{12345}}} {{\tau_{35}}} - 2 \, {{\lambda}} {{\tau_{14}}} {{\tau_{35}}} + 2 \, {{\lambda}} {{\tau_{24}}} {{\tau_{35}}} $}
\\\\
 \parbox[b]{0.9\textwidth}{\raggedright\hangafter=1\hangindent=2em$\displaystyle
 [\cEll{\{{1, 3}\}\{{2, 4, 5}\}}]= 2 \, {{\lambda}}^{3} + 2 \, {{\lambda}}^{2} {{\tau_{12}}} - 2 \, {{\lambda}}^{2} {{\tau_{45}}} - 2 \, {{\lambda}} {{\tau_{12}}} {{\tau_{45}}} + 2 \, {{\lambda}}^{2} {{\tau_{123}}} - 4 \, {{\lambda}} {{\tau_{12}}} {{\tau_{123}}} - 2 \, {{\lambda}} {{\tau_{45}}} {{\tau_{123}}} + 2 \, {{\lambda}}^{2} {{\tau_{124}}} + 2 \, {{\lambda}}^{2} {{\tau_{125}}} + 2 \, {{\lambda}}^{2} {{\tau_{134}}} + 2 \, {{\lambda}}^{2} {{\tau_{135}}} + 2 \, {{\lambda}}^{2} {{\tau_{234}}} + 2 \, {{\lambda}}^{2} {{\tau_{235}}} + 2 \, {{\lambda}}^{2} {{\tau_{145}}} - 10 \, {{\lambda}}^{2} {{\tau_{245}}} - 6 \, {{\lambda}}^{2} {{\tau_{13}}} + 2 \, {{\lambda}} {{\tau_{45}}} {{\tau_{13}}} - 4 \, {{\lambda}} {{\tau_{134}}} {{\tau_{13}}} - 4 \, {{\lambda}} {{\tau_{135}}} {{\tau_{13}}} + 6 \, {{\lambda}} {{\tau_{245}}} {{\tau_{13}}} + 2 \, {{\lambda}}^{2} {{\tau_{345}}} + 2 \, {{\lambda}} {{\tau_{12}}} {{\tau_{345}}} + 2 \, {{\lambda}}^{2} {{\tau_{1234}}} - 2 \, {{\lambda}} {{\tau_{12}}} {{\tau_{1234}}} + 4 \, {{\lambda}} {{\tau_{234}}} {{\tau_{1234}}} - 6 \, {{\lambda}} {{\tau_{13}}} {{\tau_{1234}}} + 2 \, {{\lambda}}^{2} {{\tau_{1235}}} - 2 \, {{\lambda}} {{\tau_{12}}} {{\tau_{1235}}} + 4 \, {{\lambda}} {{\tau_{235}}} {{\tau_{1235}}} - 6 \, {{\lambda}} {{\tau_{13}}} {{\tau_{1235}}} + 2 \, {{\lambda}}^{2} {{\tau_{1245}}} + 2 \, {{\lambda}} {{\tau_{12}}} {{\tau_{1245}}} - 8 \, {{\lambda}} {{\tau_{245}}} {{\tau_{1245}}} + 2 \, {{\lambda}}^{2} {{\tau_{1345}}} - 6 \, {{\lambda}} {{\tau_{13}}} {{\tau_{1345}}} + 4 \, {{\lambda}} {{\tau_{345}}} {{\tau_{1345}}} + 2 \, {{\lambda}}^{2} {{\tau_{2345}}} - 8 \, {{\lambda}} {{\tau_{345}}} {{\tau_{2345}}} + 2 \, {{\lambda}}^{2} {{\tau_{12345}}} + 2 \, {{\lambda}} {{\tau_{45}}} {{\tau_{12345}}} + 2 \, {{\lambda}} {{\tau_{123}}} {{\tau_{12345}}} - 2 \, {{\lambda}} {{\tau_{13}}} {{\tau_{12345}}} - {{\lambda}} {{\tau_{1234}}} {{\tau_{12345}}} - {{\lambda}} {{\tau_{1235}}} {{\tau_{12345}}} + 4 \, {{\lambda}} {{\tau_{1245}}} {{\tau_{12345}}} - 2 \, {{\lambda}} {{\tau_{1345}}} {{\tau_{12345}}} + 4 \, {{\lambda}} {{\tau_{2345}}} {{\tau_{12345}}} + 2 \, {{\lambda}}^{2} {{\tau_{23}}} - 2 \, {{\lambda}} {{\tau_{45}}} {{\tau_{23}}} + 2 \, {{\lambda}} {{\tau_{145}}} {{\tau_{23}}} + 4 \, {{\lambda}} {{\tau_{1234}}} {{\tau_{23}}} + 4 \, {{\lambda}} {{\tau_{1235}}} {{\tau_{23}}} + 2 \, {{\lambda}} {{\tau_{2345}}} {{\tau_{23}}} + 2 \, {{\lambda}}^{2} {{\tau_{14}}} + 2 \, {{\lambda}} {{\tau_{235}}} {{\tau_{14}}} + 2 \, {{\lambda}} {{\tau_{1234}}} {{\tau_{14}}} + 2 \, {{\lambda}} {{\tau_{1245}}} {{\tau_{14}}} - 2 \, {{\lambda}} {{\tau_{1345}}} {{\tau_{14}}} - {{\lambda}} {{\tau_{12345}}} {{\tau_{14}}} + 2 \, {{\lambda}} {{\tau_{23}}} {{\tau_{14}}} - 2 \, {{\lambda}}^{2} {{\tau_{24}}} - 2 \, {{\lambda}} {{\tau_{135}}} {{\tau_{24}}} + 4 \, {{\lambda}} {{\tau_{245}}} {{\tau_{24}}} + 2 \, {{\lambda}} {{\tau_{13}}} {{\tau_{24}}} - 4 \, {{\lambda}} {{\tau_{1245}}} {{\tau_{24}}} + 6 \, {{\lambda}} {{\tau_{2345}}} {{\tau_{24}}} + {{\lambda}} {{\tau_{12345}}} {{\tau_{24}}} + 2 \, {{\lambda}}^{2} {{\tau_{34}}} + 2 \, {{\lambda}} {{\tau_{12}}} {{\tau_{34}}} + 2 \, {{\lambda}} {{\tau_{125}}} {{\tau_{34}}} + 4 \, {{\lambda}} {{\tau_{1345}}} {{\tau_{34}}} - 8 \, {{\lambda}} {{\tau_{2345}}} {{\tau_{34}}} - {{\lambda}} {{\tau_{12345}}} {{\tau_{34}}} + 2 \, {{\lambda}}^{2} {{\tau_{15}}} + 2 \, {{\lambda}} {{\tau_{234}}} {{\tau_{15}}} + 2 \, {{\lambda}} {{\tau_{1235}}} {{\tau_{15}}} - 2 \, {{\lambda}} {{\tau_{1245}}} {{\tau_{15}}} + 2 \, {{\lambda}} {{\tau_{1345}}} {{\tau_{15}}} - {{\lambda}} {{\tau_{12345}}} {{\tau_{15}}} + 2 \, {{\lambda}} {{\tau_{23}}} {{\tau_{15}}} - 2 \, {{\lambda}} {{\tau_{24}}} {{\tau_{15}}} + 2 \, {{\lambda}} {{\tau_{34}}} {{\tau_{15}}} - 2 \, {{\lambda}}^{2} {{\tau_{25}}} - 2 \, {{\lambda}} {{\tau_{134}}} {{\tau_{25}}} + 2 \, {{\lambda}} {{\tau_{13}}} {{\tau_{25}}} - 2 \, {{\lambda}} {{\tau_{2345}}} {{\tau_{25}}} + {{\lambda}} {{\tau_{12345}}} {{\tau_{25}}} - 2 \, {{\lambda}} {{\tau_{14}}} {{\tau_{25}}} - 2 \, {{\lambda}} {{\tau_{34}}} {{\tau_{25}}} + 2 \, {{\lambda}}^{2} {{\tau_{35}}} + 2 \, {{\lambda}} {{\tau_{12}}} {{\tau_{35}}} + 2 \, {{\lambda}} {{\tau_{124}}} {{\tau_{35}}} - {{\lambda}} {{\tau_{12345}}} {{\tau_{35}}} + 2 \, {{\lambda}} {{\tau_{14}}} {{\tau_{35}}} - 2 \, {{\lambda}} {{\tau_{24}}} {{\tau_{35}}} $}
\\\\
 \parbox[b]{0.9\textwidth}{\raggedright\hangafter=1\hangindent=2em$\displaystyle
 [\cEll{\{{1}\}\{{2, 4, 5}\}\{{3}\}}]= 2 \, {{\lambda}}^{4} + 2 \, {{\lambda}}^{3} {{\tau_{12}}} - 2 \, {{\lambda}}^{3} {{\tau_{45}}} - 2 \, {{\lambda}}^{2} {{\tau_{12}}} {{\tau_{45}}} + 2 \, {{\lambda}}^{3} {{\tau_{123}}} + 4 \, {{\lambda}}^{2} {{\tau_{12}}} {{\tau_{123}}} - 2 \, {{\lambda}}^{2} {{\tau_{45}}} {{\tau_{123}}} - 4 \, {{\lambda}} {{\tau_{12}}} {{\tau_{45}}} {{\tau_{123}}} + 2 \, {{\lambda}}^{3} {{\tau_{124}}} + 2 \, {{\lambda}}^{3} {{\tau_{125}}} + 2 \, {{\lambda}}^{3} {{\tau_{134}}} + 2 \, {{\lambda}}^{3} {{\tau_{135}}} + 2 \, {{\lambda}}^{3} {{\tau_{234}}} + 2 \, {{\lambda}}^{3} {{\tau_{235}}} + 2 \, {{\lambda}}^{3} {{\tau_{145}}} - 10 \, {{\lambda}}^{3} {{\tau_{245}}} + 2 \, {{\lambda}}^{3} {{\tau_{13}}} - 2 \, {{\lambda}}^{2} {{\tau_{45}}} {{\tau_{13}}} + 4 \, {{\lambda}}^{2} {{\tau_{134}}} {{\tau_{13}}} + 4 \, {{\lambda}}^{2} {{\tau_{135}}} {{\tau_{13}}} - 10 \, {{\lambda}}^{2} {{\tau_{245}}} {{\tau_{13}}} + 2 \, {{\lambda}}^{3} {{\tau_{345}}} + 2 \, {{\lambda}}^{2} {{\tau_{12}}} {{\tau_{345}}} + 2 \, {{\lambda}}^{3} {{\tau_{1234}}} - 2 \, {{\lambda}}^{2} {{\tau_{12}}} {{\tau_{1234}}} - 4 \, {{\lambda}} {{\tau_{12}}}^{2} {{\tau_{1234}}} + 4 \, {{\lambda}}^{2} {{\tau_{234}}} {{\tau_{1234}}} + 2 \, {{\lambda}}^{2} {{\tau_{13}}} {{\tau_{1234}}} + 2 \, {{\lambda}}^{3} {{\tau_{1235}}} - 2 \, {{\lambda}}^{2} {{\tau_{12}}} {{\tau_{1235}}} - 4 \, {{\lambda}} {{\tau_{12}}}^{2} {{\tau_{1235}}} + 4 \, {{\lambda}}^{2} {{\tau_{235}}} {{\tau_{1235}}} + 2 \, {{\lambda}}^{2} {{\tau_{13}}} {{\tau_{1235}}} + 2 \, {{\lambda}}^{3} {{\tau_{1245}}} + 2 \, {{\lambda}}^{2} {{\tau_{12}}} {{\tau_{1245}}} - 8 \, {{\lambda}}^{2} {{\tau_{245}}} {{\tau_{1245}}} + 2 \, {{\lambda}}^{3} {{\tau_{1345}}} + 2 \, {{\lambda}}^{2} {{\tau_{13}}} {{\tau_{1345}}} - 4 \, {{\lambda}} {{\tau_{13}}}^{2} {{\tau_{1345}}} + 4 \, {{\lambda}}^{2} {{\tau_{345}}} {{\tau_{1345}}} + 2 \, {{\lambda}}^{3} {{\tau_{2345}}} - 8 \, {{\lambda}}^{2} {{\tau_{345}}} {{\tau_{2345}}} + 2 \, {{\lambda}}^{3} {{\tau_{12345}}} - 2 \, {{\lambda}} {{\tau_{12}}}^{2} {{\tau_{12345}}} + 2 \, {{\lambda}}^{2} {{\tau_{45}}} {{\tau_{12345}}} - 4 \, {{\lambda}} {{\tau_{12}}} {{\tau_{45}}} {{\tau_{12345}}} + 2 \, {{\lambda}}^{2} {{\tau_{123}}} {{\tau_{12345}}} + 6 \, {{\lambda}}^{2} {{\tau_{13}}} {{\tau_{12345}}} + 2 \, {{\lambda}} {{\tau_{13}}}^{2} {{\tau_{12345}}} - {{\lambda}}^{2} {{\tau_{1234}}} {{\tau_{12345}}} - {{\lambda}}^{2} {{\tau_{1235}}} {{\tau_{12345}}} + 4 \, {{\lambda}}^{2} {{\tau_{1245}}} {{\tau_{12345}}} - 2 \, {{\lambda}} {{\tau_{12}}} {{\tau_{1245}}} {{\tau_{12345}}} - 2 \, {{\lambda}}^{2} {{\tau_{1345}}} {{\tau_{12345}}} - 6 \, {{\lambda}} {{\tau_{13}}} {{\tau_{1345}}} {{\tau_{12345}}} + 4 \, {{\lambda}}^{2} {{\tau_{2345}}} {{\tau_{12345}}} + 2 \, {{\lambda}}^{3} {{\tau_{23}}} - 2 \, {{\lambda}}^{2} {{\tau_{45}}} {{\tau_{23}}} + 2 \, {{\lambda}}^{2} {{\tau_{145}}} {{\tau_{23}}} + 4 \, {{\lambda}}^{2} {{\tau_{1234}}} {{\tau_{23}}} + 4 \, {{\lambda}}^{2} {{\tau_{1235}}} {{\tau_{23}}} + 2 \, {{\lambda}}^{2} {{\tau_{2345}}} {{\tau_{23}}} + 2 \, {{\lambda}}^{3} {{\tau_{14}}} + 2 \, {{\lambda}}^{2} {{\tau_{235}}} {{\tau_{14}}} + 2 \, {{\lambda}}^{2} {{\tau_{1234}}} {{\tau_{14}}} + 2 \, {{\lambda}}^{2} {{\tau_{1245}}} {{\tau_{14}}} - 2 \, {{\lambda}}^{2} {{\tau_{1345}}} {{\tau_{14}}} - {{\lambda}}^{2} {{\tau_{12345}}} {{\tau_{14}}} + 2 \, {{\lambda}}^{2} {{\tau_{23}}} {{\tau_{14}}} - 2 \, {{\lambda}}^{3} {{\tau_{24}}} - 2 \, {{\lambda}}^{2} {{\tau_{135}}} {{\tau_{24}}} + 4 \, {{\lambda}}^{2} {{\tau_{245}}} {{\tau_{24}}} - 2 \, {{\lambda}}^{2} {{\tau_{13}}} {{\tau_{24}}} - 4 \, {{\lambda}} {{\tau_{135}}} {{\tau_{13}}} {{\tau_{24}}} + 4 \, {{\lambda}} {{\tau_{245}}} {{\tau_{13}}} {{\tau_{24}}} - 4 \, {{\lambda}}^{2} {{\tau_{1245}}} {{\tau_{24}}} + 6 \, {{\lambda}}^{2} {{\tau_{2345}}} {{\tau_{24}}} + {{\lambda}}^{2} {{\tau_{12345}}} {{\tau_{24}}} + 2 \, {{\lambda}} {{\tau_{13}}} {{\tau_{12345}}} {{\tau_{24}}} + 2 \, {{\lambda}}^{3} {{\tau_{34}}} + 2 \, {{\lambda}}^{2} {{\tau_{12}}} {{\tau_{34}}} + 2 \, {{\lambda}}^{2} {{\tau_{125}}} {{\tau_{34}}} + 4 \, {{\lambda}}^{2} {{\tau_{1345}}} {{\tau_{34}}} - 8 \, {{\lambda}}^{2} {{\tau_{2345}}} {{\tau_{34}}} - {{\lambda}}^{2} {{\tau_{12345}}} {{\tau_{34}}} + 2 \, {{\lambda}} {{\tau_{12}}} {{\tau_{12345}}} {{\tau_{34}}} + 2 \, {{\lambda}}^{3} {{\tau_{15}}} + 2 \, {{\lambda}}^{2} {{\tau_{234}}} {{\tau_{15}}} + 2 \, {{\lambda}}^{2} {{\tau_{1235}}} {{\tau_{15}}} - 2 \, {{\lambda}}^{2} {{\tau_{1245}}} {{\tau_{15}}} + 2 \, {{\lambda}}^{2} {{\tau_{1345}}} {{\tau_{15}}} - {{\lambda}}^{2} {{\tau_{12345}}} {{\tau_{15}}} + 2 \, {{\lambda}}^{2} {{\tau_{23}}} {{\tau_{15}}} - 2 \, {{\lambda}}^{2} {{\tau_{24}}} {{\tau_{15}}} + 2 \, {{\lambda}}^{2} {{\tau_{34}}} {{\tau_{15}}} - 2 \, {{\lambda}}^{3} {{\tau_{25}}} - 2 \, {{\lambda}}^{2} {{\tau_{134}}} {{\tau_{25}}} - 2 \, {{\lambda}}^{2} {{\tau_{13}}} {{\tau_{25}}} - 4 \, {{\lambda}} {{\tau_{134}}} {{\tau_{13}}} {{\tau_{25}}} - 2 \, {{\lambda}}^{2} {{\tau_{2345}}} {{\tau_{25}}} + {{\lambda}}^{2} {{\tau_{12345}}} {{\tau_{25}}} + 2 \, {{\lambda}} {{\tau_{13}}} {{\tau_{12345}}} {{\tau_{25}}} - 2 \, {{\lambda}}^{2} {{\tau_{14}}} {{\tau_{25}}} - 2 \, {{\lambda}}^{2} {{\tau_{34}}} {{\tau_{25}}} + 2 \, {{\lambda}}^{3} {{\tau_{35}}} + 2 \, {{\lambda}}^{2} {{\tau_{12}}} {{\tau_{35}}} + 2 \, {{\lambda}}^{2} {{\tau_{124}}} {{\tau_{35}}} - {{\lambda}}^{2} {{\tau_{12345}}} {{\tau_{35}}} + 2 \, {{\lambda}} {{\tau_{12}}} {{\tau_{12345}}} {{\tau_{35}}} + 2 \, {{\lambda}}^{2} {{\tau_{14}}} {{\tau_{35}}} - 2 \, {{\lambda}}^{2} {{\tau_{24}}} {{\tau_{35}}} $}
\\\\
 \parbox[b]{0.9\textwidth}{\raggedright\hangafter=1\hangindent=2em$\displaystyle
 [\cEll{\{{1}\}\{{2, 3}\}\{{4, 5}\}}]= {{\lambda}}^{4} + {{\lambda}}^{3} {{\tau_{12}}} - 3 \, {{\lambda}}^{3} {{\tau_{45}}} - 3 \, {{\lambda}}^{2} {{\tau_{12}}} {{\tau_{45}}} + {{\lambda}}^{3} {{\tau_{123}}} - 2 \, {{\lambda}}^{2} {{\tau_{12}}} {{\tau_{123}}} - 3 \, {{\lambda}}^{2} {{\tau_{45}}} {{\tau_{123}}} + 2 \, {{\lambda}} {{\tau_{12}}} {{\tau_{45}}} {{\tau_{123}}} + {{\lambda}}^{3} {{\tau_{124}}} + 2 \, {{\lambda}}^{2} {{\tau_{12}}} {{\tau_{124}}} + {{\lambda}}^{3} {{\tau_{125}}} + 2 \, {{\lambda}}^{2} {{\tau_{12}}} {{\tau_{125}}} + {{\lambda}}^{3} {{\tau_{134}}} + {{\lambda}}^{3} {{\tau_{135}}} + {{\lambda}}^{3} {{\tau_{234}}} + {{\lambda}}^{3} {{\tau_{235}}} + {{\lambda}}^{3} {{\tau_{145}}} + {{\lambda}}^{3} {{\tau_{245}}} + {{\lambda}}^{3} {{\tau_{13}}} - 3 \, {{\lambda}}^{2} {{\tau_{45}}} {{\tau_{13}}} + 2 \, {{\lambda}}^{2} {{\tau_{134}}} {{\tau_{13}}} + 2 \, {{\lambda}}^{2} {{\tau_{135}}} {{\tau_{13}}} + {{\lambda}}^{2} {{\tau_{245}}} {{\tau_{13}}} + {{\lambda}}^{3} {{\tau_{345}}} + {{\lambda}}^{2} {{\tau_{12}}} {{\tau_{345}}} + {{\lambda}}^{3} {{\tau_{1234}}} - {{\lambda}}^{2} {{\tau_{12}}} {{\tau_{1234}}} + 2 \, {{\lambda}} {{\tau_{12}}}^{2} {{\tau_{1234}}} + 2 \, {{\lambda}}^{2} {{\tau_{234}}} {{\tau_{1234}}} - {{\lambda}}^{2} {{\tau_{13}}} {{\tau_{1234}}} + {{\lambda}}^{3} {{\tau_{1235}}} - {{\lambda}}^{2} {{\tau_{12}}} {{\tau_{1235}}} + 2 \, {{\lambda}} {{\tau_{12}}}^{2} {{\tau_{1235}}} + 2 \, {{\lambda}}^{2} {{\tau_{235}}} {{\tau_{1235}}} - {{\lambda}}^{2} {{\tau_{13}}} {{\tau_{1235}}} + {{\lambda}}^{3} {{\tau_{1245}}} + 3 \, {{\lambda}}^{2} {{\tau_{12}}} {{\tau_{1245}}} + 2 \, {{\lambda}} {{\tau_{12}}}^{2} {{\tau_{1245}}} + 2 \, {{\lambda}}^{2} {{\tau_{245}}} {{\tau_{1245}}} + {{\lambda}}^{3} {{\tau_{1345}}} + 3 \, {{\lambda}}^{2} {{\tau_{13}}} {{\tau_{1345}}} + 2 \, {{\lambda}} {{\tau_{13}}}^{2} {{\tau_{1345}}} + 2 \, {{\lambda}}^{2} {{\tau_{345}}} {{\tau_{1345}}} + {{\lambda}}^{3} {{\tau_{2345}}} + 2 \, {{\lambda}}^{2} {{\tau_{345}}} {{\tau_{2345}}} + {{\lambda}}^{3} {{\tau_{12345}}} + 3 \, {{\lambda}}^{2} {{\tau_{12}}} {{\tau_{12345}}} + 2 \, {{\lambda}} {{\tau_{12}}} {{\tau_{45}}} {{\tau_{12345}}} + 9 \, {{\lambda}}^{2} {{\tau_{123}}} {{\tau_{12345}}} - {{\lambda}}^{2} {{\tau_{124}}} {{\tau_{12345}}} - {{\lambda}}^{2} {{\tau_{125}}} {{\tau_{12345}}} + 4 \, {{\lambda}}^{2} {{\tau_{13}}} {{\tau_{12345}}} + 2 \, {{\lambda}} {{\tau_{13}}}^{2} {{\tau_{12345}}} + 3 \, {{\lambda}}^{2} {{\tau_{1234}}} {{\tau_{12345}}} + 3 \, {{\lambda}}^{2} {{\tau_{1235}}} {{\tau_{12345}}} - 2 \, {{\lambda}}^{2} {{\tau_{1245}}} {{\tau_{12345}}} + 4 \, {{\lambda}} {{\tau_{12}}} {{\tau_{1245}}} {{\tau_{12345}}} - {{\lambda}}^{2} {{\tau_{1345}}} {{\tau_{12345}}} + 6 \, {{\lambda}} {{\tau_{13}}} {{\tau_{1345}}} {{\tau_{12345}}} - {{\lambda}}^{2} {{\tau_{2345}}} {{\tau_{12345}}} + 3 \, {{\lambda}} {{\tau_{2345}}}^{2} {{\tau_{12345}}} - 3 \, {{\lambda}}^{3} {{\tau_{23}}} + 5 \, {{\lambda}}^{2} {{\tau_{45}}} {{\tau_{23}}} - 2 \, {{\lambda}}^{2} {{\tau_{234}}} {{\tau_{23}}} - 2 \, {{\lambda}}^{2} {{\tau_{235}}} {{\tau_{23}}} - 3 \, {{\lambda}}^{2} {{\tau_{145}}} {{\tau_{23}}} - {{\lambda}}^{2} {{\tau_{2345}}} {{\tau_{23}}} - 2 \, {{\lambda}} {{\tau_{2345}}} {{\tau_{23}}}^{2} + {{\lambda}} {{\tau_{12345}}} {{\tau_{23}}}^{2} + {{\lambda}}^{3} {{\tau_{14}}} + {{\lambda}}^{2} {{\tau_{235}}} {{\tau_{14}}} - 2 \, {{\lambda}}^{2} {{\tau_{145}}} {{\tau_{14}}} + 3 \, {{\lambda}}^{2} {{\tau_{1234}}} {{\tau_{14}}} - {{\lambda}}^{2} {{\tau_{1245}}} {{\tau_{14}}} - {{\lambda}}^{2} {{\tau_{1345}}} {{\tau_{14}}} - {{\lambda}}^{2} {{\tau_{12345}}} {{\tau_{14}}} - 3 \, {{\lambda}}^{2} {{\tau_{23}}} {{\tau_{14}}} - 2 \, {{\lambda}} {{\tau_{235}}} {{\tau_{23}}} {{\tau_{14}}} + 2 \, {{\lambda}} {{\tau_{145}}} {{\tau_{23}}} {{\tau_{14}}} + 4 \, {{\lambda}} {{\tau_{12345}}} {{\tau_{23}}} {{\tau_{14}}} - 3 \, {{\lambda}} {{\tau_{12345}}} {{\tau_{14}}}^{2} + {{\lambda}}^{3} {{\tau_{24}}} + {{\lambda}}^{2} {{\tau_{135}}} {{\tau_{24}}} - 2 \, {{\lambda}}^{2} {{\tau_{245}}} {{\tau_{24}}} + {{\lambda}}^{2} {{\tau_{13}}} {{\tau_{24}}} + 2 \, {{\lambda}} {{\tau_{135}}} {{\tau_{13}}} {{\tau_{24}}} - 2 \, {{\lambda}} {{\tau_{245}}} {{\tau_{13}}} {{\tau_{24}}} - {{\lambda}}^{2} {{\tau_{2345}}} {{\tau_{24}}} - {{\lambda}}^{2} {{\tau_{12345}}} {{\tau_{24}}} - 4 \, {{\lambda}} {{\tau_{13}}} {{\tau_{12345}}} {{\tau_{24}}} + {{\lambda}}^{3} {{\tau_{34}}} + {{\lambda}}^{2} {{\tau_{12}}} {{\tau_{34}}} + {{\lambda}}^{2} {{\tau_{125}}} {{\tau_{34}}} + 2 \, {{\lambda}} {{\tau_{12}}} {{\tau_{125}}} {{\tau_{34}}} - 2 \, {{\lambda}}^{2} {{\tau_{345}}} {{\tau_{34}}} - 2 \, {{\lambda}} {{\tau_{12}}} {{\tau_{345}}} {{\tau_{34}}} - {{\lambda}}^{2} {{\tau_{12345}}} {{\tau_{34}}} - 2 \, {{\lambda}} {{\tau_{12}}} {{\tau_{12345}}} {{\tau_{34}}} + {{\lambda}}^{3} {{\tau_{15}}} + {{\lambda}}^{2} {{\tau_{234}}} {{\tau_{15}}} + 3 \, {{\lambda}}^{2} {{\tau_{1235}}} {{\tau_{15}}} - {{\lambda}}^{2} {{\tau_{1245}}} {{\tau_{15}}} - {{\lambda}}^{2} {{\tau_{1345}}} {{\tau_{15}}} - {{\lambda}}^{2} {{\tau_{12345}}} {{\tau_{15}}} - 3 \, {{\lambda}}^{2} {{\tau_{23}}} {{\tau_{15}}} - 2 \, {{\lambda}} {{\tau_{234}}} {{\tau_{23}}} {{\tau_{15}}} - 2 \, {{\lambda}} {{\tau_{12345}}} {{\tau_{23}}} {{\tau_{15}}} + {{\lambda}}^{2} {{\tau_{24}}} {{\tau_{15}}} + {{\lambda}}^{2} {{\tau_{34}}} {{\tau_{15}}} - 3 \, {{\lambda}} {{\tau_{12345}}} {{\tau_{15}}}^{2} + {{\lambda}}^{3} {{\tau_{25}}} + {{\lambda}}^{2} {{\tau_{134}}} {{\tau_{25}}} + {{\lambda}}^{2} {{\tau_{13}}} {{\tau_{25}}} + 2 \, {{\lambda}} {{\tau_{134}}} {{\tau_{13}}} {{\tau_{25}}} - {{\lambda}}^{2} {{\tau_{2345}}} {{\tau_{25}}} - {{\lambda}}^{2} {{\tau_{12345}}} {{\tau_{25}}} + 2 \, {{\lambda}} {{\tau_{13}}} {{\tau_{12345}}} {{\tau_{25}}} + {{\lambda}}^{2} {{\tau_{14}}} {{\tau_{25}}} - 6 \, {{\lambda}} {{\tau_{12345}}} {{\tau_{14}}} {{\tau_{25}}} + {{\lambda}}^{2} {{\tau_{34}}} {{\tau_{25}}} + {{\lambda}}^{3} {{\tau_{35}}} + {{\lambda}}^{2} {{\tau_{12}}} {{\tau_{35}}} + {{\lambda}}^{2} {{\tau_{124}}} {{\tau_{35}}} + 2 \, {{\lambda}} {{\tau_{12}}} {{\tau_{124}}} {{\tau_{35}}} - {{\lambda}}^{2} {{\tau_{12345}}} {{\tau_{35}}} - 2 \, {{\lambda}} {{\tau_{12}}} {{\tau_{12345}}} {{\tau_{35}}} + {{\lambda}}^{2} {{\tau_{14}}} {{\tau_{35}}} + {{\lambda}}^{2} {{\tau_{24}}} {{\tau_{35}}} $}
\\\\
 \parbox[b]{0.9\textwidth}{\raggedright\hangafter=1\hangindent=2em$\displaystyle
 [\cEll{\{{1}\}\{{2, 4}\}\{{3}\}\{{5}\}}]= {{\lambda}}^{5} + {{\lambda}}^{4} {{\tau_{12}}} + {{\lambda}}^{4} {{\tau_{45}}} + {{\lambda}}^{3} {{\tau_{12}}} {{\tau_{45}}} + {{\lambda}}^{4} {{\tau_{123}}} + 2 \, {{\lambda}}^{3} {{\tau_{12}}} {{\tau_{123}}} + {{\lambda}}^{3} {{\tau_{45}}} {{\tau_{123}}} + 2 \, {{\lambda}}^{2} {{\tau_{12}}} {{\tau_{45}}} {{\tau_{123}}} + {{\lambda}}^{4} {{\tau_{124}}} - 2 \, {{\lambda}}^{3} {{\tau_{12}}} {{\tau_{124}}} + {{\lambda}}^{4} {{\tau_{125}}} + 2 \, {{\lambda}}^{3} {{\tau_{12}}} {{\tau_{125}}} + {{\lambda}}^{4} {{\tau_{134}}} + {{\lambda}}^{4} {{\tau_{135}}} + {{\lambda}}^{4} {{\tau_{234}}} + {{\lambda}}^{4} {{\tau_{235}}} + {{\lambda}}^{4} {{\tau_{145}}} + {{\lambda}}^{4} {{\tau_{245}}} + {{\lambda}}^{4} {{\tau_{13}}} + {{\lambda}}^{3} {{\tau_{45}}} {{\tau_{13}}} + 2 \, {{\lambda}}^{3} {{\tau_{134}}} {{\tau_{13}}} + 2 \, {{\lambda}}^{3} {{\tau_{135}}} {{\tau_{13}}} + {{\lambda}}^{3} {{\tau_{245}}} {{\tau_{13}}} + {{\lambda}}^{4} {{\tau_{345}}} + {{\lambda}}^{3} {{\tau_{12}}} {{\tau_{345}}} + {{\lambda}}^{4} {{\tau_{1234}}} - {{\lambda}}^{3} {{\tau_{12}}} {{\tau_{1234}}} + 2 \, {{\lambda}}^{2} {{\tau_{12}}}^{2} {{\tau_{1234}}} + 2 \, {{\lambda}}^{3} {{\tau_{234}}} {{\tau_{1234}}} + 3 \, {{\lambda}}^{3} {{\tau_{13}}} {{\tau_{1234}}} + {{\lambda}}^{4} {{\tau_{1235}}} - {{\lambda}}^{3} {{\tau_{12}}} {{\tau_{1235}}} - 6 \, {{\lambda}}^{2} {{\tau_{12}}}^{2} {{\tau_{1235}}} + 2 \, {{\lambda}}^{3} {{\tau_{235}}} {{\tau_{1235}}} - {{\lambda}}^{3} {{\tau_{13}}} {{\tau_{1235}}} + {{\lambda}}^{4} {{\tau_{1245}}} - {{\lambda}}^{3} {{\tau_{12}}} {{\tau_{1245}}} + 2 \, {{\lambda}}^{2} {{\tau_{12}}}^{2} {{\tau_{1245}}} + 2 \, {{\lambda}}^{3} {{\tau_{245}}} {{\tau_{1245}}} + {{\lambda}}^{4} {{\tau_{1345}}} - {{\lambda}}^{3} {{\tau_{13}}} {{\tau_{1345}}} - 6 \, {{\lambda}}^{2} {{\tau_{13}}}^{2} {{\tau_{1345}}} + 2 \, {{\lambda}}^{3} {{\tau_{345}}} {{\tau_{1345}}} + {{\lambda}}^{4} {{\tau_{2345}}} + 2 \, {{\lambda}}^{3} {{\tau_{345}}} {{\tau_{2345}}} + {{\lambda}}^{4} {{\tau_{12345}}} + 3 \, {{\lambda}}^{3} {{\tau_{12}}} {{\tau_{12345}}} + 4 \, {{\lambda}}^{2} {{\tau_{12}}}^{2} {{\tau_{12345}}} + 4 \, {{\lambda}}^{3} {{\tau_{45}}} {{\tau_{12345}}} + 2 \, {{\lambda}}^{2} {{\tau_{12}}} {{\tau_{45}}} {{\tau_{12345}}} + 9 \, {{\lambda}}^{3} {{\tau_{123}}} {{\tau_{12345}}} - {{\lambda}}^{3} {{\tau_{124}}} {{\tau_{12345}}} - {{\lambda}}^{3} {{\tau_{125}}} {{\tau_{12345}}} + 4 \, {{\lambda}}^{3} {{\tau_{13}}} {{\tau_{12345}}} - 4 \, {{\lambda}}^{2} {{\tau_{13}}}^{2} {{\tau_{12345}}} + 3 \, {{\lambda}}^{3} {{\tau_{1234}}} {{\tau_{12345}}} + 3 \, {{\lambda}}^{3} {{\tau_{1235}}} {{\tau_{12345}}} - 2 \, {{\lambda}}^{3} {{\tau_{1245}}} {{\tau_{12345}}} - {{\lambda}}^{3} {{\tau_{1345}}} {{\tau_{12345}}} - {{\lambda}}^{3} {{\tau_{2345}}} {{\tau_{12345}}} - 3 \, {{\lambda}}^{2} {{\tau_{2345}}}^{2} {{\tau_{12345}}} + {{\lambda}}^{4} {{\tau_{23}}} + {{\lambda}}^{3} {{\tau_{45}}} {{\tau_{23}}} - 2 \, {{\lambda}}^{3} {{\tau_{234}}} {{\tau_{23}}} + 2 \, {{\lambda}}^{3} {{\tau_{235}}} {{\tau_{23}}} + {{\lambda}}^{3} {{\tau_{145}}} {{\tau_{23}}} + 4 \, {{\lambda}}^{3} {{\tau_{1235}}} {{\tau_{23}}} - {{\lambda}}^{3} {{\tau_{2345}}} {{\tau_{23}}} + 4 \, {{\lambda}}^{3} {{\tau_{12345}}} {{\tau_{23}}} + 2 \, {{\lambda}}^{2} {{\tau_{2345}}} {{\tau_{23}}}^{2} - {{\lambda}}^{2} {{\tau_{12345}}} {{\tau_{23}}}^{2} + {{\lambda}}^{4} {{\tau_{14}}} + {{\lambda}}^{3} {{\tau_{235}}} {{\tau_{14}}} + 2 \, {{\lambda}}^{3} {{\tau_{145}}} {{\tau_{14}}} - {{\lambda}}^{3} {{\tau_{1234}}} {{\tau_{14}}} - {{\lambda}}^{3} {{\tau_{1245}}} {{\tau_{14}}} - {{\lambda}}^{3} {{\tau_{1345}}} {{\tau_{14}}} - {{\lambda}}^{3} {{\tau_{12345}}} {{\tau_{14}}} + {{\lambda}}^{3} {{\tau_{23}}} {{\tau_{14}}} + 2 \, {{\lambda}}^{2} {{\tau_{235}}} {{\tau_{23}}} {{\tau_{14}}} + 2 \, {{\lambda}}^{2} {{\tau_{145}}} {{\tau_{23}}} {{\tau_{14}}} + 4 \, {{\lambda}}^{2} {{\tau_{12345}}} {{\tau_{23}}} {{\tau_{14}}} + 3 \, {{\lambda}}^{2} {{\tau_{12345}}} {{\tau_{14}}}^{2} - 3 \, {{\lambda}}^{4} {{\tau_{24}}} - 3 \, {{\lambda}}^{3} {{\tau_{135}}} {{\tau_{24}}} - 2 \, {{\lambda}}^{3} {{\tau_{245}}} {{\tau_{24}}} - 3 \, {{\lambda}}^{3} {{\tau_{13}}} {{\tau_{24}}} - 6 \, {{\lambda}}^{2} {{\tau_{135}}} {{\tau_{13}}} {{\tau_{24}}} - 2 \, {{\lambda}}^{2} {{\tau_{245}}} {{\tau_{13}}} {{\tau_{24}}} - 5 \, {{\lambda}}^{3} {{\tau_{2345}}} {{\tau_{24}}} - 5 \, {{\lambda}}^{3} {{\tau_{12345}}} {{\tau_{24}}} - 6 \, {{\lambda}}^{2} {{\tau_{13}}} {{\tau_{12345}}} {{\tau_{24}}} + {{\lambda}}^{4} {{\tau_{34}}} + {{\lambda}}^{3} {{\tau_{12}}} {{\tau_{34}}} + {{\lambda}}^{3} {{\tau_{125}}} {{\tau_{34}}} + 2 \, {{\lambda}}^{2} {{\tau_{12}}} {{\tau_{125}}} {{\tau_{34}}} + 2 \, {{\lambda}}^{3} {{\tau_{345}}} {{\tau_{34}}} + 2 \, {{\lambda}}^{2} {{\tau_{12}}} {{\tau_{345}}} {{\tau_{34}}} + 4 \, {{\lambda}}^{3} {{\tau_{1345}}} {{\tau_{34}}} + 4 \, {{\lambda}}^{3} {{\tau_{2345}}} {{\tau_{34}}} - {{\lambda}}^{3} {{\tau_{12345}}} {{\tau_{34}}} + 2 \, {{\lambda}}^{2} {{\tau_{12}}} {{\tau_{12345}}} {{\tau_{34}}} + {{\lambda}}^{4} {{\tau_{15}}} + {{\lambda}}^{3} {{\tau_{234}}} {{\tau_{15}}} + 3 \, {{\lambda}}^{3} {{\tau_{1235}}} {{\tau_{15}}} + 3 \, {{\lambda}}^{3} {{\tau_{1245}}} {{\tau_{15}}} + 3 \, {{\lambda}}^{3} {{\tau_{1345}}} {{\tau_{15}}} - {{\lambda}}^{3} {{\tau_{12345}}} {{\tau_{15}}} + {{\lambda}}^{3} {{\tau_{23}}} {{\tau_{15}}} - 2 \, {{\lambda}}^{2} {{\tau_{234}}} {{\tau_{23}}} {{\tau_{15}}} - 2 \, {{\lambda}}^{2} {{\tau_{12345}}} {{\tau_{23}}} {{\tau_{15}}} - 3 \, {{\lambda}}^{3} {{\tau_{24}}} {{\tau_{15}}} + {{\lambda}}^{3} {{\tau_{34}}} {{\tau_{15}}} - 5 \, {{\lambda}}^{2} {{\tau_{12345}}} {{\tau_{15}}}^{2} + {{\lambda}}^{4} {{\tau_{25}}} + {{\lambda}}^{3} {{\tau_{134}}} {{\tau_{25}}} + {{\lambda}}^{3} {{\tau_{13}}} {{\tau_{25}}} + 2 \, {{\lambda}}^{2} {{\tau_{134}}} {{\tau_{13}}} {{\tau_{25}}} + 3 \, {{\lambda}}^{3} {{\tau_{2345}}} {{\tau_{25}}} - {{\lambda}}^{3} {{\tau_{12345}}} {{\tau_{25}}} + 4 \, {{\lambda}}^{2} {{\tau_{13}}} {{\tau_{12345}}} {{\tau_{25}}} + {{\lambda}}^{3} {{\tau_{14}}} {{\tau_{25}}} - 2 \, {{\lambda}}^{2} {{\tau_{12345}}} {{\tau_{14}}} {{\tau_{25}}} + {{\lambda}}^{3} {{\tau_{34}}} {{\tau_{25}}} + {{\lambda}}^{4} {{\tau_{35}}} + {{\lambda}}^{3} {{\tau_{12}}} {{\tau_{35}}} + {{\lambda}}^{3} {{\tau_{124}}} {{\tau_{35}}} - 2 \, {{\lambda}}^{2} {{\tau_{12}}} {{\tau_{124}}} {{\tau_{35}}} - {{\lambda}}^{3} {{\tau_{12345}}} {{\tau_{35}}} - 2 \, {{\lambda}}^{2} {{\tau_{12}}} {{\tau_{12345}}} {{\tau_{35}}} + {{\lambda}}^{3} {{\tau_{14}}} {{\tau_{35}}} - 3 \, {{\lambda}}^{3} {{\tau_{24}}} {{\tau_{35}}} $}
\\\\
 \parbox[b]{0.85\textwidth}{\raggedright\hangafter=1\hangindent=2em$\displaystyle
 [\cEll{\{{1}\}\{{2}\}\{{3}\}\{{4}\}\{{5}\}}]= {{\lambda}}^{6} + {{\lambda}}^{5} {{\tau_{12}}} + {{\lambda}}^{5} {{\tau_{45}}} + {{\lambda}}^{4} {{\tau_{12}}} {{\tau_{45}}} + {{\lambda}}^{5} {{\tau_{123}}} + 2 \, {{\lambda}}^{4} {{\tau_{12}}} {{\tau_{123}}} + {{\lambda}}^{4} {{\tau_{45}}} {{\tau_{123}}} + 2 \, {{\lambda}}^{3} {{\tau_{12}}} {{\tau_{45}}} {{\tau_{123}}} + {{\lambda}}^{5} {{\tau_{124}}} + 2 \, {{\lambda}}^{4} {{\tau_{12}}} {{\tau_{124}}} + {{\lambda}}^{5} {{\tau_{125}}} + 2 \, {{\lambda}}^{4} {{\tau_{12}}} {{\tau_{125}}} + {{\lambda}}^{5} {{\tau_{134}}} + {{\lambda}}^{5} {{\tau_{135}}} + {{\lambda}}^{5} {{\tau_{234}}} + {{\lambda}}^{5} {{\tau_{235}}} + {{\lambda}}^{5} {{\tau_{145}}} + {{\lambda}}^{5} {{\tau_{245}}} + {{\lambda}}^{5} {{\tau_{13}}} + {{\lambda}}^{4} {{\tau_{45}}} {{\tau_{13}}} + 2 \, {{\lambda}}^{4} {{\tau_{134}}} {{\tau_{13}}} + 2 \, {{\lambda}}^{4} {{\tau_{135}}} {{\tau_{13}}} + {{\lambda}}^{4} {{\tau_{245}}} {{\tau_{13}}} + {{\lambda}}^{5} {{\tau_{345}}} + {{\lambda}}^{4} {{\tau_{12}}} {{\tau_{345}}} + {{\lambda}}^{5} {{\tau_{1234}}} - {{\lambda}}^{4} {{\tau_{12}}} {{\tau_{1234}}} - 6 \, {{\lambda}}^{3} {{\tau_{12}}}^{2} {{\tau_{1234}}} + 2 \, {{\lambda}}^{4} {{\tau_{234}}} {{\tau_{1234}}} - {{\lambda}}^{4} {{\tau_{13}}} {{\tau_{1234}}} + {{\lambda}}^{5} {{\tau_{1235}}} - {{\lambda}}^{4} {{\tau_{12}}} {{\tau_{1235}}} - 6 \, {{\lambda}}^{3} {{\tau_{12}}}^{2} {{\tau_{1235}}} + 2 \, {{\lambda}}^{4} {{\tau_{235}}} {{\tau_{1235}}} - {{\lambda}}^{4} {{\tau_{13}}} {{\tau_{1235}}} + {{\lambda}}^{5} {{\tau_{1245}}} - {{\lambda}}^{4} {{\tau_{12}}} {{\tau_{1245}}} - 6 \, {{\lambda}}^{3} {{\tau_{12}}}^{2} {{\tau_{1245}}} + 2 \, {{\lambda}}^{4} {{\tau_{245}}} {{\tau_{1245}}} + {{\lambda}}^{5} {{\tau_{1345}}} - {{\lambda}}^{4} {{\tau_{13}}} {{\tau_{1345}}} - 6 \, {{\lambda}}^{3} {{\tau_{13}}}^{2} {{\tau_{1345}}} + 2 \, {{\lambda}}^{4} {{\tau_{345}}} {{\tau_{1345}}} + {{\lambda}}^{5} {{\tau_{2345}}} + 2 \, {{\lambda}}^{4} {{\tau_{345}}} {{\tau_{2345}}} + {{\lambda}}^{5} {{\tau_{12345}}} + 3 \, {{\lambda}}^{4} {{\tau_{12}}} {{\tau_{12345}}} + 6 \, {{\lambda}}^{3} {{\tau_{12}}}^{2} {{\tau_{12345}}} + 24 \, {{\lambda}}^{2} {{\tau_{12}}}^{3} {{\tau_{12345}}} + 4 \, {{\lambda}}^{4} {{\tau_{45}}} {{\tau_{12345}}} + 2 \, {{\lambda}}^{3} {{\tau_{12}}} {{\tau_{45}}} {{\tau_{12345}}} + 9 \, {{\lambda}}^{4} {{\tau_{123}}} {{\tau_{12345}}} - {{\lambda}}^{4} {{\tau_{124}}} {{\tau_{12345}}} - {{\lambda}}^{4} {{\tau_{125}}} {{\tau_{12345}}} + 4 \, {{\lambda}}^{4} {{\tau_{13}}} {{\tau_{12345}}} + 4 \, {{\lambda}}^{3} {{\tau_{13}}}^{2} {{\tau_{12345}}} + 3 \, {{\lambda}}^{4} {{\tau_{1234}}} {{\tau_{12345}}} + 3 \, {{\lambda}}^{4} {{\tau_{1235}}} {{\tau_{12345}}} - 2 \, {{\lambda}}^{4} {{\tau_{1245}}} {{\tau_{12345}}} + 2 \, {{\lambda}}^{3} {{\tau_{12}}} {{\tau_{1245}}} {{\tau_{12345}}} - {{\lambda}}^{4} {{\tau_{1345}}} {{\tau_{12345}}} - {{\lambda}}^{4} {{\tau_{2345}}} {{\tau_{12345}}} - 3 \, {{\lambda}}^{3} {{\tau_{2345}}}^{2} {{\tau_{12345}}} + {{\lambda}}^{5} {{\tau_{23}}} + {{\lambda}}^{4} {{\tau_{45}}} {{\tau_{23}}} + 2 \, {{\lambda}}^{4} {{\tau_{234}}} {{\tau_{23}}} + 2 \, {{\lambda}}^{4} {{\tau_{235}}} {{\tau_{23}}} + {{\lambda}}^{4} {{\tau_{145}}} {{\tau_{23}}} + 4 \, {{\lambda}}^{4} {{\tau_{1234}}} {{\tau_{23}}} + 4 \, {{\lambda}}^{4} {{\tau_{1235}}} {{\tau_{23}}} - {{\lambda}}^{4} {{\tau_{2345}}} {{\tau_{23}}} + 4 \, {{\lambda}}^{4} {{\tau_{12345}}} {{\tau_{23}}} - 6 \, {{\lambda}}^{3} {{\tau_{2345}}} {{\tau_{23}}}^{2} - 9 \, {{\lambda}}^{3} {{\tau_{12345}}} {{\tau_{23}}}^{2} + {{\lambda}}^{5} {{\tau_{14}}} + {{\lambda}}^{4} {{\tau_{235}}} {{\tau_{14}}} + 2 \, {{\lambda}}^{4} {{\tau_{145}}} {{\tau_{14}}} + 3 \, {{\lambda}}^{4} {{\tau_{1234}}} {{\tau_{14}}} - {{\lambda}}^{4} {{\tau_{1245}}} {{\tau_{14}}} - {{\lambda}}^{4} {{\tau_{1345}}} {{\tau_{14}}} - {{\lambda}}^{4} {{\tau_{12345}}} {{\tau_{14}}} + {{\lambda}}^{4} {{\tau_{23}}} {{\tau_{14}}} + 2 \, {{\lambda}}^{3} {{\tau_{235}}} {{\tau_{23}}} {{\tau_{14}}} + 2 \, {{\lambda}}^{3} {{\tau_{145}}} {{\tau_{23}}} {{\tau_{14}}} - 5 \, {{\lambda}}^{3} {{\tau_{12345}}} {{\tau_{14}}}^{2} + {{\lambda}}^{5} {{\tau_{24}}} + {{\lambda}}^{4} {{\tau_{135}}} {{\tau_{24}}} + 2 \, {{\lambda}}^{4} {{\tau_{245}}} {{\tau_{24}}} + {{\lambda}}^{4} {{\tau_{13}}} {{\tau_{24}}} + 2 \, {{\lambda}}^{3} {{\tau_{135}}} {{\tau_{13}}} {{\tau_{24}}} + 2 \, {{\lambda}}^{3} {{\tau_{245}}} {{\tau_{13}}} {{\tau_{24}}} + 4 \, {{\lambda}}^{4} {{\tau_{1245}}} {{\tau_{24}}} - {{\lambda}}^{4} {{\tau_{2345}}} {{\tau_{24}}} - {{\lambda}}^{4} {{\tau_{12345}}} {{\tau_{24}}} + 2 \, {{\lambda}}^{3} {{\tau_{13}}} {{\tau_{12345}}} {{\tau_{24}}} + {{\lambda}}^{5} {{\tau_{34}}} + {{\lambda}}^{4} {{\tau_{12}}} {{\tau_{34}}} + {{\lambda}}^{4} {{\tau_{125}}} {{\tau_{34}}} + 2 \, {{\lambda}}^{3} {{\tau_{12}}} {{\tau_{125}}} {{\tau_{34}}} + 2 \, {{\lambda}}^{4} {{\tau_{345}}} {{\tau_{34}}} + 2 \, {{\lambda}}^{3} {{\tau_{12}}} {{\tau_{345}}} {{\tau_{34}}} + 4 \, {{\lambda}}^{4} {{\tau_{1345}}} {{\tau_{34}}} + 4 \, {{\lambda}}^{4} {{\tau_{2345}}} {{\tau_{34}}} - {{\lambda}}^{4} {{\tau_{12345}}} {{\tau_{34}}} + {{\lambda}}^{5} {{\tau_{15}}} + {{\lambda}}^{4} {{\tau_{234}}} {{\tau_{15}}} + 3 \, {{\lambda}}^{4} {{\tau_{1235}}} {{\tau_{15}}} + 3 \, {{\lambda}}^{4} {{\tau_{1245}}} {{\tau_{15}}} + 3 \, {{\lambda}}^{4} {{\tau_{1345}}} {{\tau_{15}}} - {{\lambda}}^{4} {{\tau_{12345}}} {{\tau_{15}}} + {{\lambda}}^{4} {{\tau_{23}}} {{\tau_{15}}} + 2 \, {{\lambda}}^{3} {{\tau_{234}}} {{\tau_{23}}} {{\tau_{15}}} + 2 \, {{\lambda}}^{3} {{\tau_{12345}}} {{\tau_{23}}} {{\tau_{15}}} + {{\lambda}}^{4} {{\tau_{24}}} {{\tau_{15}}} + {{\lambda}}^{4} {{\tau_{34}}} {{\tau_{15}}} - 5 \, {{\lambda}}^{3} {{\tau_{12345}}} {{\tau_{15}}}^{2} + {{\lambda}}^{5} {{\tau_{25}}} + {{\lambda}}^{4} {{\tau_{134}}} {{\tau_{25}}} + {{\lambda}}^{4} {{\tau_{13}}} {{\tau_{25}}} + 2 \, {{\lambda}}^{3} {{\tau_{134}}} {{\tau_{13}}} {{\tau_{25}}} + 3 \, {{\lambda}}^{4} {{\tau_{2345}}} {{\tau_{25}}} - {{\lambda}}^{4} {{\tau_{12345}}} {{\tau_{25}}} + {{\lambda}}^{4} {{\tau_{14}}} {{\tau_{25}}} + 2 \, {{\lambda}}^{3} {{\tau_{12345}}} {{\tau_{14}}} {{\tau_{25}}} + {{\lambda}}^{4} {{\tau_{34}}} {{\tau_{25}}} + {{\lambda}}^{5} {{\tau_{35}}} + {{\lambda}}^{4} {{\tau_{12}}} {{\tau_{35}}} + {{\lambda}}^{4} {{\tau_{124}}} {{\tau_{35}}} + 2 \, {{\lambda}}^{3} {{\tau_{12}}} {{\tau_{124}}} {{\tau_{35}}} - {{\lambda}}^{4} {{\tau_{12345}}} {{\tau_{35}}} + {{\lambda}}^{4} {{\tau_{14}}} {{\tau_{35}}} + {{\lambda}}^{4} {{\tau_{24}}} {{\tau_{35}}} $}

\bibliographystyle{halpha-abbrv}
\bibliography{bibliography.bib}

@article{Pand,
    author = {Pandharipande, Rahul} ,
    title = {A geometric construction of {G}etzler's elliptic relation},
    journal = {Math. Ann.},
    fjournal = {Mathematische Annalen},
    volume = {313},
    year = {1999},
    pages = {715--729}, 
    doi = {10.1007/s002080050279}
}

@article {EdidinGrahamIntersection,
    AUTHOR = {Edidin, Dan and Graham, William},
     TITLE = {Equivariant intersection theory},
   JOURNAL = {Invent. Math.},
  FJOURNAL = {Inventiones Mathematicae},
    VOLUME = {131},
      YEAR = {1998},
    NUMBER = {3},
     PAGES = {595--634},
      ISSN = {0020-9910},
   MRCLASS = {14C17 (14F99)},
  MRNUMBER = {1614555},
MRREVIEWER = {Burt Totaro},
       DOI = {10.1007/s002220050214},
       URL = {http://dx.doi.org/10.1007/s002220050214},
}

@article {BehrendLefschetz,
    AUTHOR = {Behrend, Kai A.},
     TITLE = {The {L}efschetz trace formula for algebraic stacks},
   JOURNAL = {Invent. Math.},
  FJOURNAL = {Inventiones Mathematicae},
    VOLUME = {112},
      YEAR = {1993},
    NUMBER = {1},
     PAGES = {127--149},
      ISSN = {0020-9910,1432-1297},
   MRCLASS = {14F30},
  MRNUMBER = {1207479},
       DOI = {10.1007/BF01232427},
       URL = {https://doi.org/10.1007/BF01232427},
}

@article {Coskun,
    AUTHOR = {Coskun, Izzet and Prendergast-Smith, Artie},
     TITLE = {Fano manifolds of index {$n-1$} and the cone conjecture},
   JOURNAL = {Int. Math. Res. Not. IMRN},
  FJOURNAL = {International Mathematics Research Notices. IMRN},
      YEAR = {2014},
    NUMBER = {9},
     PAGES = {2401--2439},
      ISSN = {1073-7928,1687-0247},
   MRCLASS = {14J45 (14E08 14E30)},
  MRNUMBER = {3207372},
MRREVIEWER = {Gianluca\ Occhetta},
       DOI = {10.1093/imrn/rns297},
       URL = {https://doi.org/10.1093/imrn/rns297},
}

@article {Get,
    AUTHOR = {Getzler, E.},
     TITLE = {Intersection theory on {$\overline{\mathscr M}_{1,4}$} and
              elliptic {G}romov-{W}itten invariants},
   JOURNAL = {J. Amer. Math. Soc.},
  FJOURNAL = {Journal of the American Mathematical Society},
    VOLUME = {10},
      YEAR = {1997},
    NUMBER = {4},
     PAGES = {973--998},
      ISSN = {0894-0347,1088-6834},
   MRCLASS = {14H10 (14C17 14N10)},
  MRNUMBER = {1451505},
MRREVIEWER = {Ralph\ Martin\ Kaufmann},
       DOI = {10.1090/S0894-0347-97-00246-4},
       URL = {https://doi.org/10.1090/S0894-0347-97-00246-4},
}

@misc{dilorenzo2021polarizedtwistedconicsmoduli,
      title={Polarized twisted conics and moduli of stable curves of genus two}, 
      author={Di Lorenzo, Andrea and Vistoli, Angelo},
      year={2021},
      eprint={2103.13204},
      archivePrefix={arXiv},
      primaryClass={math.AG},
      url={https://arxiv.org/abs/2103.13204}, 
}

@article {LekiliP,
    AUTHOR = {Lekili, Yank{\i} and Polishchuk, Alexander},
     TITLE = {A modular compactification of {$M_{1,n}$} from
              {$A_\infty$}-structures},
   JOURNAL = {J. Reine Angew. Math.},
  FJOURNAL = {Journal f\"{u}r die Reine und Angewandte Mathematik. [Crelle's
              Journal]},
    VOLUME = {755},
      YEAR = {2019},
     PAGES = {151--189},
      ISSN = {0075-4102},
   MRCLASS = {14F08 (14H10)},
  MRNUMBER = {4015231},
MRREVIEWER = {Abdelmoubine Amar Henni},
       DOI = {10.1515/crelle-2017-0015},
}

@article {Inchiostro,
    AUTHOR = {Inchiostro, Giovanni},
     TITLE = {Moduli of genus one curves with two marked points as a
              weighted blow-up},
   JOURNAL = {Math. Z.},
  FJOURNAL = {Mathematische Zeitschrift},
    VOLUME = {302},
      YEAR = {2022},
    NUMBER = {3},
     PAGES = {1905--1925},
      ISSN = {0025-5874},
   MRCLASS = {14H10 (14C25 14E15 14F22)},
  MRNUMBER = {4492520},
       DOI = {10.1007/s00209-022-03121-5},
}

@book {Belorousski,
    AUTHOR = {Belorousski, Pavel},
     TITLE = {Chow rings of moduli spaces of pointed elliptic curves},
      NOTE = {Thesis (Ph.D.)--The University of Chicago},
      YEAR = {1998},
     PAGES = {65},
      ISBN = {978-0591-85283-7},
   MRCLASS = {Thesis},
  MRNUMBER = {2716762},
}

@article {Schubert,
    AUTHOR = {Schubert, David},
     TITLE = {A new compactification of the moduli space of curves},
   JOURNAL = {Compositio Math.},
  FJOURNAL = {Compositio Mathematica},
    VOLUME = {78},
      YEAR = {1991},
    NUMBER = {3},
     PAGES = {297--313},
      ISSN = {0010-437X,1570-5846},
   MRCLASS = {14H10},
  MRNUMBER = {1106299},
MRREVIEWER = {R.\ F.\ Lax},
}

@article {SmythI,
    AUTHOR = {Smyth, David I.},
     TITLE = {Modular compactifications of the space of pointed elliptic
              curves {I}},
   JOURNAL = {Compos. Math.},
  FJOURNAL = {Compositio Mathematica},
    VOLUME = {147},
      YEAR = {2011},
    NUMBER = {3},
     PAGES = {877--913},
      ISSN = {0010-437X,1570-5846},
   MRCLASS = {14D23 (14H10)},
  MRNUMBER = {2801404},
MRREVIEWER = {Concettina\ Galati},
       DOI = {10.1112/S0010437X10005014},
}

@article{SmythII,
	Author = {Smyth, David I.},
	Doi = {10.1112/S0010437X11005549},
	Fjournal = {Compositio Mathematica},
	Issn = {0010-437X},
	Journal = {Compos. Math.},
	Mrclass = {14H10 (14D23 14E30 14H52)},
	Mrnumber = {2862065},
	Mrreviewer = {Hsian-Hua Tseng},
	Number = {6},
	Pages = {1843--1884},
	Title = {Modular compactifications of the space of pointed elliptic curves {II}},
	Volume = {147},
	Year = {2011},
}

@misc{stacks-project,
    shorthand    = {Stacks},
    author       = {The {Stacks Project Authors}},
    title        = {\textit{Stacks Project}},
    howpublished = {\url{https://stacks.math.columbia.edu}},
    year         = {2020},
  }

@article {DLPV,
    AUTHOR = {Di Lorenzo, Andrea and Pernice, Michele and Vistoli, Angelo},
     TITLE = {Stable cuspidal curves and the integral {C}how ring of
              {$\overline{\mathcal M}_{2,1}$}},
   JOURNAL = {Geom. Topol.},
  FJOURNAL = {Geometry \& Topology},
    VOLUME = {28},
      YEAR = {2024},
    NUMBER = {6},
     PAGES = {2915--2970},
      ISSN = {1465-3060,1364-0380},
   MRCLASS = {14C15 (14D23 14H10)},
  MRNUMBER = {4817475},
       DOI = {10.2140/gt.2024.28.2915},
}

@article {Lar,
    AUTHOR = {Larson, Eric},
     TITLE = {The integral {C}how ring of {$\overline M_2$}},
   JOURNAL = {Algebr. Geom.},
  FJOURNAL = {Algebraic Geometry},
    VOLUME = {8},
      YEAR = {2021},
    NUMBER = {3},
     PAGES = {286--318}
}

@article {Per,
    AUTHOR = {Pernice, Michele},
     TITLE = {The (almost) integral {C}how ring of {$\overline{\mathcal
              M}_{3}$}},
   JOURNAL = {J. Reine Angew. Math.},
  FJOURNAL = {Journal f\"ur die Reine und Angewandte Mathematik. [Crelle's
              Journal]},
    VOLUME = {812},
      YEAR = {2024},
     PAGES = {275--315},
      ISSN = {0075-4102,1435-5345},
   MRCLASS = {14D23 (14C15 14H10)},
  MRNUMBER = {4767394},
       DOI = {10.1515/crelle-2024-0034},
}

@conference{Mum,
				author={Mumford, David},
				title={Towards an enumerative geometry of the moduli space of curves},
				booktitle={Arithmetic and geometry, Vol. II},
				series={Progr. Math.},
				volume={36},
                    publisher={Birkh\"{a}user Boston, Boston, MA},
				year={1983}
			}

@article {Fab,
    AUTHOR = {Faber, Carel},
     TITLE = {Chow rings of moduli spaces of curves. {I}. {T}he {C}how ring
              of {$\overline{\mathscr M}_3$}},
   JOURNAL = {Ann. of Math. (2)},
  FJOURNAL = {Annals of Mathematics. Second Series},
    VOLUME = {132},
      YEAR = {1990},
    NUMBER = {2},
     PAGES = {331--419}
}

@article {Fab2,
    AUTHOR = {Faber, Carel},
     TITLE = {Chow rings of moduli spaces of curves. {II}. {S}ome results on
              the {C}how ring of {$\overline{\mathscr M}_4$}},
   JOURNAL = {Ann. of Math. (2)},
  FJOURNAL = {Annals of Mathematics. Second Series},
    VOLUME = {132},
      YEAR = {1990},
    NUMBER = {3},
     PAGES = {421--449},
}

@incollection {Iza,
    AUTHOR = {Izadi, Elham},
     TITLE = {The {C}how ring of the moduli space of curves of genus {$5$}},
 BOOKTITLE = {The moduli space of curves ({T}exel {I}sland, 1994)},
    SERIES = {Progr. Math.},
    VOLUME = {129},
     PAGES = {267--304},
 PUBLISHER = {Birkh\"{a}user Boston, Boston, MA},
      YEAR = {1995}
}

@article {PenevVakil,
    AUTHOR = {Penev, Nikola and Vakil, Ravi},
     TITLE = {The {C}how ring of the moduli space of curves of genus six},
   JOURNAL = {Algebr. Geom.},
  FJOURNAL = {Algebraic Geometry},
    VOLUME = {2},
      YEAR = {2015},
    NUMBER = {1},
     PAGES = {123--136},
      ISSN = {2313-1691,2214-2584},
   MRCLASS = {14H10 (14C15 14H51)},
  MRNUMBER = {3322200},
MRREVIEWER = {Dmitry\ Kerner},
       DOI = {10.14231/AG-2015-006},
}

@article{CL789,
   title={The {C}how rings of the moduli spaces of curves of genus 7, 8, and 9},
   ISSN={1534-7486},
   url={http://dx.doi.org/10.1090/jag/818},
   DOI={10.1090/jag/818},
   journal={Journal of Algebraic Geometry},
   publisher={American Mathematical Society (AMS)},
   author={Canning, Samir and Larson, Hannah},
   year={2023},
   month=may }

@article {BaeSchmittII,
    AUTHOR = {Bae, Younghan and Schmitt, Johannes},
     TITLE = {Chow rings of stacks of prestable curves {II}},
   JOURNAL = {J. Reine Angew. Math.},
  FJOURNAL = {Journal f\"{u}r die Reine und Angewandte Mathematik. [Crelle's
              Journal]},
    VOLUME = {800},
      YEAR = {2023},
     PAGES = {55--106},
      ISSN = {0075-4102,1435-5345},
   MRCLASS = {14H10 (14C25 14D23)},
  MRNUMBER = {4609823},
       DOI = {10.1515/crelle-2023-0018},
}

@article{Bishop,
      title={The integral Chow ring of {$\oM_{1,3}$}}, 
      author={Martin Bishop},
        journal = {arXiv e-prints},
      year={2024},
      eprint={2402.05020},
      archivePrefix={arXiv},
      primaryClass={math.AG}
}

@article {BDL1,
    AUTHOR = {Battistella, Luca and Di Lorenzo, Andrea},
     TITLE = {Wall-crossing integral chow rings of {$\overline{\mathcal
              M}_{1,n\leq4}$}},
   JOURNAL = {Forum Math. Sigma},
  FJOURNAL = {Forum of Mathematics. Sigma},
    VOLUME = {13},
      YEAR = {2025},
     PAGES = {Paper No. e196, 28},
      ISSN = {2050-5094},
   MRCLASS = {14H10 (14C15 14D23)},
  MRNUMBER = {4997998},
       DOI = {10.1017/fms.2025.10143},
       URL = {https://doi.org/10.1017/fms.2025.10143},
}

@article{CL,
       author = {{Canning}, Samir and {Larson}, Hannah},
        title = {{On the Chow and cohomology rings of moduli spaces of stable curves}},
      journal = {J. Eur. Math. Soc.},
         year = {2024},
        doi = {10.4171/jems/1543}
}

@article {Petersen,
    AUTHOR = {Petersen, Dan},
     TITLE = {The structure of the tautological ring in genus one},
   JOURNAL = {Duke Math. J.},
  FJOURNAL = {Duke Mathematical Journal},
    VOLUME = {163},
      YEAR = {2014},
    NUMBER = {4},
     PAGES = {777--793},
      ISSN = {0012-7094,1547-7398},
   MRCLASS = {14H10 (14F43 14N35)},
  MRNUMBER = {3178432},
MRREVIEWER = {Montserrat\ Teixidor i Bigas},
       DOI = {10.1215/00127094-2429916},
}

@article {Keel,
    AUTHOR = {Keel, Sean},
     TITLE = {Intersection theory of moduli space of stable {$n$}-pointed
              curves of genus zero},
   JOURNAL = {Trans. Amer. Math. Soc.},
  FJOURNAL = {Transactions of the American Mathematical Society},
    VOLUME = {330},
      YEAR = {1992},
    NUMBER = {2},
     PAGES = {545--574},
      ISSN = {0002-9947,1088-6850},
   MRCLASS = {14C15 (14C17 14H10)},
  MRNUMBER = {1034665},
MRREVIEWER = {Steven\ E.\ Landsburg},
       DOI = {10.2307/2153922},
}

@article {SmythTowards,
    AUTHOR = {Smyth, David Ishii},
     TITLE = {Towards a classification of modular compactifications of
              {$\rm{M}_{g,n}$}},
   JOURNAL = {Invent. Math.},
  FJOURNAL = {Inventiones Mathematicae},
    VOLUME = {192},
      YEAR = {2013},
    NUMBER = {2},
     PAGES = {459--503},
      ISSN = {0020-9910,1432-1297},
   MRCLASS = {14H10 (14H20 14M27)},
  MRNUMBER = {3044128},
MRREVIEWER = {Scott\ R.\ Nollet},
       DOI = {10.1007/s00222-012-0416-1},
}

@book {3264,
    AUTHOR = {Eisenbud, David and Harris, Joe},
     TITLE = {3264 and all that---a second course in algebraic geometry},
 PUBLISHER = {Cambridge University Press, Cambridge},
      YEAR = {2016},
     PAGES = {xiv+616},
      ISBN = {978-1-107-60272-4; 978-1-107-01708-5},
   MRCLASS = {14-01 (14C15 14M15 14N10)},
  MRNUMBER = {3617981},
MRREVIEWER = {Arnaud\ Beauville},
       DOI = {10.1017/CBO9781139062046},
}

@misc {Kock,
    AUTHOR = {Kock, Joachim},
     TITLE = {Notes on $\psi$ classes},
 howpublished = {\url{https://mat.uab.cat/~kock/GW/notes/psi-notes.pdf}},
    year         = {2001},
}

@article {BKN,
    AUTHOR = {Bozlee, Sebastian and Kuo, Bob and Neff, Adrian},
     TITLE = {A classification of modular compactifications of the space of
              pointed elliptic curves by {G}orenstein curves},
   JOURNAL = {Algebra Number Theory},
  FJOURNAL = {Algebra \& Number Theory},
    VOLUME = {17},
      YEAR = {2023},
    NUMBER = {1},
     PAGES = {127--163},
      ISSN = {1937-0652,1944-7833},
   MRCLASS = {14D23 (14H10)},
  MRNUMBER = {4549282},
MRREVIEWER = {Michele\ Pernice},
       DOI = {10.2140/ant.2023.17.127},
}

\subsection*{Funding} During the preparation of this work, L.B. was partially supported by the European Union - NextGenerationEU under the National Recovery and Resilience Plan (PNRR) - Mission 4: Education and Research - Component 2: From research to business - Investment 1.1 Notice Prin 2022 - DD N. 104 del 2/2/2022, titled "Symplectic varieties: their interplay with Fano manifolds and derived categories", proposal code 2022PEKYBJ – CUP J53D23003840006. L.B. is a member of INdAM group GNSAGA.
\appendix
\end{document}